%% file: main.tex
\documentclass[11pt]{amsart}
 \pdfoutput=1 
\usepackage{marginnote}
\usepackage{geometry}
\usepackage{cancel}
\usepackage{amsmath,amsthm, yhmath}
\allowdisplaybreaks[1]
\usepackage{amssymb,mathrsfs}
\usepackage{tensor}

\usepackage{slashed}

\usepackage{wrapfig}
\usepackage[pdftex]{graphicx}
\usepackage{pdfpages}
\usepackage{lmodern}
\usepackage[T1]{fontenc}

\usepackage{esint}
\usepackage{epstopdf}
\usepackage{amscd}
\usepackage{enumerate}
\usepackage{mymathsty}
\usepackage{fancyhdr} 
\usepackage{microtype}
\usepackage{hyperref}
\usepackage[all]{xy}
\xyoption{poly}

 \numberwithin{equation}{section}
	\DeclareMathOperator{\supp}{supp}
	\newcommand*{\op}[1]{{\mathrm{#1}}}
	
	\newcommand{\modop}{\bigtriangleup} %- the modular operator
	\newcommand{\logmodop}{\bigtriangledown}

	\DeclareMathSymbol{\Gamma}{\mathalpha}{operators}{0}
	\DeclareMathSymbol{\Delta}{\mathalpha}{operators}{1}
	\DeclareMathSymbol{\Theta}{\mathalpha}{operators}{2}
	\DeclareMathSymbol{\Lambda}{\mathalpha}{operators}{3}
	\DeclareMathSymbol{\Xi}{\mathalpha}{operators}{4}
	\DeclareMathSymbol{\Pi}{\mathalpha}{operators}{5}
	\DeclareMathSymbol{\Sigma}{\mathalpha}{operators}{6}
	\DeclareMathSymbol{\Upsilon}{\mathalpha}{operators}{7}
	\DeclareMathSymbol{\Phi}{\mathalpha}{operators}{8}
	\DeclareMathSymbol{\Psi}{\mathalpha}{operators}{9}
	\DeclareMathSymbol{\Omega}{\mathalpha}{operators}{10}
	\renewcommand{\epsilon}{\varepsilon}
	\DeclareMathSymbol{\T}{\mathbin}{AMSb}{"54}

	\newcommand{\cA}{\mathcal{A}}

\usepackage{mathtools}
\newcommand{\defeq}{\vcentcolon=}

\newcommand{\symd}{\partial}
\begin{document}
 \author{Yang Liu}
 \address{Department of Mathematics, Ohio State University, 
Columbus, OH 43201}
\email{liu.858@osu.edu}
%\urladdr{www.math.sc.edu/$\sim$howard} % Delete if not wanted.

\title{Modular Curvature for Toric Noncommutative Manifolds
%\large{Hopf Cyclic cohomology of Noncommutative tori}
}

\date{\today}
\thanks{The author is greatly indebted to Henri Moscovici for sharing this great research project and this paper grew out of numerous conversations with him.}
\begin{abstract}
    \input{abs}

\end{abstract}

\maketitle

\tableofcontents

\input{introv1}
\input{deformation}

\input{def_rie_geo}

\input{pse_op}
\input{wpse_diff_ops}

\input{wpse_theta}

\input{mod_cur}

\input{app}

%----------------------------
\bibliographystyle{plain}
%{alpha}

\bibliography{mylib}

\end{document}

%% file: abs.tex
 In this paper, we extend recent results on the modular geometry on noncommutative two tori to a larger class of noncommutative manifolds: toric noncommutative manifolds. We first develop  a pseudo differential calculus which is suitable  for spectral geometry on toric noncommutative manifolds. As the main application, we derive a general expression for the modular curvature with respect to a conformal change of metric on  toric noncommutative manifolds. By specializing our results to the noncommutative two  and four tori,  we recovered the modular curvature functions  found in the previous works. An important  technical aspect of the computation is that it is   free of computer assistance.

%% file: introv1.tex
% !TEX root =  main.tex

\section{Introduction}
\label{sec:intro}

In the noncommutative differential geometry program (cf. for instance, Connes's book \cite{MR1303779}), the geometric data is given in the form of  a spectral  triple $(\mathcal A, \mathcal H, D)$, where $\mathcal A$ is a $*$-algebra which serves as the algebra of coordinate functions of the underlying space, and $D$ is an unbounded self-adjoint operator such that the commutators $[D,a]$ are bounded operators on $\mathcal H$ for all $a \in \mathcal A$. More than the topological structure, the spectral data also reflects  the metric and differential structure of the geometric space.  The prototypical example  comes from spin geometry:  $(C^\infty(M), L^2(\slashed S),\slashed{D})$, where $M$ is a closed spin manifold with  spinor bundle $\slashed S$ and $\slashed D$ is the Dirac operator. 
In Riemannian geometry, local geometric invariants, such as the scalar curvature function can be recovered from the asymptotic expansion of Schwartz kernel function of the heat operator $e^{-t\Delta}$: 
\begin{align*}
	K(x,x,t) \backsim \nsum{j\geq 0}{} V_{j}(x) t^{(j-d)/2},
	\,\,\, \text{where $d$ is the dimension of the manifold.}
\end{align*}
Equivalently, one turns to the asymptotic expansion of the heat kernel trace 
\begin{align}\label{eq:intro-heatkernel}
  \Tr(fe^{-tD^{2}}) \backsim_{t\searrow 0} \nsum{j\geq 0}{} V_{j}(f,D^{2}) t^{(j-d)/2}, \,\,\, f \in \cA, 
\end{align}
which makes perfect sense in the  operator theoretic setting. In the spirit of  Connes and Moscovici's work \cite{MR3194491}, the coefficients $V_j(\cdot, D^2)$ above (in \eqref{eq:intro-heatkernel}), viewed as functionals on the algebra of coordinate functions, encode the local geometry, such as intrinsic curvatures, with respect to the metric  implemented  by the operator $D$. 
%for  a metric on a noncommutative space implemented  by the operator $D$, a good notion of intrinsic curvatures  comes from the density functions of the heat kernel coefficients $V_j(\cdot, D^2)$ in \eqref{eq:intro-heatkernel}.   
% The coefficients $V_j$ give rise to special values of the zeta function $\zeta_D(f,s)$. To summarize, we quote from  \cite{MR3194491}: ``It is the high frequency behavior of the spectrum of $D$ coupled with the action of the algebra $\mathcal A$ in $\mathcal H$ which detects the local curvature of the geometry.'' This program was carried out in great depth on  noncommutative two torus $\T^2_\theta$ $(\theta \notin \Q)$, the simplest but richest noncommutative surface.
 This approach was carried out in great depth on  noncommutative two tori. The technical tool for the computation is the pseudo differential calculus associated to a $C^*$-dynamical system which was developed in  Connes' seminal paper \cite{connes1980c}, meanwhile the computation was initiated, see \cite{cohen1992conformal}. The first application of such calculus is the Gauss-Bonnet theorem for noncommutative two torus \cite{MR2907006}. The major progress occurred in  Connes
 and Moscovici's recent work \cite{MR3194491}. The high lights of the paper contains not only the full local expression for the functional of the second heat coefficient, but also several geometric applications  of the local formulas  which demonstrate the great significance of the approach.  The appearance of the modular curvature functionals in those closed formulas gives vivid reflections of the noncommutativity.   An independent calculation for the Gauss-Bonnet theorem and  the  full expressions of the modular curvature
 functions, was carried out in  \cite{MR2956317} and \cite{MR3148618}, with a different CAS (Computer Algebra System). Modular scalar curvature on noncommutative four tori was studied in \cite{MR3359018} and \cite{MR3369894}.   Recently, in \cite{Lesch:2015aa}, the computation was extended to Heisenberg modules over noncommutative two torus and the whole calculation was greatly simplified so that CAS was no longer need. See also \cite{Bhuyain2012} and \cite{2013SIGMA...9..071R} for other related work on noncommutative two tori.\par
 It is  natural to investigate  how to implement  the program for other  noncommutative manifolds. An interesting class of examples comes from deformation of classical Riemannian manifolds, such is the Connes-Landi deformations (cf. \cite{MR1846904, MR1937657}), also called toric noncommutative manifolds in  \cite{MR3262521}. 
The underlying deformation theory, called $\theta$-deformation in the literature,  was first developed in Rieffel's  work  \cite{MR1184061}.

Following the spirit of the  previous work on noncommutative tori, we use a pseudo differential calculus to tackle the heat kernel coefficients in this paper.  The construction is the first main outcome of  this paper. 
Our pseudo differential calculus is designed to  handle two families of noncommutative manifolds  simultaneously:  tori and  spheres obtained by  the  Connes-Landi deformation.  In contrast to  noncommutative two tori,   the noncommutative four spheres are different in two essential ways:  
\begin{enumerate}[(1)]
    \item The dimension of the action torus (which is two) is less than the dimension of the underlying manifold (which is four);
    \item The underlying manifold is not parallelizable. 
\end{enumerate}
The first one implies that the torus action is not transitive, hence  the correspondent $C^*$-dynamical system will not be able to reveal the entire geometry. The second fact indicates that one should expect a  more sophisticated asymptotic formula for the product of two symbols than the one appears in  Connes' construction.
The method taken in this paper is to apply the deformation theory not only to the algebra of smooth functions on the underlying Riemannian manifold, but also to the whole pseudo differential calculus. The resulting symbol calculus blends the commutative and the noncommutative coordinates in a simplest fashion.

In order for the deformation theory to apply,  both the symbol map and the quantization map in the calculus have to be equivariant with respect to the torus action. This leads us to  work with  global pseudo differential calculus on closed manifolds in which all the ingredients are given in a coordinate-free way. Such calculus, which appeared first in Widom's work  \cite{MR538027} and \cite{MR560744}, turned out to be the perfect tool to  develop the deformation process.

 In the rest of the paper, we devote the attention to applications. In contrast to  the work  \cite{MR3194491,MR3359018,Lesch:2015aa}, we skip the construction of the spectral triple since only pseudo differential operators acting on functions are consider in this paper. As a consequence, we use  scalar Laplacian operator (instead of the Dirac operator) to define the metric and the noncommutative conformal change of metric is implemented by a perturbation of  the
 scalar Laplacian operator  via a Weyl factor $k$.
  The first consequence of the pseudo differential calculus is the existence of the  asymptotic expansion  \eqref{eq:intro-heatkernel}.   The associated modular  curvature is defined to be the functional density  with respect to the $V_2$ term in the  heat kernel asymptotic \eqref{eq:intro-heatkernel}.  It is worth to point out that the modular curvature defined here is only part of   the full intrinsic scalar curvature in   \cite[definition 1.147]{connes2008noncommutative}.  \par
  
  In this paper, we only test our pseudo differential calculus on the simplest but totally nontrivial perturbed Laplacian: $k\Delta$, which is obtained from the degree zero Laplacian $k \Delta k$ in \cite{MR3194491} by a conjugation.  Here $k$ is a Weyl factor as before, and $\Delta$ is the scalar Laplacian associated to the Riemannian metric. As  an instance of  \cite[Theorem 2.2]{MR3194491}, we  prove that the zeta function at zero is independent of the conformal perturbation, namely:
  \begin{align}
  \label{eq:intro-zeta-at-0}
  	\zeta_{k\Delta}(0) = \zeta_{\Delta}(0). 
  \end{align} 
% Due to the simplified  set up, we  define the modular curvature to be the density of functional associated to the second heat coefficient $f \mapsto V_2(f,k\Delta)$ in \eqref{eq:intro-heatkernel}. 

  The main result is the local formula for the modular curvature $\mathcal R \in C^\infty(M_\Theta)$ with respect to a perturbed Laplacian $\pi^\Theta(P_k)$ (i.e., a noncommutative conformal change of metric): 
     \begin{align}
            \label{eq:instro-scalarcur-Mdim-m}
  	\mathcal R(k) &=  \brac{  k^{-m/2} \mathcal K (\modop)(\nabla^2 k) + k^{-(m+2)/2} \mathcal G(\modop_{(1)},\modop_{(2)}) \brac{\nabla k \nabla k}}g^{-1} \\
    &+   c k^{-(\frac m2 -1)}\mathcal S_{\Delta}.\nonumber
  \end{align}
 Let us explain the notations. First, $k \in C^\infty(M_\Theta)$ is a Weyl factor,  $m =\dim M$ is an even integer, $g^{-1}$ is the metric tensor on the cotangent bundle and $\nabla$ is the Levi-Civita connection so that the contraction $(\nabla^2 k)g^{-1}$ is equal to $-\Delta k$ and  $(\nabla k \nabla k)g^{-1}$, which equals the squared length of the covector $\nabla k$ in the commutative situation, generalizes the Dirichlet quadratic form appeared in \cite[Eq. (0.1)]{MR3194491}. The  scalar curvature function $\mathcal
 S_{\Delta}$ associated to the metric $g$ appears naturally if the metric is nonflat, the coefficient $c$ is a constant depends only on the dimension of the manifold $M$. The triangle $\modop$ (compare to $\Delta$, the Laplacian operator) is  the modular operator (see. \eqref{eq:mod-cur-modop}), while for $j=1,2$, $\modop_{(j)}$ indicates that the operator $\modop$ applied only to the $j$-th factor.  The modular curvature functions $\mathcal K$ and $\mathcal G$ are computed explicitly in the last
 section. A crucial property of the modular curvature functions is that  they can be written as linear combinations of simple divided differences\footnote{``simple'' means that at most the third divided difference occurs. For the notion of divided differences, we refer to \cite[Appendix A]{2014arXiv1405.0863L} and a classical reference \cite{MR0043339}.} of the modified logarithm $\mathcal L_0 = \log s/(s-1)$, which is the generating function of Bernoulli  numbers after the substitution $s \mapsto e^s$. The significance of this feature was explained in \cite{2014arXiv1405.0863L}.
% As pointed out in \cite{2014arXiv1405.0863L}, this fact serves as one of those ``conceptual explanations'' of the two novel aspects of the modular curvature functions $\tilde{K}_0$ and $\tilde H_0$ appeared in the expression of the Gaussian
% curvature $\grad_h F$ for noncommutative two tori \cite[Thm. 4.8]{2011arXiv1110.3500C}:
% \begin{enumerate}[(1)]
%     \item $\tilde K_0$ is (upto a factor $1/8$) the generating function of Bernoulli numbers;
%     \item $\tilde H_0$ is sum of divided differences of $\tilde K_0$.
% \end{enumerate}
% 
 \par
 
The second main outcome of this paper is obtained by specializing  the result above onto dimension two. We show that the expressions of $\mathcal K$ and $\mathcal G$ agree with the result in \cite[Theorem 3.2]{Lesch:2015aa} which gives further validation for our pseudo differential calculus and the computation performed in the last section as in  \cite{Lesch:2015aa} and as a significant improvement of the previous work, the computation does not require aid from CAS. 

In dimension four, we show that the modular curvature functions are both zero with respect to the operator  $k\Delta$. Since $k\Delta$ is  the leading part of the Laplacian adapted  in \cite{MR3359018},  the non-zero contributions to the modular functions  come from the symbols of degree one and zero. This fact can be observed in \cite{MR3369894} in which the computation was simplified. \par

  The other significant  feature of our approach is that the computation is no longer require computer assistance. The efficiency of our computation relies on a tensor calculus over the  toric noncommutative manifolds which is obtained from a   deformation of tensor calculus over the toric manifolds. On a smooth manifold $M$, a tensor calculus consists of three parts: the pointwise tensor product and contraction between tensor fields,  and a connection $\nabla$ which is characterized by the
 Leibniz property. For instance, a differential operator on $C^\infty(M)$ can be represented by a finite sum $f \mapsto \sum_{\alpha } \rho_{\alpha_j} \cdot \nabla^{j} f$, where $\rho_{\alpha_j} $ is a  contravariant tensor field $\rho$ of rank $j$ so that the contraction $\rho_{\alpha_j} \cdot \nabla^{j} f$  produces a smooth function. One of the merits is this observation is that it has a  straightforward generalization to our noncommutative setting: the tensor product and contractions between tensor fields are pointwise, like functions on a manifold, therefore the deformation procedure for functions extends naturally to tensor fields. To obtain a calculus, we show that the  Leibniz property of the Levi-Civita connection still holds in the deformed setting. As an example, we see that the Dirichlet quadratic form appeared in \cite[Eq. (0.1)]{MR3194491} with respect to the
complex structure associated to the modular parameter $\sqrt{-1}$ has the following counterpart in terms of the deformed tensor calculus:
\begin{align}
\Box_{\mathfrak R}(h) = (\nabla h \otimes_\Theta \nabla h)\cdot_\Theta g^{-1},
\label{eq:intro-dir-quad-from}
\end{align}
where  $g^{-1}$ is the metric tensor on the cotangent bundle and $\otimes_\Theta$ and $\cdot_\Theta$ are deformed tensor product and contraction respectively. Being technical tools,  such deformed  tensor calculus and pseudo differential calculus  have  many other potential applications, for instance:
\begin{enumerate}[i)]
 \item the gauge theory on toric noncommutative manifolds studied in  \cite{MR3262521};
 \item exploring generalizations of Riemannian metrics on noncommutative manifolds (cf. \cite{2013SIGMA...9..071R}).   
\end{enumerate}

We end this introduction with a brief outline of the paper. Section \ref{sec:deformation-along-tn} consists of functional analytic backgrounds of  the deformation theory. We split the discussion into two parts: deformation of algebras and deformation of operators according to their roles as ``symbols'' and ``operators'' in the general framework of pseudo differential calculi. \par

 In section \ref{sec:def-of-rie-geo}, we explain that how apply the deformation process to the whole tensor calculus, which serves as preparation for section \ref{sec:pse-ops-non-mflds} and \ref{sec:wid's-cal}, which consist of the construction of pseudo differential calculus for toric noncommutative manifolds..
% Section  \ref{sec:pse-ops-non-mflds} to \ref{sec:def-wid'scal} consist of the construction of pseudo differential calculus for toric noncommutative manifolds. We first (in section \ref{sec:pse-ops-non-mflds}) deform the algebra of pseudo differential operators and then briefly recall Widom's symbol calculus in the next section. Finally, the symbol calculus for toric noncommutative manifolds is achieved by replacing  the usual tensor calculus in Widom's work by the deformed version  constructed in previous sections. \par
% That is the main result of section  \ref{sec:def-wid'scal}. \par

The remaining two sections are devoted to applications. We first sketch the proof of the existence of the heat kernel asymptotic following \cite{gilkey1995invariance} and \cite{branson1986conformal} in section \ref{sec:heat-ker-asym}. Finally, section \ref{sec:modcur} consists of explicit computation of the local formula of the associated modular curvature. Some technical parts of the computation are moved to the appendixes.
% Again, the efficiency of our algorithm comes from the tensor calculus that reduces tons of cancellations  which were carried out by  CAS in the previous work. 
% Following the work on noncommutative tori, a noncommutative metric is implemented by a perturbation of the scalar Laplacian $\Delta$ by a Weyl factor $k$, the resulting operator is denoted by $\pi^\Theta(P_k)$. As in the commutative case, most of the results on pseudo differential operators  we developed before  hold in the parametric situation which allow us to proceed the resolvent approximation. 
%  Arguing like in \cite{gilkey1995invariance} and \cite{branson1986conformal}, we establish the asymptotic expansion of  the heat kernels and meromorphic extension of the associated zeta functions. As a quick consequence, we prove the stability of the value of the zeta function $\zeta_{\pi^\Theta(P_k)}(s)$ at zero with respect to a conformal perturbation. Those are the main results in section  \ref{sec:heat-ker-asym}.  Section \ref{sec:modcur} consists of explicit computation of the local formula of the associated modular curvature (cf.  \eqref{eq:instro-scalarcur-Mdim-m}). Again, the efficiency of our algorithm comes from the tensor calculus that reduces tons of cancellations  which were carried out by  CAS in the previous work. 
% 
%  

%%% Local Variables: 
%%% mode: latex
%%% TeX-master: "main"
%%% End: 

%% file: deformation.tex
% !TEX root = main.tex  

\section{Deformation along $\T^{n}$}
\label{sec:deformation-along-tn}

\subsection{Deformation of Fr\'echet algebras} \label{subsec:deform-frech-algebr}
In this section, we will provide the functional analytic framework which is necessary for our later discussion on toric noncommutative manifold. We refer to Rieffel's monograph \cite{MR1184061} for further details, also   \cite{MR2230348}, \cite{MR3262521} and \cite{2010LMaPh..94..263Y}. All the topological vectors spaces appeared  in this paper are over the field of complex numbers.
 
\begin{defn}\label{defn:deformation-along-tn}
Let $V$ be Fr\'echet space whose topology is defined by an increasing
family of semi-norms $\norm{\cdot}_{k}$. We say $V$ is a  smooth
$\T^{n} = \R^{n}/\Z^{n}$ module if $V$ admits a $n$-torus action $\alpha_{t}: V \rightarrow V$ such that the function $t \mapsto \alpha_{t}(v)$  belongs to $C^{\infty}(\T^{n},V)$ for all $v
\in V$, moreover, we require the action is strongly continuous in the following sense: $\forall v \in V$,  given a multi-index $\mu$, we can find another integer $j'$ such that 
\begin{align}
  \label{eq:smooth-topology-defn}
   \norm{\partial^\mu_t \alpha_t(v)}_j \le C_{\mu,j,j'} \norm{v}_{j'}, \,\,\, \forall t\in\T^n,
\end{align}
where the constant $C_{j'}$ depends on $j'$ and the vector $v$. 
\end{defn}

%Recall that the $C^{\infty}$-topology on
%$C^{\infty}(\R^{n},V)$ is given the semi-norms: $\forall f \in C^{\infty}(\R^{n},V)$ 
%\begin{align}
%  \label{eq:smooth-topology-seminorms}
%  \norm{f}_{j,k} \defeq \sup_{i\leq j} \nsum{\abs x\leq k}{} \frac{1}{\mu!} \sup_{x\in \R^{n}} \norm{\partial^{(\mu)}f(x)}_{i},
%\end{align}
% where $\mu = (\mu_{1},\cdots, \mu_{n})$ is a multi-index. 

%Then for any $(j,k)\in
%  \N\times\N$, there exists an integer $l> 0$ such that 
%  \begin{align}
%    \label{eq:smooth-action-condition}
%    \norm{\alpha_{t}(v)}_{j,k} \leq C_{j,k} \norm{v}_{l}, \,\,\,
%    \forall v \in V. 
%  \end{align}

Due to the duality between $\Z^{n}$ and $\T^{n}$, Fourier theory tells us that all smooth $\T^{n}$-module  in definition \ref{defn:deformation-along-tn} are  $\Z^n$-graded: 
\begin{align*}
  V = \overline{ \dsum{r\in\Z^{n}}{} V_{r}},
\end{align*} 
where $V_{r}$ is the image of the projection $p_{r}: V \rightarrow V$:
\begin{align}\label{eq:peter-weyl-decom}
  p_{r}(v) = \int_{\T^{n}} \alpha_{t}(v)e^{-2\pi i r \cdot t} dt,
  \;\;\; v\in V. 
\end{align}
Namely, any vector in $V$ admits a isotypical  decomposition:
\begin{align}
\label{eq:def-along-isotropic-decom}
v = \sum_{r\in\Z^n} v_r, \,\, \text{with $v_r= p_r(v)$ as above}.
\end{align}
The sequence $\set{v_r}_{r\in\Z^n}$ is of rapidly decay in $r$ due to an integration by parts argument on \eqref{eq:peter-weyl-decom}. The precise estimate is given below:
\begin{prop}\label{prop:deform-frech-algebr-smoothness}
Let $V$ be a smooth $\T^{n}$-module as in definition
\ref{defn:deformation-along-tn}, whose topology is given by an
countable increasing family of semi-norms $\norm{\cdot}_{j}$ with
$j\in\N$. Then for any element $v = \sum_{r\in\Z^n} v_r \in V$ with its isotypical decomposition, then the sequence of the $j$-th semi-norms:  $\norm{p_{r}(v)}_{j}$
is of rapidly decay in $r \in \Z^{n}$. More precisely, for any integer $k,j>0$, there exist a degree $k$ polynomial
$Q_{k}(x_{1},\cdots,x_{n})$  and another large integer $j'$ such that $\forall v \in V$, 
\begin{align}
  \label{eq:peter-weyl-decom-estimate}
  \norm{p_{r}(v)}_{j} \leq \frac{C_{k,j'}}{\abs{Q_{k}(r)}}
 \norm{\alpha_{t}(v)}_{j'} , 
\end{align}
%where $\norm{\cdot}_{j,k}$ is defined in \eqref{eq:smooth-topology-seminorms}.
In particular,  the isotypical decomposition  $\sum_{r\in\Z^{n}}{} v_{r}$ 
converges absolutely to $v$. 
%with respect to all semi-norms $\norm{\cdot}_{j}$.
\end{prop}
The proof can be found in, for instance,  \cite[Lemma 1.1]{MR1184061}. 

%\begin{proof}
%Let $\Delta$ be the Laplacian operator on $\T^{n}$. For $r \in
%\R^{n}$, let $Q(r)$ be the degree two  polynomial such that $\Delta( e^{2 \pi i
%  r\cdot t}) = Q(r) e^{2 \pi i r\cdot t}$. For any integer $k>0$, apply integration by parts
%$k$ times  on the integral in \eqref{eq:peter-weyl-decom}, we get:
%\begin{align*}
%  p_{r}(v) = \int_{\T^{n}} e^{-2\pi i t \cdot r}\alpha_{t}(v)dt
%  =\int_{\T^{n}} \frac{1}{Q^{k}(r)}e^{-2\pi i t \cdot r}
%  \Delta^{k}\brac{\alpha_{t}(v)}dt.
%\end{align*}
%Due to the compactness of the torus and the strong continuity of the action \eqref{eq:smooth-topology-defn}, there exists a large index $j'\in\N$ so that:
%\begin{align*}
%\norm{\alpha_{t}(v)}_j \le C_{j'} \norm{v}_{j'}
%\end{align*}  
%Hence
%\begin{align*}
%  \norm{p_{r}(v)}_{j} \leq  \frac{C_{k,j'}}{\abs{Q^{k}(r)}}\norm{\alpha_{t}(v)}_{j'}.
%\end{align*}
%\end{proof}

Conversely, suppose $V$ admits a smooth $\Z^{n}$ grading: $V = \overline{ \dsum{r\in\Z^{n}}{} V_{r}}$, then the
$\T^{n}$ action is given by on each homogeneous component $V_{r}$
% the $\T^{n}$ module structure is defined as follows, the torus action
% on homogeneous element $v_{r} \in
% V_{r}$ is given by:
\begin{align}
  \label{eq:z-grading-torus-action}
  t\cdot v_{r} = e^{2\pi i t\cdot  r} v_{r}, \,\,\, t\in\T^{n}. 
\end{align}
A vector $v =\nsum{r\in
    \Z^{n}}{} v_{r} \in V$ is smooth respect to the torus action if and only
if for each semi-norm $\norm{\cdot}_{j}$, the sequence $\norm{v_{r}}_{j}$ decays faster than any polynomial in $r$, that is
  for each semi-norm $\norm{\cdot}_{j}$ and integer $k$, there is an integer $l$ and a
  constant $C_{j,k}$, such that 
  \begin{align}\label{eq:peter-weyl-decay}
    \norm{v_{r}}_{j} \leq C_{j,k} \frac{\norm{v}_{l}}{r^{k}}.
  \end{align} 

% \begin{prop}
%   Any $v \in V$ can be written as a infinite sum $v = \nsum{r\in
%     \Z^{n}}{} v_{r}$, where $v_{r} = p_{r}(v)$, moreover
  
% \end{prop}

% \begin{proof}
% Notice that   $\alpha_{x}(v)$ is a smooth function in $x$, apply
% integration by parts $k$-times on \eqref{eq:peter-weyl-decom}, with
% the estimate \eqref{eq:smooth-action-condition}, we can conclude~\eqref{eq:peter-weyl-decay}

% \end{proof}

\begin{defn}
\label{defn:def-along-smooth-alg-Amods}
A $\T^n$ smooth algebra  $\mathcal A$ is a smooth $\T^n$ module as in definition \ref{defn:deformation-along-tn} such that the multiplication map $\mathcal A \times \mathcal A \rightarrow \mathcal A$ is $\T^n$-equivariant and jointly continuous, that is
for every $j$, there is a $k$ and a constant $C_{j}$ such that 
\begin{align}
  \label{eq:jointly-continuity}
\norm{ab}_{j}\leq C_{k} \norm{a}_{k}\norm{b}_{k}, \,\,\, \forall a,b \in A.
\end{align}
Additionally, if $\mathcal{A}$ is a $*$-algebra, we required the $*$-operator is continuous and $\T^n$-equivariant. 
Similarly, a $\T^n$-smooth left(right) $\mathcal A$-module $V$ is a $\T^n$-smooth module while the  left(right) module structure is $\T^n$-equivariant and the jointly continuous as in \eqref{eq:jointly-continuity}.   
\end{defn}

%Now let us assume that $A$ is a smooth $\Z^{n}$-graded algebra
%whose topology is given by an family of increasing semi-norms
%$\norm{\cdot}_{j}$. Of course, we require that the multiplication in
%$A$ preserve the grading, that is for any $a_{r} \in A_{r}$, $b_{s}\in
%A_{s}$, the product $a_{r}b_{s} \in A_{r+s}$. Also the multiplication
%is assumed to be 

\begin{defn}\label{defn:defomred-alg}
Let $\mathcal A$ be a $\T^n$ smooth algebra as above. 
For a skew symmetric $n\times n$ matrix $\Theta$, we denote the corresponding bi-character: 
\begin{align}
\label{eq:def-along-bi-cha-theta}
\chi_\Theta(r,l) = e^{\pi i\abrac{r, \Theta l}},\,\,\, r,l\in\Z^n,
\end{align}
where the pairing $\abrac{\cdot, \cdot}$ is the usual dot product in $\R^n$. 
The deformation of $\mathcal A$ is a family of algebras $\mathcal A_{\Theta}$ parametrized  by $\Theta$, whose underlying topological vector
space is equal to $\mathcal A$ while the multiplication  $\times_{\Theta}$ is deformed as follow:
\begin{align}
  \label{eq:deformed-product}
  a \times_{\Theta} b = \nsum{r,s\in\Z^{n}}{}\chi_\Theta(r,l) a_{r}b_{s}, \,\,\, \forall a,b \in \mathcal A,
\end{align}
and $a = \sum_r a_{r}$, $b =\sum_s b_{s}$ are the isotypical decomposition as in \eqref{eq:peter-weyl-decom}.  
  \end{defn}

Each $\mathcal{A}_\Theta$ inherits the smooth $\T^n$ module structure from $\mathcal{A}$ and since $\chi_\Theta(r,l)$ in \eqref{eq:deformed-product} are complex numbers of length $1$, the new multiplication is jointly continuous as well. Hence, for any $n \times n$ skew symmetric matrix $\Theta$, the deformation $\mathcal{A}_\Theta$ are all smooth $\T^n$ algebra as in definition \ref{defn:def-along-smooth-alg-Amods}. 
%\begin{prop}
%  The new product is still jointly continuous: for each $j$, there is
%  a $k$ and a constant $C_{j}$ such that 
%  \begin{align}
%    \label{eq:jointly-continuity-deformed-product}
% \norm{ a \times_{\Theta} b}_{j} \leq C_{j} \norm{a}_{k} \norm{b}_{k}   
%  \end{align}
%\end{prop}
%\begin{prop}
%The action $\alpha$ on $A_{\Theta}$ is  a differentiable action and by
%algebra automorphisms with respect to the $\times_{\Theta}$ multiplication.  
%\end{prop}

\begin{prop}\label{prop:deformation-along-tn-associativity-theta-multi}
  The deformed product $\times_{\Theta}$ on $\mathcal A_{\Theta}$ is
  associative. That is, for any $a,b,c\in \mathcal{A}$,
  \begin{align}
    \label{eq:associativity-theta-multi}
    \brac{a \times_{\Theta} b} \times_{\Theta} c = a\times_{\Theta}\brac{b \times_{\Theta} c}.
  \end{align}
  If the algebra $\mathcal A$ is a $*$-algebra, the deformation $\mathcal A_\Theta$ are $*$-algebras as well with respect to the original $*$-operator, that is $\forall a,b \in \mathcal{A}$,
  \begin{align}
  \label{eq:def-along-*-prop-timestheta}
  \brac{a \times_{\Theta} b}^* = b^* \times_\Theta a^*.
  \end{align}
\end{prop}
\begin{proof}
Let $a,b,c\in \mathcal A_{\Theta}$ with their isotypical decomposition: $a =
\nsum{r}{}a_{r}$, $b = \nsum{s}{}b_{s}$ and $c = \nsum{l}{}c_{l}$, where $r$, $s$, $l$ are summed over $\Z^n$. We compute the left hand side of \eqref{eq:associativity-theta-multi},
\begin{align*}
 \brac{a \times_{\Theta} b} \times_{\Theta} c & = \nsum{k,l}{}\chi_\Theta(k,l)  
 \brac{\nsum{r+s = k}{} \chi_\Theta(r,s) a_{r}b_{s}} c_{l} \\
 &= \sum_{r,s,l} \chi_\Theta(r+s,l) \chi_\Theta(r,s) a_{r}b_{s} c_{l}\\
&= \nsum{r,s,l}{} \chi_\Theta(r,l)  \chi_\Theta(s,l) \chi_\Theta(r,s) a_{r}b_{s}c_{l},
\end{align*}
here we have used the estimate \eqref{eq:peter-weyl-decay} to exchange the order of summation. Similar computation gives us the right hand side:
\begin{align*}
  a\times_{\Theta}\brac{b \times_{\Theta} c}  = \nsum{r,s,l}{} \chi_\Theta(r,s) \chi_\Theta(r,l)\chi_\Theta(s,l) a_{r}b_{s}c_{l}. 
\end{align*}
Thus we have proved the associativity. Notice that we have not yet used the skew-symmetric property of $\Theta$. In fact, the skew-symmetric property is only necessary for the $*$-operator to survive after deformation. In particular, it implies that for the bi-character $\chi_\Theta$ defined in \eqref{eq:def-along-bi-cha-theta},  
\begin{align*}
\chi_\Theta(r,l) = \chi_\Theta(l,r)^*, \,\, \forall r,l\in\Z^n, 
\end{align*}
here the $*$ operator is the conjugation on complex numbers. Since the $*$ operator is $\T^n$-equivariant, it flips the $\Z^n$-grading of $\mathcal A$, that is, it sends the $r$ component to the $-r$ component: $(a_r)* = a^*_{-r}$, where $a  = \sum_{r\in\Z^n} a_r\in \mathcal A$. Indeed, 
\begin{align*}
(a_r)* = \brac{ \int_{\T^n}e^{-2\pi i r \cdot t} \alpha_t(a)dt }^* 
= \int_{\T^n}e^{2\pi i r \cdot t}\alpha_t(a^*) dt =  a^*_{-r}. 
\end{align*}
 Therefore:
\begin{align*}
(a \times_\Theta b)^* & = (\sum_{r,l\in\Z^n} \chi_\Theta(r,l) a_rb_l)^* 
= \sum_{r,l\in\Z^n} \chi_\Theta(r,l)^*  (b_l)^* (a_r)^*
\\
&=  \sum_{r,l\in\Z^n} \chi_\Theta(l,r) b^*_{-l}a^*_{-r}
=  \sum_{r,l\in\Z^n} \chi_\Theta(-l,-r) b^*_{-l}a^*_{-r}\\
&= \sum_{r,l\in\Z^n} \chi_\Theta(l,r) b^*_{l}a^*_{r}
= b^* \times_\Theta a^*.
\end{align*}
The proof is complete. 
\end{proof}

\begin{prop}\label{prop:functorial-property} 
  Let $\phi: \mathcal A \rightarrow \mathcal B$ be a $\T^{n}$-equivariant continuous algebra
  homomorphism, where $A$, $B$ are two $\T^n$ smooth algebras which admit 
  deformation as above. If we identify $\mathcal A$ and $\mathcal A_{\Theta}$, $\mathcal B$ and
  $\mathcal B_{\Theta}$ by the identity maps respectively, then 
  \begin{align}
    \label{eq:functorial-deformed-product}
    \phi: \mathcal A_{\Theta}\rightarrow \mathcal B_{\Theta}
  \end{align}
is still an $\T^{n}$-equivariant algebra homomorphism with respect to
the new product $\times_{\Theta}$. 
\end{prop}
\begin{proof}
For any $a,a' \in\mathcal{A}$ with the isotypical decomposition $a = \sum_r a_r$, $b = \sum_l b_l$, thanks to the equivariant property of $\phi$, we have $\phi(a_r) = \phi(a)_r$ and $\phi(b_l) = \phi(b)_l$ for any $r,l\in\Z^n$. Use the continuity of $\phi$, we compute:
\begin{align*}
\phi( a \times_\Theta a') & = \phi\brac{
\sum_{r,l \in \Z^n} \chi_\Theta(r,l) a_r a'_l 
} =\sum_{r,l \in \Z^n} \chi_\Theta(r,l) \phi(a_r)\phi(a'_l)\\
&= \sum_{r,l \in \Z^n} \chi_\Theta(r,l)\phi(a)_r\phi(a')_l \\
&= 
\phi(a)\times_\Theta \phi(a')
\end{align*}
\end{proof}

%\begin{prop}\label{prop:general-set-up-*-operator}
%Let the algebra $\mathcal A$ be a smooth $\T^{n}$-algebra as above. Suppose $\mathcal A$ has an continuous involution $*: \mathcal A \rightarrow \mathcal A, $ and $\T^{n}$ action on
%$A$ are all $*$-automorphisms. For a fixed skew symmetric matrix
%$\Theta$, the deformed algebra $A_{\Theta}$ is still a $*$-algebra
%with the same $*$ operator, that is
%\begin{align}
%  \label{eq:*-operator}
%  (a\times_{\Theta} b)^{*} = b^{*}\times_{\Theta}a^{*}.
%\end{align} 
%\end{prop}

The next proposition shows that any $\T^n$-equivariant trace on a smooth $\T^n$ algebra $\mathcal{A}$ extends naturally to a trace on all the deformations $\mathcal{A}_\Theta$.

\begin{prop}\label{prop:general-set-up-prop-trace}
Let $\tau: \mathcal A \rightarrow \C$ be a $\T^n$-equivariant trace and $\Theta$ is a $n \times n$ skey symmetric matrix as before. Then $\tau: \mathcal A_\Theta \rightarrow \C$ is a continuous linear functional for  the deformation $\mathcal{A}_\Theta$ and $\mathcal{A}$ are identical as topological vector spaces. However, $\tau$ is indeed a trace on $\mathcal{A}_\Theta$, that is $\forall a,b \in \mathcal{A}_\Theta$
\begin{align}
\label{eq:def-along-trace}
\tau( a \times_\Theta b) = \tau(ab) =  \tau(ba) = \tau(b \times_\Theta a)  .
\end{align} 
\end{prop}
\begin{proof}
From the $\T^n$-equivariant property of $\tau$, we know that for any isotypical component $a_{r}$, $(r\in\Z^{n})$, of $a \in \mathcal A$, 
\begin{align*}
 \tau(a_{r}) = \tau (\alpha_{t}(a_{r})) = \tau (e^{2\pi i t \cdot r} a_{r})
 =e^{2\pi i t \cdot r} \tau(a_{r}), \,\,\, \forall t\in\T^{n}.
\end{align*}
Therefore $\tau(a_{r}) = 0$ for all $r \neq 0$. Follows from the continuity, for any $a \in \mathcal A$, 
\begin{align*}
 \tau(a) = \tau\brac{\sum_{r\in\Z^{n}} a_{r}} = \sum_{r\in\Z^{n}}\tau\brac{ a_{r}} = \tau(a_{0}),
\end{align*}
 that is, the trace of $a$ depends only on its $\T^{n}$-invariant component. 
%Due to the duality of $\Z^n$ and $\T^n$, a linear map between two $\T^n$ modules is $\T^n$-equivariant if and only if it preserves the $\Z^n$-grades. Take $\C$ as a trivial smooth $\T^n$ module, the equivariant assumption on $\tau$ can be rephrased in terms of the $\Z^n$ grading as follows, for any $a = \sum a_r \in \mathcal{A}$ with its isotypical decomposition, 
%\begin{align*}
%\tau(a) = \tau(\sum_{r\in\Z^n} a_r) = \tau(a_0)
%\end{align*}  
%According to \eqref{eq:deformed-product},
Since $\Theta$ is skew symmetric, we get 
\begin{align*}
\tau( a \times_\Theta b) = \tau \brac{( a \times_\Theta b)_{0}} = \sum_{r\in\Z^n} \chi_\Theta(r,-r) \tau( a_r b_{-r})
=  \sum_{r\in\Z^n} \tau( a_r b_{-r}) = \tau(ab).
\end{align*}
Similar computation gives that $\tau(ba) = \tau(b \times_\Theta a)$. Therefore if $\tau$ is a trace on $\mathcal A$, then it is a trace on $\mathcal{A}_\Theta$ as well.
\end{proof}

In this paper, we have to deal with certain topological algebras which are not Fr\'echet. For instance, the 
algebra of pseudo differential operators of  integer orders, and the associated algebra of symbols, whose topology  
is  certain inductive limit of Fr\'echet topologies. More precisely, we consider a filtered algebra  $\mathcal A$ with a filtration: 
\begin{align}
  \label{eq:deformation-of-algebra-filtration}
  \cdots\subset \mathcal A_{-j}\subset\cdots \subset \mathcal A_{0}\subset \cdots \mathcal A_{j}\cdots\subset \mathcal A,
\end{align}
where  each $A_{j}$ ($j\in\Z$) is a smooth $\T^n$-module as defined before, in particular, a Fr\'echet space. As a topological vector space, the total space $\mathcal A$ is a countable strict inductive limit of $\set{\mathcal A_j}_{j\in\Z}$, the topology is just called strict inductive limit topology (cf. for instance \cite[sec. 13]{MR2296978} for more details). This topology is never metrizable unless the filtration is stabilized starting from some $\mathcal{A}_j$, therefore it is not Fr\'echet. Nevertheless, we shall not looking at the topology of the whole algebra $\mathcal A$ even when considering the continuity of the multiplication map. Instead, we focus on each $\mathcal{A}_j$, assume that the Fr\'echet topology
is defined by a  countable family of  increasing semi-norms
$\set{\norm{\cdot}_{l,j}}_{l\in\N}$. The multiplication preserves the filtration:
\begin{align}
  \label{eq:deformation-of-algebra-multiplication}
  m: \mathcal A_{j_{1}} \times \mathcal A_{j_{2}} \rightarrow \mathcal A_{j_{1} + j_{2}}
\end{align}
such that the continuity condition holds: for fixed $j_{1}$,
$j_{2}$ and a positive integer $l$, one can find a
integer $k$ and constant $C_{k,j_{1},j_{2}}$ such that  
\begin{align}
  \label{eq:deformation-of-algebra-multiplication-continuity-condition}
 \norm{m(a_{1}a_{2})}_{l,j_{1} + j_{2}} \leq C_{k,j_{1},j_{2}} \norm{a_{1}}_{k,j_{1}}
 \norm{a_{2}}_{k,j_{2}}, \,\,\,\, \forall a_{1}\in \mathcal A_{j_{1}}, \,\, a_{2}\in \mathcal A_{j_{2}}.
\end{align}
%Assume each $A_{j}$ is a smooth $\T^{n}$-module with respect to the
%semi-norms $\set{\norm{\cdot}_{l,j}}$ above, for a fixed skew
%symmetric matrix $\Theta$, 
The multiplication $m$ is deformed in a similar fashion as in \eqref{eq:deformed-product}:
\begin{align}
  \label{eq:eq:deformation-of-algebra-m-theta}
  m_{\Theta}: A_{j_{1}} \times A_{j_{2}} & \rightarrow A_{j_{1} + j_{2}} \\
           (a_{1},a_{2}) &\mapsto \nsum{r,l\in\Z^{n}}{} \chi_\Theta(r,l) m\brac{ (a_{1})_{r}, (a_{2})_{l}}.
\end{align}

Examples  are provided in the next section. 

\input{deformation_ops}

\subsection{Deformation of functions}
%\subsection{Examples: manifolds with a torus action}
\label{subsec:deformation-funtens}

Let $M$ be a smooth  manifold without boundary. For any diffeomorphism $\varphi \in \op{Diff}(M)$, we defined the pull back action on $C^\infty(M)$:
\begin{align}
    U_\varphi(f)(x) = f(\varphi^{-1}(x)),\,\,\, f\in C^\infty(M),
    \label{eq:deformation-pullback-functions}
\end{align}
which is $*$-automorphism of $C^\infty(M)$. Assume that $M$ admits a $n$-torus action: $\T^n \subset \op{Diff}(M)$, then one can quickly verify that $\T^n$ acts smoothly (cf. eq. \eqref{eq:smooth-topology-defn}) on $C^\infty(M)$ with respect to the smooth Fr\'echet topology, also the pointwise multiplication is jointly continuous (cf. eq. \eqref{eq:jointly-continuity}) , therefore we can deform the multiplication to $\times_\Theta$ following definition \ref{defn:def-along-smooth-alg-Amods} with respect to a skew symmetric matrix $\Theta$, and the new algebra 
\begin{align*}
 C^\infty(M_\Theta) \defeq ( C^\infty(M), \times_\Theta)
\end{align*}
plays the role of smooth coordinate functions on a noncommutative manifold $M_\Theta$. \par
  Later, we will assume the manifold $M$ is compact. The non-compact examples we are interested in is the cotangent bundle $T^*M$. One can easily lift the torus action to $T^*M$ by the natural extension of diffeomorphims: $\varphi \mapsto \varphi^*$, where $\varphi^*$ is the differential of $\varphi$. Thus the cotangent bundle of the noncommutative manifold $M_\Theta$ is given by the deformed algebra:
\begin{align*}
 C^\infty(T^*M_\Theta) \defeq ( C^\infty(T^*M), \times_\Theta). 
\end{align*}
\par

Another crucial example is $S\Sigma(M) \subset C^\infty(T^*M)$, the spaces of symbols of pseudo differential operators on $M$. It is a filtered algebra:
\begin{align*}
        S\Sigma = \ucup{j=-\infty}{\infty} S\Sigma^j(M),
    \end{align*}
%\begin{align*}
%\cdots \subset  S\Sigma^{j-1}(M) \subset S\Sigma^j(M) \subset \cdots,
%\end{align*}
where each $S\Sigma^j(M)$  consists of smooth functions with the estimate in  local coordinates $(x,\xi)$,
    \begin{align}
  \abs{\partial^\alpha_x \partial^\beta_\xi p(x,\xi)}\le C_{\alpha,\beta} (1+\abs\xi)^{j - \abs\beta},
        \label{eq:def_riegeo-defn-sym-estimate}
    \end{align}
    the optimized constants $C_{\alpha,\beta}$ define a family of semi-norms  that makes $S\Sigma(M)$ into a Fr\'echet space. The smoothing symbols $S\Sigma^{-\infty}$ is the intersection:
    \begin{align*}
    S\Sigma^{-\infty} = \icap{j=-\infty}{\infty} S\Sigma^j(M),
    \end{align*}
    and the quotient $\op{CL} = S\Sigma /S\Sigma^{-\infty}$ is called the space of complete symbols. \par
For any $t \in \T^n$ viewed as a diffeomorphism on $M$, let $p\in S\Sigma^j(M)$ be a symbol of order $j$, $t \mapsto U_t(p)$ is a  function valued in $C^\infty(T^*M)$. Observe that  the partial derivatives in $t$ can be written as a finite sum in local coordinates:
\begin{align*}
	\partial_t^\gamma U_t(p) = U_t\brac{
	\sum_j	\partial^{\alpha_j}_x \partial^{\beta_j}_\xi p
		},
\end{align*}
where $\gamma$, $\alpha_j$, $\beta_j$ are multi-indices. This shows not only that  $U_t(p)$ still belongs to  $S\Sigma^j(M)$ but also  the torus action is smooth (cf. definition \ref{defn:deformation-along-tn}). Therefore we can twist the pointwise multiplication on the filtered algebra $S\Sigma(M)$ as explained in subsection \ref{subsec:deform-frech-algebr},  the deformed version is denoted by
\begin{align*}
	S\Sigma(M_\Theta) = (S\Sigma (M), \times_\Theta),
\end{align*}
where $\times_\Theta$ is given in \eqref{eq:eq:deformation-of-algebra-m-theta}.\par

%% file: deformation_ops.tex
% !TEX root =  main.tex

\subsection{Deformation of operators} \label{subsec:deform-oper}
The associativity of the $\times_\Theta$ multiplication proved in proposition \ref{prop:deformation-along-tn-associativity-theta-multi} is a special instance of certain ``functoriality'' in the categorical framework explained in \cite{MR3262521}. Let us start with deformation of operators. \par

 Let $\mathcal H_{1}$ and $\mathcal H_{2}$ be two Hilbert spaces which are both strongly continuous unitary representation of $\T^n$, denoted by $t \mapsto U_{t} \in B(\mathcal{H}_1)$ and $t \mapsto \tilde U_{t} \in B(\mathcal{H}_2)$ respectively, where $t\in\T^n$.  If  no confusions arise, both representations will be  denoted by
$U_{t}$. Then $B(\mathcal H_{1}, \mathcal H_{2})$,  the space of all bounded operators from $\mathcal H_{1} $ to $\mathcal H_{2}$, becomes a $\T^n$-module via  the adjoint action: 
\begin{align}
P \in B(\mathcal H_{1}, \mathcal H_{2}) \mapsto \op{Ad}_{t}(P) \defeq
\tilde U_{t} P U_{-t}, \,\, t \in \T^n.
\label{eq:def-ops-adjaction}
\end{align}
Denote by $B(\mathcal H_{1},\mathcal H_{2})_{\infty}$, the space of all the $\T^n$ smooth vectors in $B(\mathcal H_{1},\mathcal H_{2})$. It is a Fr\'echet space on which the  topology is defined by the semi-norms $\set{\norm{\cdot}_j}_{j\in \N}$:
\begin{align}
  \label{eq:deform-oper-frechet-semi-norms}
   q_{j}(P) \defeq \nsum{\abs\beta\leq j}{} \frac{1}{\beta!}  \norm{\partial^\beta_t \op{Ad}_t(P)}_{B(\mathcal H_{1},\mathcal H_{2})}.
\end{align}
The semi-norms above are constructed in such a way that the continuity estimate \eqref{eq:smooth-topology-defn} for the torus action holds automatically. Following from proposition \ref{prop:deform-frech-algebr-smoothness}, we see that any $\T^n$-smooth operator $P$ admits an isotypical decomposition $P = \sum_{r\in\Z^n} P_r$, where the operator norms $\set{\norm{P_r}}_{r\in\Z^n}$  decays faster than polynomial in $r$, in particular the converges of the infinite sum is absolute with respect to the operator norm in $B(\mathcal{H}_1, \mathcal{H}_2)$.

Now we are ready to define the deformation map $\pi^\Theta: B(\mathcal H_{1},\mathcal H_{2})_{\infty} \rightarrow B(\mathcal H_{1},\mathcal H_{2})_{\infty}$.
\begin{defn}
    \label{defn:def-ops-pitheta}
Let $\mathcal H_{1}$ and $\mathcal H_{2}$ be two Hilbert space with strongly continuous  unitary $\T^{n}$ actions as above, that is $\forall v \in \mathcal H_j$, $(j=1,2)$, $t \mapsto t\cdot v$ is continuous in $t \in \T^n$. We denote the actions by $t \rightarrow U_{t}$ and $t \rightarrow \tilde U_{t}$ respectively. For  a fixed $n \times n$ skew symmetric matrix $\Theta$, we recall the associated bi-character $\chi_\Theta(r,l) = e^{\pi i \abrac{r,\Theta l}}$. Then the deformation map $\pi^{\Theta}:B(\mathcal H_{1},\mathcal H_{2})_{\infty} \rightarrow B(\mathcal H_{1},\mathcal H_{2})_{\infty}: P \mapsto \pi^\Theta(P)$ is defined as follows,
\begin{align}
  \label{eq:deform-oper-pi-theta-2}
  \pi^{\Theta}(P)(f) =\nsum{r,l\in\Z^{n}}{}\chi_\Theta(r,l)  P_{r}(f_{l}),\,\,\,
  P \in B(\mathcal H_{1},\mathcal H_{2})_{\infty},
\end{align}
where  $P = \sum_{r\in\Z^{n}} P_{r}$ and $f = \sum_{l\in\Z^{n}} f_{l} \in \mathcal{H}_1$ with their isotypical decomposition. We can assume that  $f$ is a $\T^{n}$-smooth vector for  the subspace of all $\T^{n}$-smooth vectors in dense in $\mathcal H_1$.  
Alternatively,  $\pi^{\Theta}(P)$ is given by
\begin{align}
  \label{eq:deform-oper-pi-theta-1}
  \pi^{\Theta}(P) =\nsum{r\in\Z^{n}}{} P_{r} U_{r \cdot \Theta/2},
\end{align}
here $r \cdot \Theta/2$ stands for the matrix multiplication between a row vector  $r$ and  $\Theta$ whose result is a point in $\T^n$. 
 \end{defn}
 \begin{rem}
 The deformed operator $\pi^{\Theta}(P)$ belongs to $B(\mathcal H_{1},\mathcal H_{2})_{\infty}$. Indeed, in \eqref{eq:deform-oper-pi-theta-1},  the each isotypical component of $P$ is perturbed by a unitary operator $U_{r \cdot \Theta/2}$, therefore the right hand side  of \eqref{eq:deform-oper-pi-theta-1} is a sum of rapidly decay sequence, which implies not only the boundedness of $\pi^{\Theta}(P)$, but also the $\T^{n}$-smoothness.
 \end{rem}

Let us give a precise estimate of the operator norm of $\pi^\Theta(P)$.
\begin{lem} \label{lem:deform-oper-smoothness-lemma}
Let $B(\mathcal H_{1},\mathcal H_{2})_{\infty}$ denote the Fr\'echet algebra of $\T^{n}$-smooth vectors in $B(\mathcal H_{1},\mathcal H_{2})$ whose topology is given by the seminorms in \eqref{eq:deform-oper-frechet-semi-norms}. Then the deformation $\pi^{\Theta}:B(\mathcal H_{1},\mathcal H_{2})_{\infty}\rightarrow B(\mathcal H_{1},\mathcal H_{2})_{\infty} $ is a continuous linear map with respect to the Fr\'echet topology with the estimate: for any multi-index $\mu$, one can find any integer $l$ large enough such that 
\begin{align}
  \label{eq:deform-oper-frechet-semi-norms-estimate}
 \norm{ \partial_{t}^{\mu}(\op{Ad}_{t}(\pi^{\Theta}(P))) }\leq C_{\mu} q_{l}(P)
\end{align}
\end{lem}

\begin{proof}
  Given $P \in B(\mathcal H_{1},\mathcal H_{2})_{\infty}$ a $\T^{n}$-smooth operator with the isotypical decomposition $P = \nsum{r\in\Z^{n}}{} P_{r}$. 
% To see $\pi^{\Theta}(P)$ is $\T^{n}$-smooth,  we would like to show that for any multi-index $\mu$,
%   \begin{align*}
%     \partial_{t}^{\mu}(\op{Ad}_{t}(\pi^{\Theta}(P)))
%   \end{align*}
% is a bounded operator. 
From the definition,  $\pi^{\Theta}(P)$ has the isotypical decomposition $\pi^{\Theta}(P) = \nsum{r\in\Z^{n}}{} P_{r} U_{r \cdot \Theta/2}$, thus
\begin{align*}
  \op{Ad}_{t}(\pi^{\Theta}(P)) =\nsum{r\in\Z^{n}}{}\op{Ad}_{t}(P_{r} U_{r \cdot \Theta/2}) = 
\nsum{r\in\Z^{n}}{} e^{2\pi i r\cdot t}P_{r} U_{r \cdot \Theta/2}.
\end{align*}
If we let $h(r)$ be the polynomial in $r$ such that $\partial_{t}^{\mu}(e^{2\pi i r\cdot t}) = h(r) e^{2\pi i r\cdot t}$, the degree of $h(r)$ is equal to $\abs \mu$, compute   
\begin{align}
\label{eq:deform-oper-higher-order-derivative}
 \partial_{t}^{\mu}(\op{Ad}_{t}(\pi^{\Theta}(P))) = \nsum{r\in\Z^{n}}{}\partial_{t}^{\mu}(e^{2\pi i r\cdot t}) P_{r} U_{r \cdot \Theta/2} 
=\nsum{r\in\Z^{n}}{} h(r) e^{2\pi i r\cdot t} P_{r} U_{r \cdot \Theta/2}. 
\end{align}
Since that $\norm{ P_{r}}$ is of rapidly decay in $r$, we can find a large integer $l$ such that 
\begin{align*}
  \norm{h(r) P_{r}} \leq \frac{C q_{l}(P)}{\abs r^{n+1}}.
\end{align*}
Therefore 
\begin{align*}
  \norm{\partial_{t}^{\mu}(\op{Ad}_{t}(\pi^{\Theta}(P)))} \leq \brac{\nsum{r\in\Z^{n}}{}\frac{C}{\abs r^{n+1}}} q_{l}(P).
\end{align*}

\end{proof}

Similar to \eqref{eq:def-along-*-prop-timestheta}, we have the compatibility between the deformation map and  
the $*$-operation (taking the adjoint) on operators.  
\begin{lem}
  \label{lem:def-ops-deform-oper-pi-theta-adjoint}
    Let $P\in B(\mathcal H_{1},\mathcal H_{2})_{\infty}$, then its adjoint $P^{*}\in
B(\mathcal H_{1},\mathcal H_{2})_{\infty}$ as well, we have
\begin{align}
  \label{eq:deform-oper-pi-theta-adjoint}
  \pi^{\Theta}(P^{*}) = \pi^{\Theta}(P)^{*} 
\end{align}
\end{lem}
\begin{proof}
  Since the torus action is unitary,  the adjoint operation is equivariant:
  \begin{align*}
  (\op{Ad}_t(P))^* = (U_t P U_{-t})^* = U_t P^* U_{-t} =\op{Ad}_t(P^*), 
  \end{align*}
  therefore for the isotypical components, $P^{*}_{r} = (P_{-r})^{*}$ for all $r \in\Z^{n}$, 
  \begin{align*}
    \pi^{\Theta}(P^{*})& = \nsum{r\in\Z^{n}}{} P^{*}_{r} U_{r\cdot \Theta/2} = \nsum{r\in\Z^{n}}{} (P_{-r})^{*} (U_{-r\cdot \Theta/2})^{*}  \\
&= \nsum{r\in\Z^{n}}{} (U_{-r\cdot \Theta/2} P_{-r})^{*} = \nsum{r\in\Z^{n}}{} (P_{-r}U_{-r\cdot \Theta/2})^{*} \\
&=\nsum{r\in\Z^{n}}{} (P_{r}U_{r\cdot \Theta/2})^{*}
\\ &= (\pi^{\Theta}(P))^{*},
  \end{align*}
here we have used the facts that $U_{-r\cdot \Theta/2}$ and $T_{r}$ commute. 
\end{proof}

The next lemma says that the deformation is somehow invertible. 
\begin{lem}\label{lem:deform-oper-theta+theta'}
  Let $\Theta$ and $\Theta'$ be two $n\times n$ skew symmetric matrices and for any $P \in B(\mathcal H_{1},\mathcal H_{2})_{\infty}$, we have
\begin{align*}
\pi^{\Theta}\circ \pi^{\Theta'}(P) = \pi^{\Theta + \Theta'}(P).
\end{align*}
In particular, we see that the deformation process is invertible, namely $\pi^{\Theta}$ and $\pi^{-\Theta}$ are inverse to each other. 
 \end{lem}
\begin{proof}
Given $P = \nsum{r\in \Z^{n}}{}P_{r}$, $\pi^{\Theta}(T) = \nsum{r\in \Z^{n}}{}P_{r} U_{r \cdot \Theta/2}$ is the isotypical decomposition of  $\pi^{\Theta}(T)$, therefore
\begin{align*}
  \pi^{\Theta'}\brac{\pi^{\Theta}(P)} =  \nsum{r\in \Z^{n}}{}P_{r} U_{r \cdot \Theta/2}U_{r \cdot \Theta'/2}
= \nsum{r\in \Z^{n}}{}P_{r} U_{r \cdot (\Theta+\Theta')/2} = \pi^{\Theta + \Theta'}(P).
\end{align*}

\end{proof}

%\begin{lem}
%  If $T \in B(\mathcal H_{1}, \mathcal H_{2})$ is a $\T^{n}$-smooth vector, so is the deformation $\pi^{\Theta}(T)$. 
%\end{lem}
If we take $\mathcal H_1$ and $\mathcal H_2$ above to be the same Hilbert space, $B(\mathcal{H}_\infty)$ becomes an $\T^n$ smooth algebra as in definition \ref{defn:def-along-smooth-alg-Amods}. Following from definition \ref{defn:defomred-alg}, we obtain a family of deformed algebras $(B(\mathcal{H}_\infty),\times_\Theta)$ parametrized by skew symmetric matrices $\Theta$. The multiplication map is obviously $\T^n$-equivariant, that is $\op{Ad}_t(P_1) \op{Ad}_t(P_2) = \op{Ad}_t(P_1 P_2)$, for all $t \in \T^n$ and for any $P_1, P_2 \in B(\mathcal{H}_\infty)$. The associativity of the $\times_\Theta$ multiplication has the following analogy.  

\begin{prop}
    \label{prop:def-op-deform-oper-pi-theta-prod}
    Keep the notations as above. The deformation map 
\begin{align*}
\pi^\Theta: (B(\mathcal{H})_\infty,\times_\Theta) \rightarrow \pi^\Theta\brac{B(\mathcal{H}_\infty)}\subset B(\mathcal H)
\end{align*}    
   is an algebra isomorphism, namely, for any $P_{1},P_{2} \in B(\mathcal H)_\infty$,
  \begin{align}
    \label{eq:deform-oper-pi-theta-prod}
    \pi^{\Theta}(P_{1}) \pi^{\Theta}(P_{2}) = \pi^{\Theta}(P_{1} \times_{\Theta} P_{2}),
  \end{align}
recall that  the deformed product $\times_{\Theta}$ is defined in \eqref{eq:deformed-product}. 
\end{prop}

\begin{proof}
The invertiblity of $\pi^\Theta$ is proved in lemma \ref{lem:deform-oper-theta+theta'}. It remains to show that it is an algebra morphism, that is for any $\T^n$-smooth vector $v \in \mathcal{H}$, we have
  \begin{align}
  \label{eq:def-ops-proof-pitheta-associavtive}
    \pi^{\Theta}(P_{1})\brac{\pi^{\Theta}(P_{2})(v)} = \pi^{\Theta}\brac{P_{1}\times_{\Theta}P_{2}}(v).
  \end{align}  
Observe that $\pi^{\Theta}(P)(v)$ can be formally written as $P
\times_{\Theta} v$ according to \eqref{eq:deform-oper-pi-theta-2}. Therefore the left hand side and the right hand side of \eqref{eq:def-ops-proof-pitheta-associavtive} becomes $P_{1} \times_{\Theta} (P_{2}\times_{\Theta} v)$ and $(P_{1} \times_{\Theta} P_{2})\times_{\Theta} v$ respectively, thus equation \eqref{eq:def-ops-proof-pitheta-associavtive} is  exactly the same as the associativity of the
$\times_{\Theta}$-multiplication proved in
\eqref{eq:associativity-theta-multi}. 
\end{proof}

%Since the torus action is unitary, the Hilbert trace is also obviously $\T^n$ equivariant, thus proposition \ref{prop:general-set-up-prop-trace} has the following counterpart.
We have seen that the isotypical decomposition of an operator $P = \sum_r P_r$ converges with respect to operator norms. The normality of the trace somehow allows itself  to pass the summation, namely $\Tr (P) = \sum_r \Tr (P_r)$ whenever $P$ is a trace-class operator.    
\begin{lem}
\label{lem:def-ops-inv-of-trace}
  Let $\mathcal H$ be a saperable Hilbert space with a strongly
  continuous unitary $\T^{n}$ action and $P = \sum_r P_r \in B(\mathcal H)_\infty$ is
  $\T^{n}$-smooth operator with its isotypical decomposition. Suppose $P$ is a trace-class operator, then so is $\pi^{\Theta}(P)$, moreover,  $\Tr P = \Tr P_{0} =\Tr \pi^{\Theta}(P)$, where $P_{0}$ is the 
  $\T^{n}$-invariant part of both $P$ and $\pi^{\Theta}(P)$.   
%  In particular, we conclude that 
%  if one of $P$ and $\pi^\Theta(P)$ is of trace-class, then 
%  \begin{align}
%  \Tr P = \Tr \pi^\Theta(P),
%  \label{eq:def-ops-sametrace}
%  \end{align}
%  because they have the same $\T^n$-invariant part according to \eqref{eq:deform-oper-pi-theta-1}.
\end{lem}

\begin{proof}
Since $\mathcal H$ is a strongly continuous unitary representation of $\T^{n}$, it admits a orthonormal decomposition 
\begin{align*}
 \mathcal  H = \overline{\dsum{l\in\Z^{n}}{} H_{l}},
\end{align*}
 in which each $H_{l}$ consists of eigenvector of the torus action:
 \begin{align*}
   H_{l} = \set{v \in \mathcal H | t\cdot v = e^{2\pi i t \cdot l} v}.
 \end{align*}
For each $H_{l}$, one can pick a orthonormal basis
$\set{\epsilon_{k,l}}_{k\in\N}$, then
$\set{\epsilon_{k,l}}_{l\in\Z^{n},k\in\N}$ is an orthonormal basis of
$\mathcal H$. Since $\nsum{r\in\Z^{n}}{}P_{r}$ convergence absolutely
in the operator norm,
\begin{align*}
   \nsum{l\in\Z^{n},k\in\N}{} \abrac{(\nsum{r\in\Z^{n}}{}P_{r})(\epsilon_{k,l}),\epsilon_{k,l}}
 = \nsum{l\in\Z^{n},k\in\N}{} \nsum{r\in\Z^{n}}{} \abrac{P_{r}(\epsilon_{k,l}),\epsilon_{k,l}},
\end{align*}
observe that for all $r\in\Z^{n}$, $P_{r}(H_{l}) \subset H_{r+l}$,
therefore $\abrac{P_{r}(\epsilon_{k,l}),\epsilon_{k,l}} = 0$ except
the case when $r = 0$. We continue the computation above:  
\begin{align*}
   \nsum{l\in\Z^{n},k\in\N}{} \nsum{r\in\Z^{n}}{} \abrac{P_{r}(\epsilon_{k,l}),\epsilon_{k,l}} = \nsum{l\in\Z^{n},k\in\N}{} \abrac{P_{0}(\epsilon_{k,l}),\epsilon_{k,l}}. 
\end{align*}
Since $P$ is traceable,  the left hand side above converges absolutely, therefore $P_{0}$ is traceable as well and has the same trace as $P$. Recall equation \eqref{eq:deform-oper-pi-theta-1}: $\pi^{\Theta}(P) = \nsum{r\in\Z^{n}}{} P_{r} U_{r \cdot \Theta/2}$. Notice that the computation above still works if $P_{r}$ is replaced by $P_{r} U_{r \cdot \Theta/2}$, therefore $\pi^{\Theta}(P)$ is of trace class and $\Tr \pi^{\Theta}(P) = \Tr (\pi^{\Theta}(P))_{0}$, where $(\pi^{\Theta}(P))_{0}$ is the $\T^{n}$-invariant part of $\pi^{\Theta}(P)$. Since $P$ and $\pi^{\Theta}(P)$ have the same invariant part, we have finished the proof.
\end{proof}

%From \eqref{eq:deform-oper-pi-theta-1}
%and \eqref{eq:deformed-product}, we can see that  when the operations $\pi^{\infty}(\cdot)$
%and $\times_{\Theta}$ are well-defined, 
%operators $P$, $\pi^{\Theta}(P)$, $P_{1}P_{2}$ and $P_{1} \times_{\Theta} P_{2}$ have the same
%$\T^{n}$-invariant part. Thus we have the following corollary. 

The following corollary is important for our later discussion. 
\begin{cor}\label{cor:deform-oper-pi-theta-trace-prop}
Let $\mathcal H$ be a separable Hilbert space. Let $P_{1}, P_{2}
\in B(\mathcal H)$ be two $\T^{n}$-smooth vectors such that at least one of them is traceable,
  then  $P_{1} P_{2}$ and $P_{1} \times_{\Theta} P_{2}$ are both traceable with
 \begin{align*}
   \Tr(P_{1}P_{2}) = \Tr(P_{1} \times_{\Theta} P_{2}).
\end{align*}
Combine the equation above with
\eqref{eq:deform-oper-pi-theta-prod}, we obtain:
\begin{align*}
  \Tr\brac{ \pi^{\Theta}(P_{1})  \pi^{\Theta}(P_{2})} = \Tr (P_{1} P_{2}).
\end{align*}
\end{cor}
\par

%% file: def_rie_geo.tex
\section{Deformation of tensor calculus}
\label{sec:def-of-rie-geo}

Let $M$ be a compact toric manifold $M$ as before, that is $\op{Diff}(M)$ contains a $n$-torus. So far, we have deformed the ``smooth structure'' of $M$,  namely we have found the counterpart of the algebra of smooth coordinate functions in the noncommutative setting in the previous section. We will extend the deformation process further by restricting the torus action. For example, if we assume that the torus action is affine: $\T^n \subset \op{Affine}(M) \subset \op{Diff}(M)$, then the
whole tensor calculus can be deformed. Since  our ultimate goal is to study curvatures, one can assume that $M$ is Riemannian and the torus acts as isometries.  We make a  formal definition here.
\begin{defn}
\label{defn:def-rie-eg-toricmflds}
A toric Riemannian manifolds $M$ is a closed (compact without boundary) Riemannian manifolds whose isometry group contains a $n$-torus. In other word, $M$ admits an $\T^n$-action as isometries. 
\end{defn}

\begin{eg}
\label{eg:def-rie-eg-torus}
Let $M = \T^n$ be the $n$-torus with the usual flat metric, while the $n$-torus acts on itself by translations. The deformation gives the well-known family called noncommutative $n$-torus.  
\end{eg}

\begin{eg}
\label{eg:def-rie-eg-four-sphere}
Consider the two torus $\T^2$ acts on the four sphere $S^4$ by embedding  $\T^2$ into $\op{SO}(5)$ as rotations in the first four coordinates:
\begin{align*}
(t_1,t_2) \in \R^2 \mapsto \brac{
\begin{matrix}
e^{-\pi i t_1}& &\\&e^{-\pi i t_2}& \\ && 1 
\end{matrix}
},
\end{align*}
where we identify $\R^5$ with $\C\oplus \C \oplus \R$. Thus any Riemannian metric which is invariant under the rotations above provides an instance of a toric manifold, for example, the  Robertson-Walker metrics with a cosmic scale factor $a(t)$:
\begin{align*}
ds^2 = dt^2 + a(t)^2 d\omega^2,
\end{align*}
where $d\omega^2$ is the round metric on $S^3$, when $a(t)= \sin t$, we recover the round metric on $S^4$.
\end{eg}

\subsection{Isospectral deformations of toric Riemannian manifolds}
%\subsection{Deformation of representations}

To finished the story of deforming smooth functions on $M$, we shall represent the deformed algebra $C^\infty(M_\Theta)$ as bounded operators, whose underlying Hilbert space is metric-related. Therefore we assume that  $M$  is a toric Riemannian manifold. It is straightforward to see that the integration  against the Riemannian metric $g$ 
  \begin{align}
      \int_M: C^\infty(M) \rightarrow \C
      \label{eq:def_riegeo_integration-map}
  \end{align}
    is $\T^n$-invariant, where $\C$ is viewed a  trivial $\T^n$ module. The pairing
    \begin{align*}
        \abrac{f,h} \defeq \int_M f \bar h dg, \,\,\,\, \forall f,h \in C^\infty(M),
    \end{align*}
    makes $C^\infty(M)$ into a pre-Hilbert space, and the completion is $\mathcal H = L^2(M)$. For each $t \in \T^n$, the pull-back action $U_t: C^\infty(M) \rightarrow C^\infty(M)$ defined in \eqref{eq:deformation-pullback-functions} can be extended to a unitary operator on $\mathcal H$. It is easy to check the equivariant property of the representation $f \in C^\infty(M) \mapsto L_f \in B(\mathcal H)$, where $L_f$ is the left multiplication operator by $f$:
    \begin{align*}
        U_t L_f U_{-t} = L_{U_t(f)}, \,\,\,\, \forall t \in \T^n.
    \end{align*}
    Recall from section \ref{subsec:deformation-funtens} that for a fixed skew symmetric matrix $\Theta$, the deformed algebra is given by replacing the multiplication in $C^\infty(M)$: $ (C^\infty(M), \times_\Theta)$.  The deformation map $\pi^\Theta: B(\mathcal H)_\infty \rightarrow B(\mathcal H)_\infty$ in section \ref{subsec:deform-oper} gives a new description of the algebra $C^\infty(M_\Theta)$ as a subalgebra of $B(\mathcal H)$.   

    \begin{prop}
     Let $M$ be a  toric Riemannian manifold and $\Theta$ is a skew symmetric matrix.  We denote by $C^\infty(M_\Theta) = \pi^\Theta(C^\infty(M)) \subset B(\mathcal H)$, the image of $C^\infty(M)$ under the deformation map $\pi^\Theta$. Then    
\begin{align*}
\pi^\Theta: (C^\infty(M),\times_\Theta) \rightarrow C^\infty(M_\Theta)
\end{align*}
is an $*$-algebra isomorphism:
 \end{prop}
 
%Due to equation \eqref{eq:pi_theta} and inevitability of $\pi^\Theta$, we may identify the algebra $\pi^\Theta(C^\infty(M))$ as an algebra whose underlying topological space is $C^\infty(M)$ while the multiplication is deformed into $\times_\Theta$. This interpretation is kind of crucial in this paper because most of the identities proven in later sections rely on pointwise properties of functions on the manifold. 
% The
%$\T^{2}$-action comes from  left translation $\alpha_{t}: \T^{2}
%\rightarrow \T^{2}: \alpha_{[t]}([s]) = [t+s]$, for $s,t\in \R^{2}$,
%where $[\cdot]$ means taking equivalent  class in the quotient
%$\R^{2}/\Z^{2}$. \par

As examples, we first compare two sets of notations on noncommutative two tori. 
\begin{eg}[Noncommutative Two Torus] 
    \label{eg:def_riegeo-non-comm-two-torus}
Let $M = \T^{2} = \R^{2}/\Z^{2}$ with the induce flat metric and $(x_{1},x_{2})$ be the coordinates on $\T^{2}$, put $e_{1}(x_{1},x_{2}) =
e^{2\pi i x_{1}}$, $e_{2}(x_{1},x_{2}) =e^{2\pi i x_{2}}$. By
elementary Fourier theory on $\T^{2}$,
$\set{e_{1}^{k}e_{2}^{l}}_{k,l\in\Z}$ serves as basis for the
$C^{\infty}(\T^{2})$, that is 
\begin{align}
  \label{eq:Fourier-Series}
f =\nsum{(k,l)\in\Z}{} f_{(k,l)}e_{1}^{k}e_{2}^{l}, \,\,\, f\in C^{\infty}(\T^{2}),
\end{align}
moreover the Fourier coefficients $\set{f_{(k,l)}}$ are of rapidly
decay in $(k,l)$. 
  For $t = (t_{1},t_{2})\in \R^{2}$, the torus action
is given by
\begin{align}
 \label{eq:-torus-action-nctori}
\alpha_{t}(e_{1}^{k}e_{2}^{l}) = e^{2\pi i t_{1}k+t_{2}l} e_{1}^{k}e_{2}^{l}. 
\end{align}
Hence, the right hand side of \eqref{eq:Fourier-Series} is the isotypical decomposition of the function $f$. 

Let $\theta \in \R\setminus\Q$, denote $\Theta = \brac{
  \begin{matrix}
    0&\theta\\ -\theta &0
  \end{matrix}
}$. The deformed algebra $C^{\infty}(\T^{2}_{\theta})$ is identical to
$C^{\infty}(\T^{2})$ as a topological vector space with the deformed
product
\begin{align}
\label{eq:deformed-product-nctorus}
  f \times_{\theta} g = \nsum{r,s\in\Z^{2}}{} e^{\pi i \abrac{r,\Theta
      s}} f_{r}g_{s} e_{1}^{r_{1}+s_{1}}e_{2}^{r_{2}+s_{2}} =
  \nsum{r,s\in\Z^{2}}{} e^{-\pi i \theta(r_{1}s_{2} - r_{2}s_{1})}
f_{r}g_{s} e_{1}^{r_{1}+s_{1}}e_{2}^{r_{2}+s_{2}},
\end{align}
where $r = (r_{1},r_{2})$, $s = (s_{1},s_{2})$, $f_{r}$ and $g_{s}$
are the Fourier coefficients of $f$ and $g$. Take $f = e_{1}$, $g =
e_{2}$ we see that 
\begin{align*}
  e_{1}\times_{\theta} e_{2} = e^{-2\pi i \theta}e_{2}\times_{\theta} e_{1}.
\end{align*}

Therefore $C^\infty(\T_\theta)$ is isomorphic to the smooth noncommutative two torus\footnote{see  \cite{MR3194491} for the definition.} $A^\infty_\theta$ as a topological algebra.

\end{eg}

%\begin{eg}[Noncommutative $n$-Torus]
%	\label{eg:def-rie_non-n-tours}
%	The general noncommutative $n$-torus is given by the deformation of $n$-torus acting on itself via translations as in example \ref{eg:def-rie-eg-torus}. Let $\Theta = (\theta)_{ij}$ be a skew symmetric matrix of dimension $n$ and then the deformed algebra $C^\infty(\T^n_\Theta)$ is generated by $e_j = e^{\pi i t_j}, j=1,2,\dots n$ subject to the relations 
%	\begin{align*}
%		e_j \times_\Theta e_k = e^{2\pi i \theta_{jk} }e_k \times_\Theta e_j, 
%	\end{align*}
%	where $t = (t_1, \dots t_n)$ are the coordinates on $\T^n$. Elements in $C^\infty(\T^n_\Theta)$ is of the form 
%	\begin{align*}
%		f = \sum_{r\in\Z^n} f_r E^r,
%	\end{align*}
%	where $f_r \in \C$ and $E^r = e_1^{r_1} \cdots e_n^{r_n}$ with $r = (r_1, \dots,r_n)$, while the sequence $\set{f_r}_{r\in\Z^n}$ is of rapidly decay in $r$. 
%\end{eg}
%
\begin{eg}[Noncommutative Four Sphere]\label{eq:def-rie-non-four-sphere}
	
  Let $M = S^4\subset \R^{5}$ be the  unit four sphere with the two-torus rotation action defined in example \ref{eg:def-rie-eg-four-sphere}. Let $\R^5 = \C \oplus \C \oplus \R$ with coordinates $(z_1,z_2,x)$, then the polynomial functions on $S^4$ is generated by the coordinate functions $z_1$, $z_2$ and $x$ with the relation 
  \begin{align*}
  	z_1z_1^* + z_2z_2^* + x^2 = 1.
  \end{align*}
  While the pull-back action on functions on generators  is given by
  \begin{align*}
  	z_1 \mapsto e^{2\pi i t_1} z_1, \,\, z_2 \mapsto e^{2\pi i t_2} z_2, \,\, x \mapsto x,  
  \end{align*}
  where $(e^{2\pi i t_1},e^{2\pi i t_2}) \in\T^2$. Choose $\Theta$ to be
  \begin{align*}
  \Theta = \brac{
  	\begin{matrix}
  	0& \theta \\ -\theta &0
  	\end{matrix}
  	}, \,\,\, \theta \in \R,
  \end{align*} 
  then the resulting multiplication $\times_\Theta$ is presented by the relations on generators: $z_1 \times_\Theta z_2 = e^{2\pi i \theta} z_2 \times_\Theta z_1$ while $x_0$ is central. 
\end{eg}

\begin{eg}
 [Higher dimensional noncommutative tori and spheres]
 Let $\Theta$ be a $n \times n$ skew symmetric matrix. Consider $\T^n$ acts on itself via translations, the resulting deformed algebra $C^\infty(\T^n_\Theta)$ is called a noncommutative $n$-torus. \par
To construct noncommutative spheres, we consider the rotation action of $\T^n$ on $\R^{2n}$ (resp. $\R^{2n+1}$):
\begin{align*}
 t = (t_1, \dots, t_n) \mapsto \brac{
 \begin{matrix}
 e^{2 \pi i t_1} &&\\ &\ddots&  \\ &&e^{2 \pi i t_n}
 \end{matrix}}, \,\,\, 
 %\text{resp.} \,\,\,
%  t = (t_1, \dots, t_n) \mapsto \brac{
%  \begin{matrix}
%  e^{2 \pi i t_1} && &\\ &\ddots& & \\ &&e^{2 \pi i t_n}\\ &&& 1
% \end{matrix}}. 
 \end{align*}
 resp.
 \begin{align*}
   t = (t_1, \dots, t_n) \mapsto \brac{
 \begin{matrix}
 e^{2 \pi i t_1} && &\\ &\ddots& & \\ &&e^{2 \pi i t_n}\\ &&& 1
\end{matrix}}.  
 \end{align*}
  The induced action on $S^{2n-1}$ (resp. $S^{2n}$) gives rise to the noncommutative sphere $C^\infty(S^{2n-1}_\Theta)$ (resp.   $C^\infty(S^{2n}_\Theta)$).
 
\end{eg}

% It is worth to remark that, to deform of $C^\infty(M)$ does not depend on the choice of the Riemannian metric $g$, as long as it is $\T^n$-invariant.  In fact, the deformed algebra $C^\infty(M_\Theta)$ reflects only the smooth structure of the underlying "space".

 The complete geometric space of $M_\Theta$ is given by a spectral triple $(C^\infty(M_\Theta), \mathcal H, D)$ (or a twisted version, cf. \cite[Sec. 1]{MR3194491}). For example, if we assume that the toric Riemannian manifold $M$ is moreover a toric spin manifold with the spinor bundle $\slashed S$ and the Dirac operator $\slashed D$, then  $(C^\infty(M_\Theta), L^2(\slashed S), \slashed D)$ is indeed a spectral triple. Moreover it satisfies all axioms of a noncommutative spin geometry
proposed by Connes (for instance, in \cite{MR1441908}). We shall not touch the detail construction of the spectral triple and examination of those axioms in this paper, 
and refer to \cite{MR1846904}, \cite{MR1937657}, \cite{2010LMaPh..94..263Y}.\par  

%Instead of a spectral triple, the metric data of the noncommutative manifold $M_\Theta$ is given by the Hilbert space $L^2(M)$ whose inner product is built on the measure on $M$ defined by the Riemmanian metric. For the replacement  of the Dirac type operator, we have the scalar Laplacian operator $\Delta$. In order to describe the symbol calculus for pseudo differential operators,  we will deform the Levi-Civita connection $\nabla$ and the underlying tensor calculus. 
%

%\subsection{The Levi-Civita connection on $M_\Theta$}
\subsection{Deformation of tensor calculus}

For a diffeomorphism $\varphi: M\rightarrow M$, one can lift it to the tesnor bundle over $M$ by its differential $d\varphi$ and  the dual $(d\varphi)^*$:
\begin{align}
    d\varphi_x:T_xM\rightarrow T_{\varphi(x)} M,\,\,\,
    (d\varphi_{x})^*: T_{\varphi(x)}^* M \rightarrow T_x^*M, \,\,\, \forall x\in M.
    \label{eq:def-riegeo-dvarphi}
\end{align}

For vector fields $X$ and one forms $\omega$ we have the similar pull-back action:
\begin{align}
    U_\varphi(X)|_p \defeq d\varphi_{\varphi^{-1}(x)} (X|_{\varphi^{-1}(p)}),\,\,\,\,
    U_\varphi(\omega)|_p \defeq (d\varphi_x^{-1})^* \omega|_{\varphi^{-1}(p)},
    \label{eq:def-riegeo-pullback-tensors}
\end{align}
where $p\in M$. We verify that the contraction is equivariant, indeed,
\begin{align*}
    U_\varphi(X) U_\varphi(\omega)|_p& = \omega|_{\varphi^{-1}(p)}\cdot \brac{
        d\varphi^{-1}\circ d\varphi(X|_{\varphi^{-1}(p)}) 
    } \\
    &=\omega|_{\varphi^{-1}(p)} \cdot X|_{\varphi^{-1}(p)} \\
    &= U_\varphi(\omega \cdot X)|_{p}.
\end{align*}
One extends $U_\varphi$ to all tensor fields $\Gamma(\mathcal T M)$ in a natural way such that the tensor product and the contraction are equivariant. In particular, if  $M$ is a toric Riemannian manifold and  $\mathcal T M$ be the tensor bundle of all ranks over $M$,
thus any tensor field $s$ admits a isotypical decomposition $s = \sum_{r\in\Z^n} s_r$.  If we treat the tensor product $\otimes$ and the contraction $\cdot$ as  generalized versions (to all tensor fields) of the multiplication between functions, then \eqref{eq:deformed-product} gives rise to a deformed version: $\otimes_\Theta$ and $\cdot_\Theta$ respectively.  \par

The essential component of the tensor calculus is a connection $\nabla$ with the Leibniz property (see prop. \ref{prop:def_riegeo-Leibnitz-rule}). For this purpose, we required that the torus acts as affine transformations on $M$. Such  transformations $\varphi$ is define by the property of  preserving linear connections. For example,  let $\nabla$ be a linear connection and $\varphi$ be
a affine transformation on $M$, for any vector
field  $X$ and $Y$, we have $\nabla_X Y  = \nabla_{U_\varphi(X)} U_\varphi(Y)$.  Rewrite this according to the notations above: $U_\varphi( \nabla Y) = \nabla U_\varphi(Y)$. In general, we have the following lemma.

% Let $\nabla$ be the Levi-Civita connection with respect to the given metric. Since the torus acting as isometries, we can expect the $\T^n$-equivariant property given in the commutative diagram:
%     \begin{align*}
%         \xymatrix{ \Gamma(\mathcal T^r_s M) \ar[r]^-{\nabla^j}\ar[d]_{U_\varphi} & \Gamma( (T^*M)^{\otimes^j}  \otimes \mathcal T^r_s M) \\
% \Gamma(\mathcal T^r_s M) \ar[r]^-{\nabla^j} & \ar[u]_{U_\varphi^{-1}}\Gamma((T^*M)^{\otimes^j}\otimes \mathcal T^r_s M)
% , }
%     \end{align*}
% here $U_\varphi$ is the pull-back action with respect to an isometry $\varphi$ on $M$. To be precise, we state it as a lemma. 
% 
% As the crucial consequence, the Leibniz's property holds in the deformed setting. To obtain the commutative diagram above, it is suffice to prove the case in which $j=1$:
% The main result of this section is proposition \ref{prop:def_riegeo-Leibnitz-rule}, the Leibniz property of the Levi-Civita connection $\nabla$ with respect to the deformed tensor product and contraction. \par
% 
%In fact,  the Leibniz property of $\nabla$ is equivalent to the $\T^n$-equivariant property demonstrated in the following 
\begin{lem}
	\label{lem:def-riegeo-equiv-nabla}
	 Given an affine transformation $\varphi: M \rightarrow M$ with the pull-back action $U_\varphi$, for any  tensor field $s \in \Gamma(\mathcal T M)$,
  we have:
\begin{align}
 U_\varphi( \nabla s) = \nabla U_\varphi(s). 
\end{align}
\end{lem}
The proof is skipped here. See  section 4 of chapter 1 of \cite{opac-b1085301} for a short exploration, the classical textbook on this subject is  \cite{MR1393940}.

A straightforward consequence of lemma \ref{lem:def-riegeo-equiv-nabla} is that the covariant derivative $\nabla$ preserves the isotypic decomposition, namely, for any tensor field  with its  isotypic decomposition $s = \sum_{r\in\Z^n} s_r$, we have:
\begin{align}
\label{eq:def_riegeo-isodec-connection}
  (\nabla^j s)_r = \nabla^j (s)_r. 
\end{align}
 
Now we are ready to prove the  Leibniz property. 
\begin{prop}[The Leibniz property]
    \label{prop:def_riegeo-Leibnitz-rule}
    Given two tensor fields on $M$, $s_1 \in \Gamma(\mathcal T^r_s M)$ and $s_2 \in \Gamma(\mathcal T^{r'}_{s'} M)$, we have 
    \begin{align}
        \nabla (s_1 \otimes_\Theta s_2) = \nabla s_1 \otimes_\Theta  s_2 + s_1 \otimes_\Theta \nabla s_2.
        \label{eq:mod-cur-productrule-tensors}
    \end{align}
    For the contraction map $\times_\Theta$, given a vector field $X$ and a one-form $\omega$, we get
     \begin{align}
        d (X \cdot_\Theta \omega) = (\nabla X) \cdot_\Theta\omega + X \cdot_\Theta \nabla \omega.
        \label{eq:mod-cur-productrule-contraction}
    \end{align}
\end{prop}

\begin{proof}
    The proof of \eqref{eq:mod-cur-productrule-tensors} and \eqref{eq:mod-cur-productrule-contraction} are almost identical, thus we only show the first one. \par
    \begin{align*}
        \nabla (s_1 \otimes_\Theta s_2)& =\nabla\brac{ 
            \sum_{\mu,\nu \in Z^n}\chi_\Theta(\mu,\nu) (s_1)_\mu \otimes (s_2)_\nu
        } \\
        &=   \sum_{\mu,\nu \in Z^n}\chi_\Theta(\mu,\nu)\brac{ \nabla (s_1)_\mu \otimes (s_2)_\nu
        + (s_1)_\mu \otimes \nabla (s_1)_\mu}, \\
        &=    \sum_{\mu,\nu \in Z^n}\chi_\Theta(\mu,\nu) (\nabla s_{1})_{\mu}\otimes (s_2)_\nu
        + \sum_{\mu,\nu \in Z^n}\chi_\Theta(\mu,\nu)  (s_{1})_{\mu} \otimes (\nabla s_{2})_{\nu}\\
         &= \nabla s_1 \otimes_\Theta  s_2 + s_1 \otimes_\Theta \nabla s_2.
    \end{align*}
   The crucial step is from the second equal sign to the third one, which requires the $\T^{n}$-equivariant property \eqref{eq:def_riegeo-isodec-connection} for the connection. Also we can switch $\nabla$ and the summation $\sum_{\mu,\nu \in Z^n}$ for $\nabla$ is a continuous map with respect to the smooth Fr\'echet topologies.
%    apply equation \eqref{eq:def_riegeo-isodec-connection}, we continue: 
%    \begin{align*}
%                \nabla (s_1 \otimes_\Theta s_2)& = 
%                \sum_{\mu,\nu \in Z^n}\chi_\Theta(\mu,\nu) \brac{
%                  (\nabla s_1)_\mu \otimes (s_2)_\nu
%              + (s_1)_\mu \otimes  (\nabla s_1)_\mu}\\
%        &= \nabla s_1 \otimes_\Theta  s_2 + s_1 \otimes_\Theta \nabla s_2.
%    \end{align*}
    %Here $\chi_\Theta(\mu,\nu) = e^{2\pi i\abrac{\mu, \Theta \nu}}$.
      
\end{proof}

\subsection{Lifting the tensor calculus to the cotangent bundle $T^*M$}
\label{subsec:lift-tensor-cal}
The symbol calculus of pseudo differential operator involves not only the smooth functions but also the whole pull-back tensor fields (smooth sections of pull-back tensor bundles)  on the cotangent bundle $T^*M$: 
%By "pull-back tensor bundles", we mean that the bundles are obtained by pulling back tensor bundles over $M$ via the natural projection $\pi:T^*M \rightarrow M$:
\begin{align}
\label{eq:def-riegeo-puall-back-tensors-defn}
	\mathcal B^r_s M \defeq \pi^* \mathcal T^r_s M \subset \mathcal T^r_z (T^*M),
\end{align}
where $\pi:T^*M \rightarrow M$ is the natural projection, and $r$, $s$ are the contravariant and the covariant rank respectively as before, $\mathcal B M$ denotes the collection of tensor fields of all ranks. From analytical point of view, the only difference between sections of $\mathcal B^r_s M$ and ordinary  tensor fields on $M$ (sections of $\mathcal T^r_s M)$) is that the base point coordinates of the former tensor fields  depend on $(x,\xi)$. \par

As in the previous section, we assume that $\varphi: M\rightarrow M$ is an affine transformation. Keep in mind that $d\varphi_x:T_xM \rightarrow T_{\varphi(x)}M$ for all $x\in M$. The ``inverse dual'' gives rise a lift of $\varphi$ to $T^*M$, denoted by $\varphi^*$
\begin{align}
    \label{eq:def-riegeo-lift-diff-to-cotangent}
 \varphi^*:\xi_x \in T_x^*M \mapsto (d\varphi^*_x)^{-1}(\xi_x) 
 =(d\varphi^{-1}_x)^*(\xi_x) \in T^*_{\varphi(x)}M,
\end{align}
for all $x \in M$. %It commutes with the projection map $\pi:T^*M\rightarrow M$.
For a  function $f\in C^\infty(T^*M)$  on $T^*M$, we define  $U_\varphi(f)(\xi_x) \defeq f((\varphi^*)^{-1}(\xi_x))$ and   similar to  \eqref{eq:def-riegeo-pullback-tensors}, we extend $U_\varphi$ to pull-back tensor fields:
\begin{align}
    U_\varphi(X)(\xi_x) \defeq d\varphi_{\varphi^{-1}x} \brac{X\Big|_{(\varphi^{-1})^*(\xi_x)}}, \,\,\,
U_\varphi(\omega) (\xi_x) \defeq  (d\varphi_x^{-1})^* \brac{
\omega|_{(\varphi^{-1})^* (\xi_x)} 
},
\label{eq:def-riegeo-U_f-tensorfields}
\end{align}
where $X$ is a pull back vector field, hence the evaluation $X|_{(\varphi^{-1})^*(\xi_x)}$ belongs to $T_{\varphi^{-1}x}M$, thus 
\begin{align*}
    d\varphi_{\varphi^{-1}x} \brac{
    X|_{(\varphi^{-1})^*(\xi_x)}
    } \in T_x^*M
\end{align*}
as expected. Similar explanation for the one form $\omega$. As a consequence, we can quickly verify that $U_\varphi(X) U_\varphi(\omega) = U_\varphi(X\cdot \omega)$, which means the natural pairing is equivariant. We extend $U_\varphi$ to  pull-back tensor fields of all ranks in such a way as before  that the pointwise tensor product and contraction are both equivariant. As a result, the deformed tensor product and contraction $\otimes_\Theta$ and $\cdot_\Theta$ can be
extended to all pull-back tensor fields. \par

Since we are on the cotangent bundle, the horizontal differential (along the fibers) and the vertical differential which is similar to the connection on the underlying manifold, appear naturally. Before that, we have to introduce another key ingredient for the calculus. \par
Consider a pseudo differential operator $P$ acting on $C_{c}^\infty(\R^n)$ with  symbol $p(x,\xi) \in C^\infty(\R^n,\R^n)$:
\begin{align}
    (P f)(x) = \int_{\R^n}e^{-i \xi\cdot (x -y)} p(x,\xi)f(y)dyd\eta
    ,\,\,\,  \forall f \in C_{c}^\infty(\R^n). 
\end{align}
The function $l(x,\xi,y) = \xi \cdot (y-x)$ plays a significant role in the quantization map above. Its generalization to manifolds is a smooth function $\ell(\xi_x,y) \in C^\infty(T^*M\times M)$. The linearity in $\xi$ becomes the linearity of $\ell$ on each fiber of $T^*M$, but, the linearity in $x$ has no straightforward analogy. However, when the manifold $M$ is equipped with a connection $\nabla$ (on the cotangent bundle of $M$), the linearity in $x$ can be
described as the vanishing of higher order ($\geq 2$) symmetrized covariant derivatives (along the $y$ variable) $\symd^k \ell(\xi_x,y)$ at $x = y$ for any $k\ge 2$.  \par

Another motivation is related to H\"omander's perspective on pseudo differential operators in \cite{MR0180740}, which is well-explained in \cite{MR2627820}.

\begin{defn}\label{defn:def_riegeo_defn-linearfun}
    Let $M$ be a smooth manifold with a connection $\nabla$ (on the cotangent bundle). Let $\symd^j \defeq \op{Sym}\circ\nabla^{j}$ 
    be the $j$-th symmetrized covariant derivative. A phase function with respect to a given connection is a real-valued smooth function $\ell(\xi_x,y)\in C^\infty(T^*M\times M)$ such that for fixed base point $x$, $\ell$ is linear in $\xi_{x} \in T^*_{x}M$ and such that for all $\xi_{x}$, the symmetrized covariant derivatives (along $y$) satisfies:
    \begin{align}
        \symd^j \ell(\xi_x,y) |_{y=x} = 
        \begin{cases}
            \xi_x, & j=1, \\ 0, & j\neq 1.
        \end{cases}
        \label{eq:def_riegoe-defn-linearfun}
    \end{align}
\end{defn}
The existence of such functions was proved in \cite[Proposition 2.1]{MR560744}. Phase functions defined by  \eqref{eq:def_riegoe-defn-linearfun} are, by no means unique.  Geometrical, $\ell$ can be constructed locally using the exponential map associated to  the given connection: $\ell(\xi_x,y) = \abrac{\xi_x, \exp^{-1}_x y}$, where $\exp$ is the exponential map associated to  the connection $\nabla$ (cf.\cite{MR2627820} and \cite{MR973171}).  Observe that the property
\eqref{eq:def_riegoe-defn-linearfun} is invariant under affine transformations. Namely, for an affine transformation $\varphi$ on $M$, we define the action on $C^\infty(T^*M\times M)$ in the following way to make things equivariant:
\begin{align}
    U_\varphi(\ell)(\xi_x,y) = \ell ((\varphi^{-1})^* \xi_x, \varphi^{-1} y),
    \,\,\, \ell \in C^\infty(T^*M\times M). 
   % \label{eq:def_riegeo-pullback-linearization}
\end{align}
Follows from lemma \ref{lem:def-riegeo-equiv-nabla}, $U_\varphi(\nabla^j \ell) = \nabla^j U_\varphi( \ell)$. In particular, if $\ell(\xi_x,y)$ satisfies \eqref{eq:def_riegoe-defn-linearfun}, so does $U_\varphi(\ell)$. Therefore when the torus acting as affine transformations $\T^n\subset \op{Affine}(M)$, we start with any phase function $\tilde \ell$, the average over $\T^n$
    \begin{align}
       \ell \defeq \int_{\T^n} U_t(\tilde \ell) dt
        \label{eq:def_riegeo-inv-l}
    \end{align}
is $\T^n$-invariant. Hence from now on, we simply assume that $\ell$ is invariant under the torus action. \par

Two important consequences follows from the invariant property. 
\begin{enumerate}
    \item The mixed derivatives are a family of invariant tensor fields on $T^*M$
\begin{align*}
    D^j\nabla^i\ell(\xi_x,x) \defeq D^j\nabla_y^i\ell(\xi_x,y) |_{y=x}.
\end{align*}
Therefore after deformation, they become center elements (commute with everything else) in the deformed tensor calculus.
\item In particular, $d\ell = \nabla \ell$ is invariant. Follows from \eqref{eq:def-riegeo-U_f-tensorfields} and \eqref{eq:def-riegeo-lift-diff-to-cotangent}, the pointwise version can be written as 
    \begin{align}
        \label{eq:def_riegeo-inv-l-pointwise}
        d\ell((\varphi^{-1})^*\xi_x,\varphi^{-1} y) = (\varphi^{-1})^* d\ell(\xi_x, y). 
    \end{align}
    The left hand side  is a covector in $T_{\varphi^{-1} y}^*M$, meanwhile, in the right hand side, we have $d\ell(\xi_x, y) \in T^*_yM$ and 
$(\varphi^{-1})^*: T^*_yM \rightarrow T_{\varphi^{-1} y}^*M$. 
\end{enumerate}

We are ready to define the horizontal and the vertical differential on pull-back tensor fields on $T^*M$ that play the role of $\partial^\alpha_x$ and $\partial^\beta_\xi$ respectively in equation \eqref{eq:def_riegeo-defn-sym-estimate}.

\begin{defn}
\label{defn:def_riegeo-vertical-D}
	The vertical derivative $D$ is the differential along the fibers of on $T^*M$. For any $x \in M$, $p\in \mathcal B M$ a pull-back tensor field, the $j$-th derivative $D^j p$ evaluate at point $\xi_x\in T^*_xM$ gives rise to a $j$-linear function on $T^*_xM$, thus $D^jp$ is a contravariant $j$-tensor. The precise definition is given as follows: for an integer $j\ge 1$, $D^j: C^\infty(T^*M) \rightarrow \Gamma(B^j_0 M)$
	\begin{align}
	\begin{split}
	\,\,&(D^j p)|_{\xi_x}\cdot (\omega_1 \otimes \cdots\otimes \omega_j)\\
	=& \,\, \frac{d}{ds_1}\Big|_{s_1=0}\dots\frac{d}{ds_j}\Big|_{s_j=0}
	p\brac{\xi_x + s_1\omega_1 + \cdots + s_j \omega_j
	},
	\end{split}
		\label{eq:def_riegeo-vertical-der}
	\end{align}
	where $p \in \mathcal B M$, $\omega_1, \dots, \omega_j \in T^*_x M$.
	
\end{defn}

    The vertical differential $D$ is $\T^n$-equivariant: 
\begin{prop}
        \label{prop:def_riegeo-equiv-vertical-der}
    The vertical differential $D$ is equivariant with respect to diffeomorphisms of $M$. Namely, let $\varphi: M \rightarrow M$ be a diffeomorphism, $U_\varphi$ is the induced pull-back action, then 
    \begin{align}
        D^j U_\varphi(p) = U_\varphi \brac{(D^j p)
        }
        \label{eq:def_riegeo-equiv-vertical-der}
    \end{align}
\end{prop}
\begin{proof}
We only prove the case in which $p \in C^\infty(T^*M)$ is a smooth function and $j=1$. The general stituation can be handled in a similar way. \par
Let $\omega(\xi_x) = \omega(x)$ be a pull-back one form, where $\omega(x)$ is a one-form on $M$. We would like to show that $(D U_\varphi(p)) \cdot \omega = (U_\varphi(D p)) \cdot \omega$. For the left hand side, according to \eqref{eq:def_riegeo-vertical-der} and \eqref{eq:def-riegeo-lift-diff-to-cotangent}:
\begin{align*}
    \brac{(D U_\varphi(p)) \cdot \omega}\Big|_{\xi_x} 
   & = \frac{d}{ds}\Big|_{s=0} U_\varphi(p) \brac{
    \xi_x+ s\omega|_x
    }\\ 
    &= \frac{d}{ds}\Big|_{s=0} p\brac{
        (d\varphi^{-1})^*(\xi_x+ s\omega|_x)
    }.
\end{align*}
On the other hand,
\begin{align*}
    \brac{(U_\varphi(D p)) \cdot \omega}\Big|_{\xi_x}& = U_\varphi\brac{
        Dp \cdot U_{\varphi^{-1}}(\omega)
    }\Big|_{\xi_x}  = 
    \brac{
   Dp \cdot U_{\varphi^{-1}}(\omega) 
    }\Big|_{(d\varphi^{-1})^*\xi_x} \\
    &= \frac{d}{ds}\Big|_{s=0} p\brac{
    (d\varphi^{-1})^*\xi_x + U_{\varphi^{-1}}(s\omega) \Big|_{(d\varphi^{-1})^*\xi_x}
    }.
\end{align*}
From \eqref{eq:def-riegeo-lift-diff-to-cotangent}, we see that
\begin{align*}
   U_{\varphi^{-1}}(s\omega) \Big|_{(d\varphi^{-1})^*\xi_x} 
   =  (d\varphi^{-1})^* \brac{ s\omega\Big|_{d\varphi^* \circ (d\varphi^{-1})^*\xi_x}}
   =(d\varphi^{-1})^* (s\omega |_x).
\end{align*}
Therefore the proof is complete. 
%     Start with the right hand side, let $\omega_1, \dots, \omega_j \in T^*_x M$,
%     \begin{align*}
%         &  (D^j U_\varphi(p))|_{\xi_x}\cdot \omega_1 \otimes \cdots\otimes \omega_j\\
%         =&\,\,  \frac{d}{ds_1}\Big|_{s_1=0}\dots\frac{d}{ds_j}\Big|_{s_j=0}
% U_\varphi(p)\brac{\xi_x + s_1\omega_1 + \cdots + s_j \omega_j
%     }\\
%     =& \,\, \frac{d}{ds_1}\Big|_{s_1=0}\dots\frac{d}{ds_j}\Big|_{s_j=0}
%     p\brac{(\varphi^{-1})^*
%     \brac{ \xi_x + s_1\omega_1 + \cdots + s_j \omega_j
%     }
%     }
%     \end{align*}
%     while, the right hand side:
%     \begin{align*}
%         & U_\varphi\brac{ (D^j p)}\Big|_{\xi_x}\cdot \omega_1 \otimes \cdots\otimes \omega_j\\
%         =&\,\, (D^jp)\Big|_{(\varphi^{-1})^* \xi_x} \cdot (\varphi^{-1})^*\omega_1 \otimes \cdots\otimes (\varphi^{-1})^*\omega_j\\
%         =&\,\, \frac{d}{ds_1}\Big|_{s_1=0}\dots\frac{d}{ds_j}\Big|_{s_j=0}
%         p\brac{(\varphi^{-1})^* \xi_x+ s_1 (\varphi^{-1})^*\omega_1 \otimes \cdots\otimes s_j (\varphi^{-1})^*\omega_j
%         }
%     \end{align*}
%     Notice that $(\varphi^{-1})^*$ is linear on the fibers on $T^*M$, then the proof is complete.
\end{proof}

% $\partial_x$ in local coordinates $(x,\xi)$,  the horizontal covariant derivative affects both $x$ and $\xi$ directions. 
% Let $p(\xi_x) \in C^\infty(T^*M)$ be a smooth function on the cotangent bundle, to define the covariant derivative at $\xi_x \in T^*M$, we first extend the valule $p(\xi_x)$ to a small neighborhood of the base point $x$ by considering the following function in $y$: $p(d\ell(\xi_x,y)$.

Finding candidates for  horizontal derivative is not as straightforward as for the vertical one. As a price to pay for a coordinate free construction, the horizontal derivative involves both $x$ and $\xi$ in local coordinate. The two directions are linked by the phase function $\ell(\xi_x,y)$.\par

For a fixed $\xi_x \in T^*M$, the exterior derivative in $y \in M$: $d_y \ell (\xi_x,y)$ gives rise  a one form supported near by $x$, then 
for any smooth function $p \in C^\infty(T^*M)$, the evaluation $p(d\ell(\xi_x,y))$ produces  a smooth function in $y$, which extends the value $p(\xi_x)$ to a small neighborhood of $x$. Indeed, at $x=y$, $d\ell(\xi_x,y)|_{y=x} = \xi_x$. Hence the covariant derivatives $\nabla^{j}_y p(d\ell(\xi_x,y))$, $j=0,1,2,\dots$, make sense near by $x$.

% we define a one form $\omega|_{\xi_x}$ using $\ell$:
%  \begin{align*}
%  (\omega|_{\xi_x})(y) = d_y \ell (\xi_x,y)    
%  \end{align*}
% supported near by $x$. Let $p(\xi_{x}) \in C^{\infty}(T^{*}M)$ be a symbol of some pseudo differential operator, then
% $p(\omega|_{\xi_x}(y))$ is a function in $y$ supported near by $x$. 
 \begin{defn}
 	\label{defn:def_riegeo-horizontal-D}
 	Keep the notations as above. The $j$-th horizontal covariant derivative of a symbol $p$ is given by:
 	\begin{align}
 	(\nabla^{j} p)(\xi_x) = \nabla^{j}_y p(d\ell(\xi_x,y)) |_{y=x},
 	\,\,\,p(\xi_x) \in C^\infty(T^*M).
  	\label{eq:def_riegeo-horizontal-der}
 	\end{align}
 %, where $p(\xi_x)$ is a pull-back tensor field on $T^*M$. 
% 	$\nabla p(\xi_x)$  involves not only the $x$-variables but also the fiber part via the linearization function . 
%thus $p((\omega|_{\xi_x})(y)) \in C^\infty(M)$ supported near by $x$. We define $\nabla p$ at $\xi_x$ to be: 
 \end{defn}
 
\begin{rem} 
	\mbox{}
	\begin{enumerate}[1)]
	\item When evaluating at $y=x$, the value in the right hand side of \eqref{eq:def_riegeo-horizontal-der} does not depend on the choice of the phase function $\ell$ as long as the property \eqref{eq:def_riegoe-defn-linearfun} is fulfilled. 
	\item The vertical and horizontal derivatives $D$ and $\nabla$ can be extended to  all pull-back tensor fields (cf. \eqref{eq:def-riegeo-puall-back-tensors-defn}): $p(\xi_x) \in C^\infty(\mathcal B^r_s M)$, where $(r,s)$ is the rank.  
    \item The vertical $D$ and horizontal $\nabla$ derivatives commute with each other, thus the  \emph{mixed derivatives}  $D^j \nabla^l p(\xi_x)$ is well-defined. 
	\end{enumerate}
     
\end{rem}

%As an analogy of lemma \ref{lem:def_riegeo-compute-covder}, 
%\begin{lem}
%    \label{lem::def_riegeo-ver-coder-exp}
%    Keep the notations in lemma $\ref{lem:def_riegeo-equi-l-function}$.
%    Let $p \in C^\infty(T^*M)$, then the $j$-the vertical covariant derivative is given by 
%    \begin{align}
%       \begin{split}
%        & \,\, ( \nabla^j p )|_{\xi_x} \cdot (Y_1\otimes \cdots\otimes Y_j) \\
%       = & \,\, \frac{d}{ds_1}\Big|_{s_1=0}\dots\frac{d}{ds_j}\Big|_{s_j=0}
%       p\brac{(\omega|_{\xi_x})(\exp_x(s_1 Y_1 + \dots + s_j Y_j)
%       }
%       \end{split}
%        \label{eq:def_riegeo-ver-coder-exp}
%    \end{align}
%\end{lem}

Let us prove the equivariant property for the horizontal derivative $\nabla$.

%\begin{lem} \label{lem:def_riegeo-equi-l-function}
    Let $\varphi:M\rightarrow M$ be an affine transformation on $M$. According to \eqref{eq:def_riegeo-inv-l},  we  can assume that the phase function $\ell$ is $U_\varphi$-invariant:
\begin{align*}
\ell(\tilde\xi_{\tilde x},\varphi^{-1}(y)) = \ell(\xi_x,y), \,\,\,
\text{$y$ is near by $x$}. 
\end{align*}
It follows that $d\ell$ is $U_\varphi$-invariant as well: $U_\varphi(d\ell) (\xi_x, y)= d\ell (\xi_x, y)$. Pointwisely,
\begin{align}
    %\label{eq:def_riegeo-inv-l-pointwise}
    (d\varphi^{-1})^* \brac{
        d\ell \brac{ (d\varphi^{-1})^*\xi_x, \varphi^{-1} y} } = d\ell (\xi_x, y).
\end{align}
   %     and $\ell(\xi_x,y)$ is a $U_\varphi$-invariant phase function in the sense of \eqref{eq:def_riegeo-pullback-linearization}.   For each $\xi_{x}$, denote by $(\omega\Big|_{\xi_x})(y) = d_y \ell(\xi_x,y)$, the one-form supported near by the point $x \in M$, where $d_y$ is the exterior derivative on the $y$-variable. 
%     \begin{align}
%         (\varphi^{-1})^* \brac{(\omega\Big|_{\xi_x})(y)
%     } = (\omega\Big|_{(\varphi^{-1})^* \xi_x})(\varphi^{-1}(y)) 
%         \label{eq:def_riegeo-d-lin-fun}
%     \end{align}
%\end{lem}
\begin{proof}
 The one form $d\ell (\xi_x, y)$ can be treated as a pull-back one form via the projection $\op{pr_2}:T^*M \times M \rightarrow M$, similar to \eqref{eq:def-riegeo-U_f-tensorfields}, the $U_\varphi$ action looks like:
\begin{align*}
    U_\varphi(d\ell) (\xi_x, y) \defeq   (d\varphi^{-1})^* \brac{
        d\ell \brac{ (d\varphi^{-1})^*\xi_x, \varphi^{-1} y} }.
\end{align*}
Let $Y$ be a vector field on $M$ supported near by a point $x$.  Denote $\tilde x = \varphi^{-1}(x)$, $\tilde y = \varphi^{-1}(y)$ and $\tilde\xi_{\tilde x} = (d\varphi^{-1})^*(\xi_x)$.
\begin{align*}
    \brac{U_\varphi(d\ell) \cdot Y} \Big|_{(\xi_x, y)} &=
    (d\ell)(\tilde\xi_{\tilde x},\varphi^{-1} y) \cdot d\varphi^{-1}(Y|_y)\\
    &=
    \frac{d}{ds} \Big|_{s=0}\ell \brac{\tilde\xi_{\tilde x}, \exp_{\tilde y} \brac{
    sd\varphi^{-1}(Y|_y)
} }.
\end{align*}
Since $\varphi$ is affine transformation, $\exp_{\tilde y} \brac{
    sd\varphi^{-1}(Y|_y)
} = \varphi^{-1} \brac{\exp_y(sY|_y)}$. We continue the calculation and use the invaraint property of $\ell$:
\begin{align*}
   \brac{U_\varphi(d\ell) \cdot Y} \Big|_{(\xi_x, y)} & = \frac{d}{ds} \Big|_{s=0}\ell \brac{
       \tilde\xi_{\tilde x}, \exp_{\tilde y} \brac{
    sd\varphi^{-1}(Y|_y)
} }\\
&= \frac{d}{ds} \Big|_{s=0}\ell \brac{
(d\varphi^{-1})^*(\xi_x), \varphi^{-1} \brac{\exp_y(sY|_y)}
} \\
&= \frac{d}{ds} \Big|_{s=0}\ell(\xi_x,\exp_y(sY|_y)) =( d\ell \cdot Y) \Big|_{(\xi_x, y)}.
\end{align*}
    
    \end{proof}

% \begin{proof}
% We first prove that for any $Y \in T^*_{\varphi^{-1} (y)}M$, where $y$ belongs to a small neighborhood of $x$:
% \begin{align}
%    ( \omega\Big|_{\xi_x})(y) \cdot \varphi_* Y = \omega\Big|_{(\varphi^{-1})^* \xi_x}(\varphi^{-1}y) \cdot Y
%     \label{eq:def_riegeo-eq1}
% \end{align}
% The right hand side, by definition, is equal to
% \begin{align*}
%     \frac{d}{dt}\Big|_{t=0}\ell( (\varphi^{-1})^*\xi_x,\exp_{\varphi^{-1}(y)}tY)
%     & = \frac{d}{dt}\Big|_{t=0}\ell\brac{\xi_x, \varphi(\exp_{\varphi^{-1}(y)}tY)
%     } \\
%     &= \frac{d}{dt}\Big|_{t=0} \ell (\xi_x, \exp_y t\varphi_* Y)\\
%     &= \omega(\xi_x)|_y \cdot \varphi_* Y
% \end{align*}
% 
% To finish the proof, 
% \begin{align*}
%     (\varphi^{-1})^* \brac{
%     (\omega\Big|_{\xi_x})(y)
% } \cdot Y =( \omega\Big|_{\xi_x})(y)\cdot \varphi_* Y
%     =  \omega\Big|_{(\varphi^{-1})^* \xi_x}(\varphi^{-1}y)\cdot Y.
% \end{align*}
% 
% \end{proof}

\begin{prop}\label{prop:def-riegeo-equi-horder}
    For any integer $j\ge1$, the $j$-th  vertical covariant derivative $\nabla^j$ is equivariant with respect to the group of affine transformations on $M$. Namely, for any  $\varphi \in \op{Affine}(M)$:
    \begin{align}
        U_\varphi\brac{\nabla^j p
        } = \nabla^j U_\varphi(p),
        \label{eq:def_riegeo-equi-vertical-der}
    \end{align}
    where $p\in \mathcal{B}^r_s M$ is a pull-back tensor field over $T^*M$ of rank $(r,s)$. 
\end{prop}
\begin{proof}
    We will verify the case when $p \in C^\infty(T^*M)$ is only a function and $j=1$. The gerneral cases can be work out in a similar way with the Leibniz property of the connection. We also assume that $\ell$ is $\varphi$ invariant so that \eqref{eq:def_riegeo-inv-l-pointwise} holds. \par
Let $Y$ be a pull-back vector field and we will identify $Y|_{\xi_x}$ with $Y|_x$ in the rest of the computation. We need to show that  
\begin{align*}
    U_\varphi\brac{\nabla p
        } \cdot Y = \nabla U_\varphi(p) \cdot Y. 
\end{align*}
Start with the left hand side,
\begin{align*}
    & \brac{\nabla U_\varphi(p) \cdot Y}\Big|_{\xi_x} \\
   =&\,\,\, \frac{d}{ds}\Big|_{s=0} U_\varphi(p) \brac{
    d\ell(\xi_x, \exp_x s Y|_x)
    } 
    =\frac{d}{ds}\Big|_{s=0} p \brac{
        (\varphi^{-1})^* \brac{
       d\ell(\xi_x, \exp_x s Y|_x) 
        }
    } \\
    =&\,\,\,
    \frac{d}{ds}\Big|_{s=0} p \brac{
        d\ell \brac{
            (\varphi^{-1})^* \xi_x, \varphi^{-1} \exp_x s Y|_x
        }
    }.
\end{align*}
On the other hand,
\begin{align*}
    \brac{ U_\varphi(\nabla p) \cdot Y} \Big|_{\xi_x}& =
    U_\varphi\brac{
        \nabla p \cdot U_{\varphi^{-1}}(Y)
    }\Big|_{\xi_x} = \brac{
        \nabla p \cdot U_{(\varphi^{-1}}(Y)
    }\Big|_{\varphi^{-1})^* \xi_x} \\
    &=\frac{d}{ds}\Big|_{s=0} p \brac{
   d\ell \brac{
       (\varphi^{-1})^* \xi_x,  \exp_{\varphi^{-1}x} s U_{\varphi^{-1}}(Y) |_{\varphi^{-1}x}
        } 
    }.
\end{align*}
To finish the proof, we just have to observe that $U_{\varphi^{-1}}(Y) |_{\varphi^{-1}x} = d\varphi^{-1}(Y|_x)$ and 
\begin{align*}
    \exp_{\varphi^{-1}x} sd\varphi^{-1}(Y|_x) = \varphi^{-1} \exp_x sY|_x, \,\,\, s\ge0,
\end{align*}
provided that $\varphi$ is an affine transformation. 

\end{proof}

Similar to tensor fields on $M$, we define the deformed contraction $\cdot_\Theta$ and tensor product $\otimes_\Theta$ between pull-back tensor fields and 
 the Leibniz rule below follows from the equivariant property proved above. 
\begin{prop}
	 \label{prop:def_riegeo-leibnitz-rul-v-j-der}.
 Given two pull-back tensor fields $s_1, s_2 \in \mathcal B M$, we have  
 \begin{align}
 \begin{split}
 \nabla (s_1 \otimes_\Theta s_2) & = \nabla s_1 \otimes_\Theta  s_2 + s_1 \otimes_\Theta \nabla s_2, \\
  D (s_1 \otimes_\Theta s_2) & = D s_1 \otimes_\Theta  s_2 + s_1 \otimes_\Theta D s_2.
 \end{split}
 \label{eq:def_riegeo-leibnitz-rul-v-j-der}.	
 \end{align}
 Same results holds for the deformed contraction $\cdot_\Theta$ . 
\end{prop}

%% file: pse_op.tex
% !TEX root =  main.tex

\section{ Pseudo differential operators on toric noncommutative manifolds}
\label{sec:pse-ops-non-mflds}

Let $M$ be a closed  manifold with a torus action $\T^n \subset \op{Diff}(M)$ as before.  We would like to apply the construction in section \ref{subsec:deform-oper} to $\Psi(M)$, the algebra of pseudo differential operators acting on $C^\infty(M)$. Namely, the algebra of pseudo differential operators on the noncommutative manifold $M_{\Theta}$ is just the image of $\Psi(M)$ under the deformation map 
 $\pi^\Theta$ (see definition \ref{defn:def-ops-pitheta}). We shall describe a Fr\'echet topology for pseudo differential operators so that the convergences of  the series in \eqref{eq:deformed-product} make sense. \par
%  Since half of the pseudo differential operators are unbounded, the convergence of the twisted convolution in equation \eqref{eq:deformed-product} needs to be re-investigate in the Fr\'echet topology discussed in the previous section. \par
%%Let $M$ be a toric Riemannian manifold (cf. definition \ref{defn:def-rie-eg-toricmflds}) with a $\T^n$ action. The main goal of this section is to extend the deformation map $\pi^\Theta$ (cf. \eqref{eq:deform-oper-pi-theta-2}) to the whole algebra of pseudo differential operators. The essential technical obstacle is the  convergence of the isotypical decomposition of a pseudo differential operator $P = \sum_{r\in\Z^n} P_r$ when $P$ is unbounded.   
%

Denote by $\Psi^{j}(M)$, $j\in\Z$, consists of pseudo differential operators whose symbol $p(x,\xi)$, when localized on some open chart, belongs to $S\Sigma^{j}(M)$ defined in \eqref{eq:def_riegeo-defn-sym-estimate}. As usual, we put
    \begin{align}
        \Psi(M) = \ucup{d\in\Z}{} \Psi^d(M), \,\,\, \Psi^{-\infty}(M) = \icap{d\in\Z}{} \Psi^d(M).
        \label{eq:pse-op-wholeops-smoothingops}
    \end{align}

In this paper, we only need pseudo differential operators of integer orders. The following characterization (called Beals-Cordes type in \cite{MR1371144}) of zero-order pseudo differential operators was proved in \cite{MR0435933} with a correction \cite{MR523608}, also in \cite{MR0535694,MR0474431}.
%proposition was stated in \cite{MR1371144}, but the proof goes back to \cite{MR0435933} and \cite{MR1314815}.

\begin{prop} \label{prop:char-pseu-op-zero-order} 
  Let $M$ be a closed manifold and $\Psi^{0}(M)$ be the
  space of all zero order pseudo differential  operators. Given a  operator $P:C^\infty(M)\rightarrow
C^\infty(M)$, then $P \in \Psi^{0}(M)$ if and only if for any finite
  collection of first order differential operators $\mathfrak F =
  \set{F_{1},\cdots,F_{l}}$, we have  
  \begin{align}
    \label{eq:char-pseu-op-zero-order}
    \op{ad}F_{l}\cdots \op{ad}F_{1}(P)\in B(L^{2}(M))
  \end{align}
where $\op{ad}F_{j}(P) = [F_{j},P]$, $1\leq j\leq l$.
\end{prop}

%\begin{proof}
%Given $P \in \Psi^d(M)$ and $F$ a differential operator of order one, it is well-known that the commutator $[F,P]$ is a pseudo differential operator of order $d+1-1 = d$. Therefore $\op{ad}F_{l}\cdots \op{ad}F_{1}(P) \in  \Psi^0(M)$ provided that $P \in  \Psi^d(M)$, thus the boundedness follows. \par
%Conversely, let us assume the boundedness in equation  \eqref{eq:char-pseu-op-zero-order}, we would like to show that $P$ belongs to $\Psi^0(M)$. Indeed, the problem is local: choose an finite open cover $\set{O_j}$ for $M$ with a partition of unity $\set{\varphi_j}$, then $P = \sum_j P_j = \sum_j \varphi_j P \varphi_j$, it is suffices to show that each $P_j$ is a pseudo differential operator on $\R^m$. Observe that each $(P_j)\indices{_{(\beta)}^{(\alpha)}}$ defined in  \eqref{eq:pse_op-Palphabeta} can be written as a finite sum of terms like $\op{ad}F_{l}\cdots \op{ad}F_{1}(P)$, hence is bounded on $L^{2}(M)$, apply proposition \ref{prop:pse_op-symtop-opnormtop-R^m}, we conclude that each $P_j$ is a pseudo differential in local coordinates. Hence the sum constitute a pseudo differential operator on the manifold $M$. 
%\end{proof}

The existence of elliptic operators allows us to extend the characterization above to pseudo differential operators of all integer orders. 
\begin{cor}\label{cor:char-pseu-op-arbitary-order}
  Let $M$ be a closed manifold with associated Sobolev
  spaces $\set{\mathcal H_{s}}_{s\in\R}$ and $\Psi^{d}(M)$ be the
  space of pseudo differential  operators of order $d \in \Z$. Given a continuous linear operator $P: C^\infty(M) \rightarrow C^\infty(M)$ that admits a bounded extension from $\mathcal H_{d}$ to  $\mathcal H_{0}$, $P$  belongs to $\Psi^{d}(M)$ if and only if
for any finite
  collection of first order differential operators $\mathfrak F =
  \set{F_{1},\cdots,F_{l}}$, we have 
\begin{align}
    \label{eq:char-pseu-op-arbitary-order}
    \op{ad}F_{l}\cdots \op{ad}F_{1}(P)\in B(\mathcal H_{d},\mathcal H_{0}).
  \end{align} 
\end{cor}

Corollary \ref{cor:char-pseu-op-arbitary-order} leads us to consider the following family of semi-norms on
$\Psi^{d}(M)$ indexed by $(j,\mathfrak F)$, where $j\in\Z$ and $\mathfrak F =
\brac{F_{1},\cdots,F_{k}}$ is a finite collection of first order 
differential operators. A crucial property for pseudo differential operators acting on functions is that the commutator of two operators $[P,Q]$ is of order $\op{ord}(P)+\op{ord}(Q)-1$. In particular, if $P \in \Psi^{d}(M)$,  the iterated commutator:
$[F_{k},\cdots,[F_{1}, P]] $ still belong to $\Psi^{d}(M)$, thus defines
a bounded operator on $\mathcal H_{j+d} \rightarrow \mathcal H_{j}$. 
We define a semi-norm $\norm{\cdot}_{(j,\mathfrak F)}$ on  $\Psi^{d}(M)$ as follows:
\begin{align}
  \label{eq:pse_op-pseops-op-seminorms}
  \norm{P}_{(j,\mathfrak F)} = \norm{[F_{k},\cdots,[F_{1}, P]]}_{j+d,j},
\end{align}
where on the right hand side, $\norm{\cdot}_{j+d,j}$ is the operator norm
from $\mathcal H_{j+d}(M) \rightarrow \mathcal H_{j}(M)$. \par
Corollary \ref{cor:char-pseu-op-arbitary-order} implies that any Cauchy sequence in $\Psi^{d}(M)$ with respect to the family of semi-norms above converges to a pseudo differential operator. Due to the compactness of $M$, one can find a a countable increasing subfamily in the seminorms that define the same 
Fr\'echet topology. We summarize the facts as below:
\begin{prop}\label{prop:pse-op-Frechetproperty-pseudo-ops}
%\mbox{}
 Keep the notations as above.
\begin{enumerate}[$1)$]
 \item For all $d\in\Z$, the semi-norms $\norm{\cdot}_{(j,\mathfrak F)}$ defined in \eqref{eq:pse_op-pseops-op-seminorms} make $\Psi^d(M)$ into a Fr\'echet space. 
 \item For $d_1< d_2 \in \Z$, then the inclusion 
   $(\Psi^{d_1}(M), \norm{\cdot}_{(d_1,\mathfrak F)}) \rightarrow  (\Psi^{d_2}(M), \norm{\cdot}_{(d_2,\mathfrak F)})$ is continuous, which makes $\Psi^{d_1}(M)$ into a closed sub-Fr\'echet space of $\Psi^{d_2}(M)$.
 \item The subspace of smooth operators $\Psi^{-\infty}(M) = \icap{s\in \R}{} \Psi^{s}$ is a two-sided closed ideal in  $\Psi(M) = \ucup{s\in \R}{} \Psi^{s}$. 
\end{enumerate}
    
%    In particular, if $d_1< d_2$, then the inclusion 
%    $(\Psi^{d_1}(M), \norm{\cdot}_{(d_1,\mathfrak F)}) \rightarrow  (\Psi^{d_2}(M), \norm{\cdot}_{(d_1,\mathfrak F)})$ is continuous. 
\end{prop}
It is well-known that pseudo differential operators on a closed manifold $M$ is stable under the action of the diffeomorphism group of $M$, more precisely, given any diffeomorphism $\varphi: M \rightarrow M$, let $U_\varphi: C^\infty(M) \rightarrow C^\infty(M)$ be the pull-back operator defined in \eqref{eq:deformation-pullback-functions}, for any $P \in \Psi^d(M)$, then the conjugation $U_\varphi P
U_\varphi^{-1}$ still belongs to $\Psi^d(M)$. Therefore the Fr\'echet spaces $\Psi^d(M)$ ($d\in\Z$) become $\T^n$-modules via the adjoint action: 
\begin{align}
    t\cdot P \defeq \op{Ad}_t(P) = U_t P U_{-t},\,\,\, \forall t \in \T^n, \,\,\,P \in \Psi^d(M).
    \label{eq:pse_op-adjoint-Psi-order-d}
\end{align}
%where $U_t: C^\infty(M) \rightarrow C^\infty(M)$ is the pull-back action (cf. \eqref{eq:deformation-pullback-functions}) with respect to an isometry $t$. 
%We need to show that the action is smooth so that the deformation can be applied. 

In order to apply the deformation machinary, we need to show that the function: 
$t \mapsto \op{Ad}_t(P)$ is smooth in $t \in \T^n$ with respect to the F\'echet topology on $\Psi^d(M)$ (cf. estimate \eqref{eq:smooth-topology-defn}), for all $d\in\Z$. 

Let $t = (t_1,\dots,t_n)$ be a coordinate system on  $\T^n$, and then $\set{\partial_{t_1},\dots, \partial_{t_n}}$ constitute a basis of the Lie algebra, the push-forward vector fields via the action on $M$ are denoted by $\set{X_1,\dots, X_n}$.

\begin{prop}
% Let $M$ be a toric Riemannian manifold with a $n$-torus action as
% before and $\mathcal H_{s}$ $(s\in\R)$ be the family of Sobolev spaces
% associated to functions on $M$. Let $t = (t_1,\dots,t_n)$ be a coordinate on the torus $\T^n$, and then $\set{\partial_{t_1},\dots, \partial_{t_n}}$ constitute a basis of its Lie algebra, the push-forward vector fields on $M$ are denoted by $\set{X_1,\dots, X_n}$.  
Keep the notations as above. Given $P\in \Psi^{d}(M)$ be a pseudo differential operator of order $d$, the operator-valued function $t \rightarrow \op{Ad}_{t} (P)$ is  smooth in $t$. Moreover, for any  finite collection of first order differential
 operators $\mathfrak F = \set{F_{1},\cdots,F_{k}}$ and any
 multi-index  $\mu = (\mu_{1},\cdots, \mu_{j})$, one can find  another finite collection of first order operators $\mathfrak
 F'$ such that: 
 \begin{align}
   \label{eq:smooth-topology-for-pseudo-op}
   \norm{\partial^{\mu}_{t}\op{Ad}_{t} (P)}_{(s,\mathfrak F)} \leq C \norm{P}_{(s,\mathfrak F')},
% \frac{1}{k!}\nsum{\abs\mu\leq k}{}
%    \norm{P}_{(s,\mathfrak F_{\mu})}, 
   \,\,\, t\in \T^{n}, \,\,\, s \in\R.
 \end{align}
 where the constant $C$ depends on $\mathfrak F$ and $\mu$. The subscript $s \in \R$ means that the operator norm is the one from the $s+d$-th to the $s$-th Sobolev space.  
\end{prop}

\begin{proof}
     Apply the product rule onto \eqref{eq:pse_op-adjoint-Psi-order-d}, we see that 
\begin{align*}
  \partial_{t_i}\brac{\op{Ad}_{t}(P)} = \op{Ad}_{t}(\op{ad}(X_{i})(P)),
\end{align*}
where $\op{ad}(X_{i})$ is the commutator: $\op{ad}(X_{i})(P) =
[X_{i},P]$. Similarly, the higher order  partial derivatives are given by:
\begin{align}\label{eq:higher-order-derivatives}
    \partial_{t_{\mu_{1}}} \cdots \partial_{t_{\mu_{j}}}\brac{\op{Ad}_{t}(P)} =
 \op{Ad}_{t} \brac{\op{ad}(X_{\mu_{1}})\cdots \op{ad}(X_{\mu_{j}}) (P)}.
  % U_{t}[\op{ad}(X_{i_{1}},\cdots [X_{i_{j}}, P]]U_{-t} =\op{Ad}_{t}\brac{[X_{i_{1}},\cdots [X_{i_{j}}, P]]} .
\end{align}
Since the right hand side above is a pseudo differential operator, we have proved that the function $t \rightarrow \op{Ad}_{t} (P)$ is  smooth. \par
To show the estimate \eqref{eq:smooth-topology-for-pseudo-op}, we observe that for any vector field $F$, 
\begin{align*}
 \op{ad}(F) \circ \op{Ad}_t = \op{Ad}_t \circ \op{ad}( \op{Ad}_{-t}(F)).
\end{align*}
Thus for $\mathfrak F = \set{F_{1},\cdots,F_{l}}$, a finite collection of vector fields on $M$, 
\begin{align*}
  &\op{ad}(F_{1})\cdots \op{ad}(F_{l})\brac{\partial^{\mu}_{t} \op{Ad}_{t}(P)} \\
 =\,\, & \op{ad}(F_{1})\cdots \op{ad}(F_{l})\brac{\op{Ad}_{t} \brac{\op{ad}(X_{i_{1}})\cdots \op{ad}(X_{i_{j}})(P)}
} \\
=\,\, & \op{Ad}_{t}\brac{\op{ad}(\op{Ad}_{-t} (F_{1}))\cdots \op{ad}(\op{Ad}_{-t} (F_{1})) \op{ad}(X_{\mu_{1}})\cdots \op{ad}(X_{\mu_{j}}) (P)}.
\end{align*}

Since the torus action is unitary, we can drop $\op{Ad}_{t}$ when computing operator norms:
\begin{align*}
  \norm{\partial_t^\mu \op{Ad}_{t}(P)}_{(s,\mathfrak F)} \leq C \norm{ P }_{(s,\mathfrak F')} 
\end{align*}
with 
\begin{align*}
  \mathfrak F' = \set{F_1,\cdots, F_l, X_{\mu_{1}},\cdots,X_{\mu_{k}}}.
\end{align*}
%Same arguments works as well when $\partial_{i_{1}\cdots
%      i_{k}}$ is replaced by $\partial^{\mu}$ for any multi-index $\mu$.

% Combine \eqref{eq:higher-order-derivatives} and
% \eqref{eq:smooth-topology-seminorms}, the estimate
% \eqref{eq:smooth-topology-for-pseudo-op} follows directly. 
\end{proof}

Follows from the smoothness of the torus action, the right hand side of the isotypical decomposition 
\begin{align*}
    P = \sum_{r\in\Z^n} P_r, \,\,\,\,\, \forall P \in \Psi^d(M),
\end{align*}
with $P_{r} = \int_{\T^{n}} \op{Ad}_{t}(P) e^{-2\pi ir \cdot t}
dt$, converges to $P$ with respect to the Fr\'echet topology defined above. Moreover, for each semi-norm $\norm{\cdot}_{d,\mathfrak F}$, the sequence $\norm{P_r}_{d,\mathfrak F}$ is of rapidly decay in $r$.  
Fixed a skew symmetric $n\times n$ matrix $\Theta$, the definition  \ref{defn:def-ops-pitheta} of  the deformation map $\pi^\Theta$ is extended to pseudo differential operators of all orders, namely, 
for all $P \in \Psi(M)$, $\pi^\Theta (P)$ is defined by 
\begin{align}
  \pi^\Theta (P)(f) \defeq \sum_{r,l\in\Z^n} P_r (f_l),\,\,\,
  \forall f \in C^\infty(M). 
\end{align}
 Alternatively,
\begin{align}
\pi^\Theta (P) \defeq \sum_{r\in\Z^n} P_r U_{r.\Theta/2},
    \label{eq:pse-op-pi-theta-defn}
\end{align}
where $r \cdot \Theta/2$ denotes the matrix multiplication in which $r$ is a row vector. \par
For each $d\in\Z$, we denote by $\Psi^d(M_\Theta)$ the image of $\Psi^d(M)$ under $\pi^\Theta$,  $\Psi(M_\Theta)$ and $\Psi^{-\infty}(M_\Theta)$ are the union and intersection as in \eqref{eq:pse-op-wholeops-smoothingops} respectively. Due to the continuity of the map $\pi^\Theta$ (cf. \eqref{eq:deform-oper-frechet-semi-norms-estimate}) we see that the order of $P$ is stable under the deformation:
\begin{lem}
  \label{prop:deform-pseudo-diff-preserves-the-order}
  Let $P\in \Psi^{d}(M)$ is a pseudo differential operator on $M$ of
order $d \in \Z$. Then $\pi^{\Theta}(P): C^{\infty}(M) \rightarrow
C^{\infty}(M)$ define in \eqref{eq:pse-op-pi-theta-defn} extends to a
bounded operator from $\mathcal H_{s} \rightarrow \mathcal H_{s-d}$
for all $s \in \R$.
\end{lem}

We summarize some crucial properties of $\Psi(M_\Theta)$ in the following proposition.

\begin{prop}\label{prop:theta-prod-well-defined}
Let $\Theta$ be a $n\times n$ skew symmetric matrix. The filtrated algebra $\Psi(M)$ of all pseudo differential operators admits the following deformation. For any pseudo
differential operators $P$ and $Q$, order $d_{1}$ and $d_{2}$
respectively, the $\times_{\Theta}$ multiplication 
\begin{align}
  \times_{\Theta}:& \Psi^{d_{1}}\times \Psi^{d_{2}} \rightarrow \Psi^{s_{1} + s_{2}}\\
  &(P,Q) \mapsto P\times_{\Theta}Q =  \nsum{r,l\in\Z^{n}}{}e^{\pi i
    \abrac{r,\Theta l}} P_{r}Q_{l}
    \label{eq:pse-op-theta-prod-pseops}
\end{align}
is  well-defined. Due to the skew symmetric property of $\Theta$, the $\times_\Theta$ multiplication is compatible with the original $*$-operation in $\Psi(M)$, namely:
\begin{align}
    \label{eq:pse-op-*-pseops-theta}
    ( P \times_\Theta Q)^* = Q^* \times_\Theta P^*.
\end{align}
Therefore $(\Psi(M),\times_\Theta)$ is a filtrated $*$-algebra. Follows from  lemma $\ref{lem:def-ops-deform-oper-pi-theta-adjoint}$ and
proposition $\ref{prop:def-op-deform-oper-pi-theta-prod}$, we obtain that the deformation map $\pi^\Theta$ makes $(\Psi(M_\Theta),\cdot)$ into a filtrated $*$-algebra, where $\cdot$ denotes the composition between operators. More explicitly, we have for any $P,Q \in \Psi(M)$,
\begin{align}
    \label{eq:pse-op-pitheta-algera-morphism}
    \pi^\Theta( P \times_\Theta Q) = \pi^\Theta(P) \pi^\Theta(P),\,\,\,\,
    \pi^\Theta(P^*) =  \pi^\Theta(P)^*.  
\end{align}
\end{prop}

\begin{proof}

   Notice that composition of operators 
    \begin{align*}
  \Psi^{s_{1}}\times \Psi^{s_{2}} \rightarrow \Psi^{s_{1} + s_{2}} :
  (P,Q) \mapsto PQ, \,\,\,\, s_{1},s_{2} \in\R,
\end{align*}
satisfies the jointly continuity (cf. \eqref{eq:jointly-continuity}) with respect to the operator norms, plus the rapidly decay property in the components of the  isotypical decomposition, we conclude that  the infinite sum in the right hand side of \eqref{eq:pse-op-theta-prod-pseops} converges with respect to the Fr\'echet topology. \par %moreover, the limit is indeed a pseudo differential operator via corollary  \ref{cor:char-pseu-op-arbitary-order}.
After justifying the convergence,  \eqref{eq:pse-op-*-pseops-theta} is an instance of proposition \ref{prop:deformation-along-tn-associativity-theta-multi}, while \eqref{eq:pse-op-pitheta-algera-morphism}  is a straightforward generalization of lemma $\ref{lem:def-ops-deform-oper-pi-theta-adjoint}$ and proposition  \ref{prop:def-op-deform-oper-pi-theta-prod}. At last, the deformation $\pi^\Theta$ has an inverse $\pi^{-\Theta}$ (cf. \ref{lem:deform-oper-theta+theta'}), therefore, it is an filtrated $*$-algebra isomorphism
between $(\Psi(M),\times_\Theta)$ and $(\Psi(M_\Theta),\cdot)$. 
\end{proof}

% \begin{defn}\label{defn:deformed-pse-ops}
% The deformed algebra of pseudo differential operators 
% \begin{align}
%   \label{eq:deformed-pse-ops}
% \Psi^\infty(M_\Theta) = \set{\pi^{\Theta}(P) \,|\, P \in \Psi(M)}  
% \end{align}
% is a subalgebra of operators on
% Sobolev spaces $\set{\mathcal H_{s}}_{s\in\R}(M)$ with common domain
% $C^{\infty}(M)$. Similar to the deformation of $C^{\infty}(M)$, $\Psi^\infty(M_\Theta)$ can be identified with $\Psi(M)$ at the level of topological
% vector spaces endowed with the $\times_{\Theta}$-multiplication in
% Proposition  \ref{prop:deformed-product-of-ops}. Also from
% lemma \ref{prop:deform-pseudo-diff-preserves-the-order} and
% proposition \ref{prop:deformation-adjoint}, the filtration and
% $*$-operation on $\Psi^\infty(M)$ can be moved to $\Psi^\infty(M_\theta)$ by this
% identification. 
% \end{defn}

% \begin{rem}
%   Combine lemma \ref{lem:smoothing-ops} and lemma
% \ref{lem:deform-pseudo-diff-boundedness}, we conclude that the
% deformation map $\pi^{\Theta}$ maps $\Psi^{-\infty}(M)$ into
% itself. We denote by $\pi^{\Theta}(P)\backsim \pi^{\Theta}(Q)$ when $\pi^{\Theta}(P) - \pi^{\Theta}(Q)$ is a smoothing operator.
% \end{rem}

Similar to the commutative case, we define the deformed algebra of classical pseudo differential operators to be the  quotient $\pi^\Theta(\Psi(M)) / \pi^\Theta(\Psi^{-\infty}(M))$, where $\Psi^{-\infty}(M)$ is the space of smoothing operators.  Given a pseudo differential operator $P$, the deformation $\pi^\Theta(P)$  is not a pseudo differential operator anymore. Indeed, $\pi^\Theta(P)$ is not pseudo-local (cf. \cite[lemma 1.2.7]{gilkey1995invariance}) in general due to the fact that the support of $P$ is distorted by the torus action. However,  for smoothing operators, we do have $\pi^\Theta(\Psi^{-\infty}(M))
= \Psi^{-\infty}(M)$.

% \begin{lem}
%     \label{lem:pse_op-smoothing-ops-two-defns}
%     Let $M$ be a compact Riemannian manifold without boundary.   Let $P: C^\infty(M) \rightarrow C^\infty(M)$ be a continuous (with respect to the Fr\'echet topology given by the partial derivatives in local coordinate) linear operator. 
%     
%     
%     Then $P \in \Psi^{-\infty}(M)$ if and only if $P$ has a smooth the Swartz kernel $K(x,y)$ over $M \times M$. 
% \end{lem}

\begin{lem}  \label{lem:pse_op-smoothing-ops-two-defns}
    Let $M$ be a compact smooth manifold without boundary.    Given $P: C^\infty(M) \rightarrow C^\infty(M)$  a continuous $($with respect to the Fr\'echet topology given by the partial derivatives in local coordinate$)$ linear operator with the smoothing property that for any $s,t \in \R$, $P$ can be extended to a bounded operator between associated Sobolev spaces: from $\mathcal H^s$ to $\mathcal H^t$. In other words, there exits a constant $C_{s,t}$ such that  for all $f \in C^\infty(M)$
    \begin{align}
        \norm{P f}_s \le C_{s,t} \norm{f}_t.
        \label{eq:pse_op-smoothing-ops-two-defns}
    \end{align}
Then the   the distributional kernel of $P$, $K(x,y)$ over $M \times M$,  belongs to $C^\infty(M\times M)$. 
\end{lem}
\begin{proof}
    The proof can be achieved by applying  the Sobolev lemma to the estimate in \cite[Lemma 1.2.9]{gilkey1995invariance}
\end{proof}

Let $P \in \Psi^{-\infty}(M)$ be a smoothing operator, then the estimate \eqref{eq:pse_op-smoothing-ops-two-defns} holds for $\pi^\Theta(P)$ for all real number $s$. Hence the operator $\pi^\Theta(P)$ has a smooth Schwartz kernel by  lemma \ref{lem:pse_op-smoothing-ops-two-defns} above. We summarize the fact in the proposition below:
\begin{prop} \label{prop:pse_op-smoothing-ops-pi-theta}
    Let $\mathcal H$ be the Hilbert space of $L^2$-functions on a toric Riemannian manifold $M$. As a $\T^n$-smooth subspace of $B(\mathcal H)$, $\Psi^{-\infty}(M)$ is stable under the deformation map $\pi^\Theta$, that is
    \begin{align*}
        \pi^\Theta\brac{\Psi^{-\infty}(M)
        }  \subset \Psi^{-\infty}(M).
    \end{align*}
\end{prop}

We define the deformation of classical pseudo differential operators $\op{CL}(M_\Theta)$ to be the quotient:

 \begin{align}
     \op{CL}(M_\Theta) & = \Psi^\infty(M_\Theta)/ \Psi^{-\infty}(M) = \pi^\Theta(\Psi^\infty(M))/\Psi^{-\infty}(M)
     \\
     & = \pi^\Theta \brac{ \Psi^\infty(M)/ \Psi^{-\infty}(M)
     }   = \pi^\Theta(\op{CL}(M)).
     \label{eq:pse_op-def-classical-pseops}
 \end{align}

%%% Local Variables: 
%%% mode: latex
%%% TeX-master: "main"
%%% End: 

%% file: wpse_diff_ops.tex
% !TEX root =  main.tex

\section{Widom's pseudo differential calculus}
\label{sec:wid's-cal}

In the literature, symbol calculus for pseudo differential operators on manifolds was developed by pasting Fourier integral operators  on open sets of $\R^n$. In such construction, only the leading symbol of an operator  is well-defined as a function on the cotangent bundle, the rest of them  depend heavily on the chosen local coordinates and the transformation rules between different coordinate systems are quite cumbersome. It is a natural question to ask for an invariant (independent of the choice of
coordinates) construction of such calculus. The construction was developed long time ago, in Widom's work  \cite{MR560744} and \cite{MR538027}, also \cite{MR0310708}, later various modifications were suggested \cite{MR2627820}, \cite{MR973171}, \cite{MR1425328}, \cite{MR728863}, \cite{MR2124544}. As a price to pay, the resulting symbol calculus is quite sophisticated geometrically and combinatorially: tensor calculus is heavily involved several multi-indices
are required in the expression of the product formula of two symbols. In this section, we will follows Widom's work, focus on explaining the notations appeared in the main results and refer most of the technical proofs to \cite{MR560744} and \cite{MR538027}. 

% We will briefly recall Widom's symbol calculus and its deformation to the noncommutative setting. The main result is the asymptotic formula  $p \star q$ for the product of two symbols. We shall focus on explaining the
% 
% All the proofs can be found in Widom's work  \cite{MR560744} and \cite{MR538027}, therefore won't be repeated here.  
\par
In this section, $M$ is a always a compact smooth manifold without boundary endowed with a torsion free connection $\nabla$.

%\subsection{The symbol calculus}
% \subsection{Connection and differential operators}
% 
% 
% We start with some basic facts about the symmetrized covariant derivatives $\symd$. 
% Let $\op{Sym}$ denote the symmetrization of tensors in all their indices, and denote by  $\stensor$, the symmetric tensor:
% \begin{align}
% u \stensor v \defeq \op{Sym}( u \otimes v)
% \label{eq:wpse_sym}
% \end{align}
% then one can quickly verify that
% \begin{align}
% \op{Sym}(\nabla^k (u \otimes v)) = \sum_{j=0}^k \binom ki (\op{Sym}(\nabla^j u))\stensor 
% \op{Sym}(\nabla^{k-j} v).
% \label{eq:wpse_sym-nabla}
% \end{align}
% For any integer $j\in \N$, the $j$-th symmetrized covariant derivative  $\symd^j$ is equal to $\nabla^j$ followed by a symmetrization:
% \begin{align}
% \symd^j u  = \op{Sym}(\nabla^j u).  
% \label{eq:wpse_sym_derivative}
% \end{align}
%  The higher order product rule will be needed later:
% \begin{align}
% \symd^j (u_1 \otimes \cdots \otimes u_r) =\sum\frac{j!}{i_1 \cdots i_r}\symd^{i_1} u_1 \stensor \dots \stensor \symd^{i_r} u_r,
% \label{eq:wpse_sym_derivative_prodrule}
% \end{align}
% where the sum is over $i_1 + \cdots i_r = k$. \par
%  Each symmetric contravariant tensor $\rho_\alpha$ can be identified with a  polynomial function  on $T^* M$: $\rho_\alpha(\xi_x)$, which is indeed a componet of the total  symbol of $P$. Such correspondent between operators and symbols will be extended to all pseudo differential operators in this section. \par
% 
  
  \subsection{The symbol map and the quantization map }
We start with an observation on differential operators. Using the connection $\nabla$, any differential operator $P$ on $M$ can be defined in a coordinate-free way as of the form of a finite sum:
\begin{align}
\label{eq:defn:wpse-diff-dop-coordfree}
  P(f) = \sum_\alpha \rho_\alpha \cdot (\nabla^\alpha f),
\end{align}
where each $ \rho_\alpha$ is a  contravariant tensor fields  such that the contraction $ \rho_\alpha \cdot (\nabla^\alpha f)$ gives rise a smooth function on $M$. If the tensor field $\rho_\alpha$ is symmetric, then the polarization gives rise a polynomial on the cotangent bundle, which leads to the classical notion of symbols of differential operators.  The  bottom line here is that if we replace the contraction in \eqref{eq:defn:wpse-diff-dop-coordfree} by the deformed version $\cdot_\Theta$, then we
recovered the deformed operator $\pi^\Theta(P)$.\par

Let $\ell(\xi_x,y) \in C^\infty(T^*M\times M)$ be a phase function as in definition \ref{defn:def_riegeo_defn-linearfun} with respect to the given connection $\nabla$.    
Then $\ell(\cdot,y)$ can be thought as vector field on $M$ (supported near by the point $y$). With the interpretation, we denote:
\begin{align}
    \ell(\cdot,y)^k = \ell(\cdot,y)\otimes \cdots \otimes \ell(\cdot,y),
    \label{eq:wpse_tensor_l}
\end{align}
which is a symmetric $k$-th order contravariant tensor field and so the pairing 
\begin{align}
    \nabla^k f(x_0) \cdot \ell(x_0,x)^k
\end{align}
is well-defined and, by the symmetry of $l(x_0,x)^k$, 
\begin{align}
   \nabla^k f(x_0)\cdot \ell(x_0,x)^k = \symd^k f(x_0) \ell(x_0,x)^k,
\end{align}
where $\symd^k$ is the symmetrization of $\nabla^k$ as before. An interesting  consequence of such construction is the analogy of  the Taylor's expansion formula on manifolds (cf.\cite[Prop. 2.2]{MR560744}):
\begin{align}
        f(x) \backsim \sum_{j=0}^\infty \frac{1}{j!}  \nabla^j f(x_0) \cdot \ell(x_0,x)^j, \,\,\, 
        f \in C^\infty(M). 
    \end{align}

%\begin{prop}
%    \label{prop:wpse_taylor_exp_on_M}
%    For each point $x_0 \in M$ and each integer $N$ the function 
%    \begin{align}
%        f(x) - \sum_{j=0}^N \frac{1}{j!}  \nabla^j f(x_0) \cdot \ell(x_0,x)^j 
%    \end{align}
%    vanished to order $N+1$ at $x_0$. 
%\end{prop}

\begin{defn}
    \label{defn:wpse-symbol-map}
    Let $M$ be a smooth manifold with a linear connection $\nabla$, and $\ell(\xi_x,y) \in C^\infty(T^*M \times M)$ be a phase function in definition \ref{defn:def_riegeo_defn-linearfun} with respect to $\nabla$. Denote by $\psi_\Delta \in C^\infty(M\times M)$ a cut-off function such that $\psi_\Delta = 1$ is equal to $1$ on a small neighborhood of the diagonal and also $\supp \psi_\Delta$ is still closed to the diagonal so that $\psi_\Delta(x,y) \neq 0$ implies $d_y
    \ell(\xi_x, y) \neq 0$ for all $\xi_x \neq 0$. 
    For any pseudo differential operator $P:C^\infty(M)\rightarrow C^\infty(M)$, the symbol $\sigma(P)$ of $P$ is a smooth function on $T^*M$:
\begin{align}
    \sigma(P)(\xi_x)  = P \psi_\Delta(x,y) e^{i \ell(\xi_x,y)} \Big|_{y=x},
    \label{eq:wpse_symmap}
\end{align}
where the operator $P$ acts on the $y$ variable. 
\end{defn}

\begin{rem}
	Up to smoothing operators, the symbol map $\sigma$ is independent of the choice of the cut-off function $\psi_\Delta$, and the choice of the connection $\nabla$ and the phase function $\ell(\xi_x,y)$.  
\end{rem}
 
On the other hand, pseudo differential operators are obtained by quantizing symbols. 
\begin{defn} \label{defn:wpse-diff-quantization}
% Recall that $\Psi(M) = \Psi^\infty(M)/\Psi^{-\infty}(M)$ is the space of classical pseudo differential operators on $M$, also for each $d \in \Z$, $\widetilde \Psi^d(M) =\Psi^d(M)/ \Psi^{-\infty}(M)$. 
 Let $\psi_\Delta$ be the cut-off function used in definition \ref{defn:wpse-symbol-map}.
 For any $f \in C^\infty(M)$, put $y = \exp_x Y$, the  quantization map $\op{Op}: S\Sigma^d(M) \rightarrow  \Psi^d(M)$, $d\in\Z^n$ is defined as follows:
\begin{align}
    (\op{Op}(p)f)(x) =\frac{1}{(2\pi)^m} \int_{T^*_xM}\int_{T_xM} e^{-i\abrac{\xi_x, Y}} p(\xi_x) \psi_\Delta(x,y) f(\exp_x Y) dY d\xi,
    \label{eq:wpse-diff-quantization-defn}
\end{align}
where $m$ is the dimension of the manifold $M$, $\abrac{\xi_x,Y}$ is the canonical pairing: $T_x^*M \times T_x M \rightarrow \R$, and $dY$, $d\xi$ denote the  densities that are dual to each other. Different choices of the cut-off function   $\psi_\Delta$ give rise the same quantization map modular smoothing operators. 
% The Lebesgue measure $dY$ and $d\xi$ are chosen such that the Fourier inversion holds:
% \begin{align*}
%     \phi(0) = \int_{T^*_xM}\int_{T_xM} e^{-i\abrac{\xi_x, Y}} dY d\xi
% \end{align*}
% where $\phi$ is a smooth functions on $T_xM$ supported in a small neighborhood of the origin. 
\end{defn}

It is well-known that $\sigma$ and $\op{Op}$ are inverse to each other upto smoothing operators.

As an exmaple, we compute the symbol of the scalar Laplacian $\Delta$.
\begin{lem}
    \label{lem:wpse-diff-sym-lap}
    The symbol of the scalar Laplacian $\Delta: C^\infty(M) \rightarrow C^\infty(M)$ is equal to $\sigma(\Delta) = \abs\xi^2$, where $\abs\xi^2$ is the squared length function on $T^*M$.
\end{lem}

\begin{proof}
    In  local coordinates, 
 	\begin{align}
 	\Delta f = - (\nabla^2 f)_{ij}g^{ij}.
 	\label{eq:mod-cur-laplacian}
 	\end{align}
 To compute the symbol, we can ignore the cut-off function $\psi_\Delta$ because $\Delta$ is a differential operator:
 	\begin{align*}
 	\nabla^2 e^{i\ell(\xi_x,y)} = e^{i\ell(\xi_x,y)} \brac{i \nabla^2 \ell(\xi_x,y)
 		- \nabla \ell(\xi_x,y)\nabla \ell(\xi_x,y)}.
 	\end{align*}
    Note that $\nabla$ is torsion-free, thus $\nabla^2 $ is equal to its symmetrization $\symd^2$. Therefore by the definition of $\ell$.  $\nabla^2 \ell=0$ and  $\nabla_j \ell(\xi_x,y) = \xi_j$.  Pair $\nabla^2 e^{i\ell(\xi_x,y)}$ with the metric on $T^*M$: $g^{-1} = g^{ij}$, we obtain:
 	 	\begin{align}
 	\sigma(\Delta)(\xi_x) = \xi_i \xi_j g^{ij} = \abs\xi^2.
 	\label{eq:mod-cur-sym-delta}
 	\end{align}
 	where $(g^{ij})$ is the metric tensor on $T^*M$. 
\end{proof}

 Given two pseudo differential operators $P$ and $Q$ with symbols $p$ and $q$ respectively, the symbol of the composition $PQ$ is given by an asymptotic product
$\sum_j a_j(p,q)$, where the $a_j(\cdot,\cdot)$ are bi-differential operators. Similar to the case of  differential operators in $C^\infty(M)$ described \eqref{eq:defn:wpse-diff-dop-coordfree}, the bi-differential operators $a_j(p,q)$ on $C^\infty(T^*M) \times C^\infty(T^*M)$ are obtained as the contraction between mixed derivatives of $p$ and $q$ (cf. \eqref{defn:def_riegeo-horizontal-D}) and   tensor fields of suitable ranks so that the contraction produces a smooth function.  

\begin{prop}
    \label{prop:wpse-asym-sybproduct}
    
    Keep the notations as above, for any two pseudo differential operators $P$ ans $Q$ with $p = \sigma(P)$ and $q = \sigma(Q)$, then $\sigma(PQ) =p \star q $ has the following asymptotic form: 
\begin{align}
        p \star q   \backsim  \sum_{j=0}^\infty a_j(p,q),
        \label{eq:wpse-diff-startheta-asym-sybproduct}
    \end{align}
    where  $a_j(\cdot,\cdot)$ are bi-differential operators reducing the total degree by $j$, namely,  for any $d,d'\in\Z$, 
    \begin{align*}
    a_j(\cdot,\cdot): S\Sigma^d(M) \times S\Sigma^{d'}(M) \rightarrow S\Sigma^{d+d'-j}(M).
    \end{align*}
     Here $\backsim$ means that, if we truncate the sum to the first $K$ terms, then the remainder is of order $d+d' -K$, where $d$ and $d'$ are the order of $P$ and $Q$ respectively. \par
   %Taking the symbols $\rho_{(\alpha)_k,(\beta)_k}$ in equation \eqref{eq:wpse-diff-rho-alpha-beta-defn} into account, 
    The bi-differential operators $a_j$  are given by:
%  \begin{align}
%          a_j(p,q) = \sum_{\substack{k\ge 0,  (\beta)_k\ge 2\\
%              j = k-l-\abs{(\alpha)_k}-\abs{(\beta)_k}
%          }}
%          \frac{i^{j}}{k!(\alpha)_k! (\beta)_k! l!}
%         (D^{l + \abs{(\alpha)_k}} p ) (\symd^{l} D^{\abs{(\beta)_k}}q ) \cdot \rho_{(\alpha)_k,(\beta)_k}
%     \end{align}
    \begin{align}
        \begin{split}
       a_j(p,q) =& \sum 
       \frac{i^{-j}}{k!\alpha_0!\alpha_1! \cdots \alpha_k! \beta_1!\cdots \beta_k!} \\
        &(D^{\alpha_0+\sum_1^k \alpha_k}p) ( D^{\sum_1^k \beta_s} \nabla^{\alpha_0}q  )
        (\nabla^{\alpha_1 + \beta_1}\ell)\cdots(\nabla^{\alpha_k+\beta_k}\ell),     
        \end{split}
        \label{eq:wpse-diff-aj-defn} 
    \end{align}
    where the summation is taken over:
    \begin{align*}
        j = -(k - \alpha_0 -\sum_1^k (\alpha_s+\beta_s)) \ge 0, \,\,\, \alpha_0 \ge 0, \,\,\, \alpha_1, \dots \alpha_k \ge 1, \,\,\,
        \beta_1, \dots , \beta_k \ge 2. 
    \end{align*}
The operation between all factors in \eqref{eq:wpse-diff-aj-defn} is mixed type of tensor product and contraction bewteen tensor fields. The contraction occurs between the contravariant and covariant tensors with the same index, thus eventually yields a smooth function.
    
\end{prop}

\begin{rem} \mbox{}
	\begin{enumerate}[1)]
        \item The vertical and the horizontal differentials $D$ and $\nabla$ are defined in section \ref{subsec:lift-tensor-cal}.
		\item In \eqref{eq:wpse-diff-aj-defn}, the order of the multiplication of $p$, $q$ and  $\ell$ is arranged in such a way that it is works for pseudo differential operators on vector bundles. 
		\item The constrains $\alpha_1, \dots \alpha_k \ge 1$ and $\beta_1, \dots , \beta_k \ge 2$ gives great simplification for the first a few $a_j$'s.  
%             Notice that the vertical differential $D^k q$ always produces a symmetric $k$-tensor, therefore  one can replace each 
%             $\nabla^{\alpha_s+\beta_s}\ell$ by its symmetrization over $\alpha$ and $\beta$: $\symd^{\alpha_s} \symd^{\beta_s}\ell$
%             Since $\symd^j l(\xi_x,y)|_{y=x} =0 $ for $j\ge 1$,  the $ \rho_{(\alpha)_k,(\beta)_k}$ is equal to zero if one of the $\alpha$ components is zero. Hence the summation in \eqref{eq:wpse-diff-aj-defn} is over  $(\alpha)_k \in \Z^k_+$.
		\item The first a few $a_j$'s are listed below: 
		\begin{align}
		a_0(p,q) & = pq, \label{eq:wpse-a0}\\
		a_1(p,q) & = -i D p \nabla q, \label{eq:wpse-a1}\\
		a_2(p,q) & = -\frac 12 D^2 p \nabla^2 q - \frac 12  (D p) (\nabla^2 q) (\nabla^3 \ell).
		\label{eq:wpse-a2}
		\end{align}
		 
	\end{enumerate}

\end{rem}

\subsection{ the Schwartz Kernels and the trace formula}

% The Quantization map $\op{Op}$ is not unique in general, for example, it can be defined in local coordinates via Fourier integrals. However, we need one that is equivariant with respect to the action of diffeomorphisms. 
% 
The last piece we need in the symbol calculus is the trace formula of an operator from its symbol, provided the operator is of trace-class. Since the symbol only defines an operator upto smoothing operators. We measure the error by dilation parameter $t \in [1,\infty)$. 
% 
%  Let us state the result that links the Schwartz kernel  of a pseudo differential operator and its symbol. We refer to   \cite[Theorem 5.7]{MR538027} for proofs.   Strictly speaking,  the whole construction of the  symbol calculus works only between classical pseudo differential operators and classical symbols (``classical'' simply means upto smoothing operators or symbols) due to certain choices we have made, such as the connection, the phase function, the cut-off function, etc.
% Therefore there is no surprise that certain error appears when we try to recover the  Schwartz kernels from the symbols. The error is controlled by a new parameter $t \in (0,\infty)$, viewed as a dilation of symbols. 
  
  \begin{defn}\label{defn:wpse-diff-dilation-ops}
  	A dilation  of a pseudo differential operator $P$ is a family of pseudo differential operators $P_t$ with $t \in [1,\infty)$ given by dilation of the symbols functions, in local coordinate:
  	\begin{align}
  	p_t(x,\xi) = p(x,\xi/t).
  	\label{eq:wpse-diff-dilation-ops}
  	\end{align}
  \end{defn}

Let $(M,g)$ be a $m$ dimensional closed Riemannian manifold with the metric tensor $g$. The canonical $2m$-form on the cotangent bundle $T^*M$ is given in local coordinates $(x,\xi)$ by
\begin{align}
\label{eq:wpse-diff-canonical-twom-form}
	\Omega = dx_1 \wedge \cdots \wedge dx_m \wedge d\xi_1 \wedge \cdots \wedge d\xi_m,
\end{align}
while $\Omega = dg d\xi_{x,g^{-1}}$ with the volume form $dg = (\det g) dx_1 \wedge \cdots dx_m$ and 
\begin{align}
\label{eq:wpse-diff-measure-on-T^*_xM}
	d\xi_{x,g^{-1}} = (\det g)^{-1}d\xi_1 \wedge \cdots \wedge d\xi_m
\end{align}
defines a measure on the fiber $T^*_xM$. 

\begin{prop}\label{prop:wpse-diff-sym-to-kernel-dia}
Keep the notations as above, let $P$ be pseudo differential operator $P$ whose order is less than $m$, the dimension of the manifolds such that the Schwartz kernel
function $k_P(x,y)$ exists, then the Schwartz kernel of the dilation  $P_t$ $($see definition $\ref{defn:wpse-diff-dilation-ops}$$)$  on the diagonal is given by 
    \begin{align}
        k_{P_t}(x,x) = \frac{t^m}{(2\pi)^m} \int_{T^*_xM} \sigma(P) d\xi_{g^{-1}} + O(t^{-N}), \,\,\, \forall N \in\N.
        \label{eq:wpse-diff-sym-to-kernel-dia}
    \end{align}
 In particular, we obtain the trace formula for $P_t$:
 \begin{align}
     \Tr(P_t) = \frac{t^m}{(2\pi)^m} \int_{T^*M} \sigma(P) \Omega + O(t^{-N}), \,\,\, \forall N \in\N,
 \end{align}   
 where $\Omega = dg d\xi_{x,g^{-1}}$ is the canonical $2m$-form defined in \eqref{eq:wpse-diff-canonical-twom-form}. 
\end{prop}

 See \cite[Theorem 5.7]{MR538027} for the proof.

%% file: wpse_theta.tex
% !TEX root =  main.tex

\subsection{Deformation of the symbol calculus}
\label{sec:def-wid'scal}
Now let us take the torus action into account. The symbol calculus depends on the choice of a linear connection, therefore if we assume that the torus acts as affine transformations $\T^n\subset \op{Affine}(M)$ so that the the connection is preserved as explain in section \ref{sec:def-of-rie-geo}, the all the constructions, such as the symbol map and the quantization map, are 
$\T^n$-equivariant.  The Fr\'echet topologies on symbols and pseudo differential operators are constructed in such a way that both the symbol map and the quantization map are continuous. As a consequence, for any pseudo differential operator $P = \sum_{r\in\Z^n} P_r$ and its symbol $p = \sum_{r\in\Z^n} p_r$ with their isotypical decomposition, we have
\begin{align*}
    \sigma(P_r) \backsim \sigma(P)_r \backsim p_r, \,\,\, \op{Op}(p_r) \backsim \op{Op}(p)_r \backsim P_r, 
\,\,\, \sigma(P) \backsim \sum_{r\in\Z^n} p_r, \,\,\, \op{Op}(q) \backsim \sum_{r\in\Z^n} P_r.
\end{align*}
Given two deformed pseudo differential operators $\pi^\Theta(P)$ and $\pi^\Theta(Q)$ with symbols $p$ and $q$ respectively, our goal is to find an deformed the asymptotic symbol product $\star_\Theta$ such that the error
\begin{align*}
    \pi^\Theta(P)\pi^\Theta(Q) - \pi^\Theta(\op{Op}(p \star_\Theta q))
\end{align*}
belongs to $\Psi^{-\infty}(M)$. Since $\pi^\Theta(P)\pi^\Theta(Q) = \pi^\Theta(P\times_\Theta Q)$, $p \star_\Theta q$ should be $\sigma(P\times_\Theta Q)$ upto a smoothing operator. We formally compute:
\begin{align*}
     \sigma(P \times_\Theta Q) &= \sum_{\mu,\nu\in\Z^n} \chi_\Theta(\mu,\nu) \sigma(P_\mu Q_\nu) 
 \backsim  \sum_{\mu,\nu\in\Z^n} \chi_\Theta(\mu,\nu) p_\mu \star q_\mu\\
 &\backsim  \sum_{\mu,\nu\in\Z^n}  \sum_{j=0}^\infty   \chi_\Theta(\mu,\nu)a_j(p_\mu,q_\nu)
 \backsim  \sum_{j=0}^\infty  \sum_{\mu,\nu\in\Z^n} \chi_\Theta(\mu,\nu)a_j(p_\mu,q_\nu)\\
 &=  \sum_{j=0}^\infty  a_j(p,q)_\Theta.
 \end{align*}
The summation $ \sum_{j=0}^\infty$ simply means if we truncate the sum to first $N$ terms, the remainder belongs to symbol of order $-J_N$, where $J_N \rightarrow \infty$ as $N \rightarrow \infty$.  Therefore  $ \sum_{j=0}^\infty$ behaves like a finite sum. \par

It remains to compute $a_j(p,q)_\Theta$. Recall from \eqref{eq:wpse-diff-aj-defn}, each $a_j(p,q)_\Theta$ is obtained by contraction and tensor product between tensor fields. Both operations are equivariant with respect to the torus action. Take $a_1$, for instance,
\begin{align*}
    a_1(p,q)_\Theta & = \sum_{\mu,\nu\in\Z^n} \chi_\Theta(\mu,\nu)a_1(p_\mu,q_\nu)
    = \sum_{\mu,\nu\in\Z^n} \chi_\Theta(\mu,\nu) (-i) (D p_r) (\nabla q_l) \\
    & \sum_{\mu,\nu\in\Z^n} \chi_\Theta(\mu,\nu) (-i) (Dp)_r (\nabla q)_l = (-i) (Dp) \cdot_\Theta (\nabla q).
\end{align*}
In general, for all $j\ge 0$, one can quickly verify that  $a_j(p,q)_\Theta$ is of the same form as $a_j(p,q)$ in  \eqref{eq:wpse-diff-aj-defn}, in which the tensor product and the contraction are replaced by the deformed version $\otimes_\Theta$ and $\cdot_\Theta$ respectively. 

We summarize the discussion above in the proposition:
\begin{prop}
        \label{prop:wpse-theta-startheta}
        Keep the notations in proposition $\ref{prop:wpse-asym-sybproduct}$. Let $M$ be a closed  manifold with a $n$-torus action  so that the previous deformation machinery applies.   
     Let  $\pi^\Theta(P)$ and $\pi^\Theta(Q)$ be the deformation of pseudo differential operators. Denote $p = \sigma(P)$ and $q = \sigma(Q)$, then 
     \begin{align*}
         \pi^\Theta(P)\pi^\Theta(Q) \backsim \pi^\Theta(\op{Op}( p\star_\Theta q),
     \end{align*}
     with 
     \begin{align*}
         p\star_\Theta q \backsim \sum_{j=0}^\infty a_j(p,q)_\Theta,
     \end{align*}
     where the bi-differential operators $a_j(\cdot,\cdot)$ are the deformation of $a_j(\cdot,\cdot)$ in proposition $\ref{prop:wpse-asym-sybproduct}$:
     \begin{align*}
      a_j(\cdot,\cdot)_\Theta=   \sum_{\mu,\nu\in\Z^n}\chi_\Theta(\mu,\nu) a_j(p_\mu,q_\nu),\,\,\,\, \chi_\Theta(\mu,\nu) = e^{i \abrac{\mu,\Theta\nu}}.
     \end{align*}
     Precisely, $ a_j(p,q)_\Theta$ can be obtained  from  $ a_j(p,q)$ $($in equation \eqref{eq:wpse-diff-aj-defn}$)$ by replacing the pointwise tensor product and contraction by the deformed version $\otimes$ and $\times_\Theta$ $($cf. sec.$\ref{subsec:deformation-funtens}$$)$: 
     \begin{align}
         \begin{split}
       a_j(p,q)_\Theta =& \sum 
       \frac{i^{-j}}{k!\alpha_0!\alpha_1! \cdots \alpha_k! \beta_1!\cdots \beta_k!} \\
        &(D^{\alpha_0+\sum_1^k \alpha_k}p)\cdot_\Theta ( D^{\sum_1^k \beta_s} \nabla^{\alpha_0}q  ) 
        (\nabla^{\alpha_1 + \beta_1}\ell) \cdots(\nabla^{\alpha_k+\beta_k}\ell),     
        \end{split}
        \label{eq:wpse-theta-aj-theta}
     \end{align}
     where the summation is  over:
    \begin{align*}
        j = -(k - \alpha_0 -\sum_1^k (\alpha_s+\beta_s)) \ge 0, \,\,\, \alpha_0 \ge 0, \,\,\, \alpha_1, \dots \alpha_k \ge 1, \,\,\,
        \beta_1, \dots , \beta_k \ge 2. 
    \end{align*}
%      \begin{align}
%           a_j(p,q)_\Theta = \sum_{\substack{k\ge 0,  (\beta)_k\ge 2\\
%              j = k-l-\abs{(\alpha)_k}-\abs{(\beta)_k}
%          }}
%          \frac{i^{j}}{k!(\alpha)_k! (\beta)_k! l!}
%         \rho_{(\alpha)_k,(\beta)_k} \otimes_\Theta \symd^{l} D^{\abs{(\beta)_k}}q \otimes_\Theta D^lD^{\abs{(\alpha)_k}} p,
%          \label{eq:wpse-theta-aj-theta}
%      \end{align}
%      here we use  $\otimes$ to denote both deformed tensor product and contraction between mixed tensors.  Also as the same as in the commutative situation, for each $j\ge 0$, $a_j(p,q)_\Theta$ reduce the total degree of $p$ and $q$ by $j$.
 \end{prop}
 \begin{rem}
     Notice that  the phase function $\ell$ is $\T^n$-invariant, so are the covariant derivatives $\nabla^k \ell$ ($k=0,1,2,\dots$), as a result, when  $\nabla^k \ell$ are involved, the deformed and undeformed tensor product and contraction make no difference. 
 \end{rem}

\section{Heat kernel asymptotic}
\label{sec:heat-ker-asym}

\subsection{Resolvent approximation}
\label{subsec:resolvent-app}
The deformation symbol calculus can be quickly upgraded  to the parametric version. Let $\pi^\Theta(P)$ be the deformation of a second order differential operator $P$, whose symbol equals $\sum_0^2 p_j$ with $p_j$ of degree $j$ ($j=0,1,2$). Let $\lambda$ be the resolvent parameter that lies in some cone region in the complex plane, denote $p_2(\lambda) = p_2 - \lambda$, and $p_j(\lambda) = p_j$ for $j=0,1$. We would like to solve the resolvent equation
\begin{align*}
  (p_{2}(\lambda) + p_{1} + p_{0} )\star_{\Theta} (b_{0}(\lambda) +
  b_{1}(\lambda) + \cdots) \backsim 1,
\end{align*}
where $b_\kappa(\lambda)$ are parametric symbols of order $-2-\kappa$, $\kappa=0,1,2,\dots$. With the symbol calculus in hand, if the inverse of $p_2(\lambda)$ exists, which will be  $b_0(\lambda)$, that is $a_0(b_0(\lambda),p_2(\lambda))_\Theta \backsim a_0(p_2(\lambda),b_0(\lambda))_\Theta \backsim 1$,  then $b_\kappa$ can be constructed inductively as follows:
\begin{align}\label{eq:wpse-theta-inductive-formula-for-rev-approximation-2}
      b_{\kappa}(\lambda) = -b_0(\lambda) \times_\Theta \brac{\sum_{\substack{-\kappa = \mu-2-\nu-j \\
     \nu < \kappa}}{}a_{j}\brac{ p_{\mu}(\lambda) , b_{\nu}(\lambda)  }_{\Theta} }
\end{align}
or
\begin{align}\label{eq:wpse-theta-inductive-formula-for-rev-approximation}
  b_{\kappa}(\lambda) = -\brac{\sum_{\substack{-\kappa = \mu-2-\nu-j \\
     \nu < \kappa}}{}a_{j}\brac{b_{\nu}(\lambda), p_{\mu}(\lambda) }_{\Theta} }
\times_{\Theta} b_{0}(\lambda),
\end{align}
 where $\mu = 0,1,2$ and $\nu = 0,1,2,\cdots$.\par
 We will need $b_1$ and $b_2$ later:
% For example, we write down $b_1$ and $b_2$ explicitly, which will be needed later: 
\begin{align}
    b_1 & =  ((-i) Db_0 \times_\Theta \nabla p_2 + b_0 \times_\Theta p_1) \times_\Theta (-b_0)
    \label{eq:wpse-theta-b1}
    \\
    b_2  & =   \left(\phantom{\frac12} b_0 \times_\Theta p_0 + p_1 \times_\Theta b_1(\lambda) + (-i)Db_0 \times_\Theta \nabla p_1 + (-i)Db_1 \times_\Theta \nabla p_2
   \nonumber \right.  \\ 
  & - \left. \frac12 D^2 b_0 \times_\Theta \nabla^2 p_2 
- \frac12 (\nabla^3 \ell)Db_0 \times_\Theta D^2 p_2 \right)
\times_\Theta (-b_0).
    \label{eq:wpse-theta-b2}
\end{align}

\subsection{Perturbation of the scalar Laplacian $\Delta$ via a Weyl factor}
\label{subsec:perturbation-of-delta}
%Before talking about curvatures, one needs a notion of metrics. 
From now on, we shall focus a specific family of deformed elliptic geometric differential operators which represents a family of noncommutative metrics. Ellipticity simply means we know how to inverse its leading symbol in the deformed symbol algebra. \par

 Parallel to the work on noncommutative tori (cf.\cite{MR2907006, MR3194491, MR3148618, MR3369894}), the new family of noncommutative metrics is obtained by a noncommutative conformal change of the Riemannian metric  that we started with before the deformation. In terms of spectral point of view, the new metrics we are looking at in this paper  are given by a perturbation of the scalar Laplacian Weyl factor in the noncommutative coordinate algebra. 
%The resulting operators are still elliptic in our noncommutative setting, namely, the leading symbols are
%invertible (up to a smoothing symbol) in the deformed algebra $S\Sigma(M_\Theta, \Lambda)$. 

\begin{defn} \label{defn:wpse_diff-defn-Weylfactor}
    Let $M$ be a 
    toric  Riemannian manifold as in definition $(\ref{defn:def-rie-eg-toricmflds})$ and $C^\infty(M_\Theta) = \pi^\Theta\brac{C^\infty(M)}\subset B(\mathcal H)$ is the algebra of smooth functions on its deformation $M_\Theta$ with respect to a  skew-symmetric $n\times n$ matrix $\Theta$. Here, $\mathcal H$ is the Hilbert space of $L^2$-functions on $M$ as before. Denote by $C(M_\Theta)$ the $C^*$-algebra of the operator norm completion of $C^\infty(M_\Theta)$. A \emph{Weyl factor} $\pi^\Theta(k)$ is an  element in $C^\infty(M_\Theta) \subset C(M_\Theta)$ (that is $k \in C^\infty(M)$) which is invertible and positive.
\end{defn}

The Weyl factor $k$ is always of the form $k = e^h$ for some self-adjoint operator $h$. To be precise, let $h \in C^\infty(M)$ real-valued smooth function so that is it self-adjoint as an operator. Follows from lemma \ref{lem:def-ops-deform-oper-pi-theta-adjoint},  $\pi^\Theta(h)$ is also self-adjoint. Let $k =\exp_\Theta(h)$ so that $\pi^\Theta(k) = e^{\pi^\Theta(h)}$. In the rest of the computation, we shall drop $\pi^\Theta(\cdot)$ and apply the deformed calculus on to smooth
function $k$ and $h$ directly.

\begin{defn}
    \label{defn:wpse_theta-pert-laplacian}
 Fixed a Weyl factor $\pi^\Theta(k)$ with $k \in C^\infty(M)$, a deformed differential operator  $\pi^\Theta(P_k)$ is called of  perturbed Laplacian type if its  spectrum is contained in $[0,\infty)$, and $P_k \in \Psi^2(M)$ is a differential operator whose symbol is of the form $\sigma(P_k) = p_2 + p_1 + p_0$,
when $j=0,1$, $p_j(\xi_x)\in S\Sigma^j(M)$ is polynomial in $\xi$ of degree $j$. For $j=2$, we require the leading symbol is of the form:
\begin{align*}
    p_2(\xi_x) =k \abs\xi^2,
\end{align*}
where $\abs\xi^2 = \abrac{\xi,\xi}_{g^{-1}}$ is the squared length function on $T^{*}M$ with respect to the given Riemannian metric $g^{-1}$. 
\end{defn}
%Originally, $p_2(\xi_x)$ is defined to be $k \times_\Theta \abs\xi^2$, since 
Notice that the function $\abs\xi^2$ is $\T^{n}$-invariant,  we have $k \times_\Theta \abs\xi^2=k \abs\xi^2$. Hence $p_2(\xi_x)$ is the symbol of the operator $\pi^\Theta(k) \Delta$. When the complex parameter $\lambda$ is off the positive real line,  $k \abs\xi^2 - \lambda$ is invertible in $S\Sigma(M_\Theta,\Lambda)$, the inverse is denoted by $b_{0}(\xi,\lambda) = (k \abs\xi^2 - \lambda)^{-1}$. 
For the smoothness of the exposition, we move the detail construction of $b_{0}(\xi,\lambda)$ to  appendix \ref{App:inverseof_p2}. 

\subsection{Some estimates}
Parametric pseudo differential operators is more that just a map $\lambda \mapsto P(\lambda)$. For instance, we require that our families of operators is holomorphic in $\lambda$, moreover, when parametric symbols are considered, differentiating its symbol in the parameter reduces the order of the operator:
    \begin{align}
        \abs{D^\alpha_x D^\beta_\xi D^j_\lambda p(x,\xi,\lambda)
        } \le C_{\alpha,\beta,j}\brac{1+ \abs\xi^2 + \abs\lambda
        }^{(d - \abs\beta-2j)/2}.
        \label{eq:wpse_theta-syml-class-para}
    \end{align}

As a result, the operator norm has some control in the resolvent parameter: 

% Let us follow  the \cite{gilkey1995invariance} to outline the existence of the heat kernel asymptotic.
% The operator norm estimate in terms of $\lambda$ is essential for our discussion. Recall \cite[Lemma 1.7.1 (b)]{gilkey1995invariance}:

\begin{lem}%[Lemma 1.7.1 (b) in \cite{gilkey1995invariance}]
    \label{prop:wpse-theta-opnorm-symnorm}
    Given any positive integer $k$, one can find another integer $K > 0$ such that for any $\pi^\Theta(Q(\lambda)) \in \Psi^{-K}(M_\Theta,\Lambda)$, the operator norm of $\pi^\Theta(Q(\lambda)) \in B(\mathcal H^{-k},\mathcal H^{k})$ has the estimate:
    \begin{align}
        \norm{\pi^\Theta(Q(\lambda))}_{-j,j} \le C (1 + \abs\lambda)^{-j}.
        \label{eq:wpse-theta-opnorm-symnorm}
    \end{align}
\end{lem}
\begin{proof}
    This fact is well-known, see Lemma 1.7.1 (b) in \cite{gilkey1995invariance} for instance,   for parametric pseudo differential operators on closed smooth manifolds. Namely, we have $Q(\lambda) \in \Psi^{-K}$   implies $\norm{Q(\lambda)}_{-j,j} \backsim O(\abs\lambda^{-j})$ as  $\abs\lambda \rightarrow \infty$. Now we take the torus action into account, notice that for any  partial derivative in $t\in\T^n \mapsto \op{Ad}_t(Q(\lambda))$, the resulting operator
    $\partial_t^\mu\op{Ad}_t(Q(\lambda))$ with a multi-index $\mu$,  still lies in $\Psi^{-K}$ if $Q(\lambda)$ does. In particular,  
\begin{align*}
    \norm{ \partial_t^\mu\op{Ad}_t(Q(\lambda))}_{-j,j} \backsim O(\abs\lambda^{-j})\,\,\,
    \text{as $\abs\lambda \rightarrow \infty$}, \,\,\, \forall t\in\T^n.
\end{align*}
From the standard Fourier theory on torus, 
\begin{align*}
    \norm{ \pi^\Theta(Q)}_{-j,j} \le C \sup_{t\in\T^n} \norm{ \partial_t^\mu\op{Ad}_t(Q(\lambda))}_{-j,j},
\end{align*}
provided that the total order of the partial derivative $\mu$ is greater that the dimension of the torus. The proof is complete. 
\end{proof}

Let us consider a smooth family $k_s = \exp_\Theta(sh)$ of Weyl factors, where $s\in[0,1]$ and $h$ is a real valued function on $M$ so that it is a self-adjoint operator. Recall that we have dropped the deformation map $\pi^\Theta(\cdot)$ when dealing with $k$ and its logarithm $h$,  the calculation is taken in the deformed version. So we will simply write $k_s = e^{sh}$ in the rest of the computation. \par
We fix a $h$  and consider the family of Laplacian operators perturbed by $k_s = e^{sh}$:
\begin{align}
    \label{eq:wpse-theta-def_Ps}
    P_s = k_s \Delta + \text{lower order terms},  
\end{align}
where $\Delta$ is the scalar Laplacian operator. In our application, the lower order terms involve the Weyl factor $k$ upto its second covariant derivative. In particular, the symbol of $P_s$ depends smoothly in $s$ and differentiating in $s$ do not increase the order of the symbol (or the operator).  \par

% Let $\pi^\Theta(P_k(s))$ be a smooth family of perturbed Laplacian with $s\in[0,\varepsilon)$, namely, for fixed $s$, $\pi^\Theta(P_k(s))$ is a  perturbed Laplacian as in definition $\ref{defn:wpse_theta-pert-laplacian}$. 
    
Denote the resolvent by $R(\lambda,s) = ( P_s - \lambda)^{-1}$. Base on the fact that the spectrum of $P_s$ is contained in $[0,\infty)$ for all $s\in[0,1]$, one can repeat the argument in \cite{gilkey1995invariance} lemma 1.6.6 to show that for any integer $j\ge 0$,
    \begin{align*}
        \norm{R(\lambda,s) }_{-j,j} \backsim O(\abs\lambda^{l(j)}),\,\,\, \text{as $\abs\lambda \rightarrow \infty$.}
    \end{align*}
    The power $l(j)$ is positive and uniform with respect to $s$. \par
    The resolvent approximation in section \ref{subsec:resolvent-app} gives us a sequence of deformed pseudo differential operators $\set{R_j(\lambda,s)}_{j=0}^\infty$ such that the difference $R - R_j$ is of order $-j-1$ for each $j$. With lemma  \ref{prop:wpse-theta-opnorm-symnorm},  arguing like in \cite{gilkey1995invariance} lemma 1.7.2, we conclude that for any integer $j\ge 0$,one can choose $l$ larger enough so that 
\begin{align}
\norm{
        R(\lambda,s) - R_l(\lambda,s)
    }_{-j,j} \backsim O(\abs\lambda^{-j}), \,\,\,
    \text{as $\abs\lambda \rightarrow \infty$.}
    \label{eq:wpse-opnorm-estimate-resolvent}
    \end{align}

We also need the variation in $s$:
\begin{align*}
    \frac{d}{ds} \brac{ R(\lambda,s) - R_l(\lambda,s) } &=
   R(\lambda,s)  \brac{I - (P_s - \lambda)R_l(\lambda,s)} \\
  &- R(\lambda,s) \frac{d P_s}{ds} R(\lambda,s) \brac{I - (P_s - \lambda)R_l(\lambda,s)}.
\end{align*}
When apply the operator norm $\norm{\cdot}_{-j,j}$ ($j\ge 0$) on both sides, the error term has the estimate:
\begin{align*}
    \norm{I - (P_s - \lambda)R_l(\lambda,s)}_{-j,j}    \backsim O(\abs\lambda^{-j}),
\end{align*}
provide $l$ large enough. As a result, 
\begin{align}
    \label{eq:wpse-opnorm-estimate-resolvent-ds}
   \norm{
    \frac{d}{ds} \brac{ R(\lambda,s) - R_l(\lambda,s) }
    }_{-j,j} \backsim O(\abs\lambda^{-j}), \,\,\,
    \text{as $\abs\lambda \rightarrow \infty$.}
\end{align}
For the complete argument, see \cite{branson1986conformal} equations (3.12) to (3.14). 

\subsection{Heat kernels and the variation}
% Now we are ready to establish the heat kernel expansion of the perturbed Laplacian $\pi^\Theta(P_k)$. By the assumption the spectrum of $\pi^\Theta(P_k)$ is contained in $[0,\infty)$, thus the heat operator 
% Keep the notations as in the previous section.

Notice that the spectrum of $P_s$ is contained in $[0,\infty)$, therefore  the heat operator 
    $e^{-tP_s}$ can be defined by a contour integral:
    \begin{align}
        \label{eq:wpse-theta-heat-op}
        e^{-t P_s } = \frac{1}{2\pi i} \int_C e^{-t\lambda}(P_s - \lambda)^{-1} d\lambda,  
    \end{align}
    where $C$ is a curve in the complex plane that circle around $[0,\infty)$ in such a way that 
        \begin{align*}
            e^{-ts} = \frac{1}{2\pi i} \int_C e^{-t\lambda}(s - \lambda)^{-1}ds,\,\,\, \forall s\ge 0.
        \end{align*}

If one replaces the resolvent $(P_s - \lambda)^{-1}$ by its $l$-th approximation $R_l(\lambda,s)$ in the integral above, denote $E_l(t,s) = \int_C e^{-t\lambda} R_l(\lambda,s)d\lambda$, then the estimate \eqref{eq:wpse-opnorm-estimate-resolvent} gives rise to 
\begin{align} \label{eq:wpse-theta-opnorm-error-est}
    \norm{e^{-t P_s } - E_l(t,s)}_{-j,j}   \backsim o(t^{j}) \,\,\,
    \text{as $t \rightarrow 0$,}
\end{align}
provided $l$ large enough. In order to study the trace, it is easier to  consider the associated pseudo differential operators before the deformation, for which the Schwartz kernel functions make sense. Write $e^{-t P_s }$ and $E_l(t,s)$  as
\begin{align*}
    e^{-t P_s } = \pi^\Theta(\tilde E(t,s)), \,\,\, E_l(t,s) = \pi^\Theta\brac{
    \tilde E_l(t,s)
    },
\end{align*}
where $\tilde E$ and $\tilde E_l$ are smoothing operators (means of order $-\infty$) on $M$. Thanks to the continuity of $\pi^\Theta$, \eqref{eq:wpse-theta-opnorm-error-est} implies that: $\norm{ \tilde E_s - \tilde E_l(\lambda,s)}_{-j,j}   \backsim o(t^{j})$, in terms of their Schwartz kernel functions (see \cite{gilkey1995invariance}, lemma 1.7.3): we have 
\begin{align*}
    \abs{
        H_{\tilde E_s}(x,y,t) - H_{\tilde E_l(t,s)}(x,y,t)
    }_{L^2_j} \backsim o(t^{j}) \,\,\,
    \text{as $t \rightarrow 0$.}
\end{align*}
In particular, for the $L^2$-trace,
\begin{align*}
    \Tr \brac{
   \tilde E_s(t) -\tilde E_l(\lambda,s)
    } \backsim o(t^{j}), \,\,\,
    \text{as $t \rightarrow 0$.}
\end{align*}
%Combine this with the fact that the deformation map preserves the trace (lemma \ref{lem:def-ops-inv-of-trace}),
Since the trace is preserved under the deformation process (lemma \ref{lem:def-ops-inv-of-trace}),
we have proved that $\Tr e^{-tP_s}$ has an asymptotic expansion as $t\rightarrow 0$ given by $\Tr E_l(t,s)$, $l=0,1,2,\dots$. To be precise: for given integer $j\ge0$, one can find $l$ large enough such that
\begin{align}
    \Tr \brac{
    e^{-t P_s } - E_l(t,s)
    } \backsim o(t^{j}) , \,\,\,
    \text{as $t \rightarrow 0$.}
    \label{eq:wpse-trace-asym}
\end{align}

The variation in $s$ can be performed in a similar fashion. Start with  \eqref{eq:wpse-opnorm-estimate-resolvent-ds}, we have the operator norm estimate:
    \begin{align*}
    \norm{
        \frac{d}{ds} \brac{
            e^{-t P_s } - E_l(t,s)}
    }_{-j,j}   \backsim o(t^{j}) \,\,\,
    \text{as $t \rightarrow 0$,}
\end{align*}
which leads to the Schwartz kernel estimate:
\begin{align*}
    \abs{ \frac{d}{ds} \brac{
        H_{\tilde E_s}(x,y,t) - H_{\tilde E_l(t,s)}(x,y,t) }
    }_{L^2_j} \backsim o(t^{j}) \,\,\,
    \text{as $t \rightarrow 0$.}  
\end{align*}
Finally we can conclude that the heat kernel asymptotic can be differentiated term by term in the parameter $s$:
\begin{align}
    \Tr\brac{
        \frac{d}{ds} \brac{ 
       e^{-t P_s } - E_l(t,s) 
        }
    }\backsim o(t^{j}),\,\,\,
    \text{as $t \rightarrow 0$,}
    \label{eq:wpse-trace-asym-ds}
\end{align}
provide $l$ large enough. More details can be found in \cite{branson1986conformal}. Again, \eqref{eq:wpse-trace-asym-ds} means that the heat kernel asymptotic defined by $E_l(t,s)$ can be differentiate term by term in $s$.

% Follows from \eqref{eq:wpse-opnorm-estimate-resolvent-ds}, we have similar estimates:   
% \begin{align*}
%     \norm{
%         \frac{d}{ds} \brac{
%             e^{-t P_s } - E_l(\lambda,s)}
%     }_{-j,j}   \backsim o(t^{j}) \,\,\,
%     \text{as $t \rightarrow 0$,}
% \end{align*}
% and 
% \begin{align*}
%     \abs{ \frac{d}{ds} \brac{
%         H_{\tilde E_s}(x,y,t) - H_{E_l(\lambda,s)}(x,y,t) }
%     }_{L^2_j} \backsim o(t^{j}) \,\,\,
%     \text{as $t \rightarrow 0$.}  
% \end{align*}
% 
% In particular, for the $L^2$-trace, 
% \begin{align*}
%     \Tr \brac{
%     e^{-t P_s } - E_l(\lambda,s)
%     } \backsim o(t^{j}), \,\,\,\,
%     \Tr\brac{
%         \frac{d}{ds} \brac{ 
%        e^{-t P_s } - E_l(\lambda,s) 
%         }
%     }\backsim o(t^{j}) 
% \end{align*}
% as $t\rightarrow 0$, for all $j\ge 0$, provided $l$ large enough. \par

To complete the formula of the heat kernel asymptotic, it remains to compute the trace: $\Tr(E_l) = \Tr(\tilde E_l)$ for each $l=0,1,2,\dots$.  Let $\set{b_j(\lambda,s)}_{j=0}^\infty$ be the sequence of symbols defined inductively in \eqref{eq:wpse-theta-inductive-formula-for-rev-approximation-2} with respect to the perturbed Laplacian $P_s$. Although the symbols $b_j$ are constructed out of deformed calculus, the following homogeneity (of degree $-2-j$) in $(\xi,\lambda)$ is preserved as in the classical situation:
\begin{align}
    b_j(c\xi,c^2\lambda) = c^{-2-j} b_j(\xi,\lambda), \,\,\, \forall c>0, \,\, j=0,1,2\dots.
    \label{eq:wpse-homegen-prop-b_j}
\end{align}
As a consequence (see \cite{gilkey1995invariance} section 1.7),  the  following symbols 
\begin{align}
    \tilde e_j(t,s) = \frac{1}{2\pi i}\int_C e^{-t\lambda} b_j(\lambda,s) d\lambda,
    \,\,\, j=0,1,2\dots,
    \label{eq:wpse-theta-int-bj-over-C}
\end{align}
belong to $S\Sigma^{-\infty}$. We define $\tilde E_l = \sum_{j=1}^l \op{Op}(e_j)$ and $E_l = \pi^\Theta(\tilde E_l)$. When taking the homogeneity property \eqref{eq:wpse-homegen-prop-b_j} into account, the trace formula in proposition \ref{prop:wpse-diff-sym-to-kernel-dia} becomes:
\begin{align*}
    H_{b_j(\lambda)}(x,x,\lambda) = \frac{1}{(2\pi)^m}\int_{T_x^*M} b_j(\xi_x,\lambda) d\xi_{x,g^{-1}} + O(\abs\lambda^{-N})
\end{align*}
for all positive integer $N$. Here $H_{b_j(\lambda)}$ is the Schwartz kernel of the quantization operator of $b_j$. After the contour integral, we get, for any positive integer $N$:
     \begin{align*}
    H_{e_j(\lambda)}(x,x,t) = \frac{1}{2\pi i}\frac{1}{(2\pi)^m}
    \int_C e^{-t\lambda}\int_{T_x^*M} b_j(\xi_x,\lambda) d\xi_{x,g^{-1}} d\lambda + o(t^{-N}), \,\,\, \text{as $t\rightarrow 0$.}
\end{align*}
Use the homogeneity property \eqref{eq:wpse-homegen-prop-b_j} again, one can perform a substitution $ t \lambda \rightarrow \lambda$ to show that
\begin{align*}
    \int_C e^{-t\lambda}\int_{T_x^*M} b_j(\xi_x,\lambda) d\xi_{x,g^{-1}} d\lambda
    = t^{(j-n)/2} \int_C e^{-\lambda}\int_{T_x^*M} b_j(\xi_x,\lambda) d\xi_{x,g^{-1}} d\lambda.
\end{align*}
Finally, we have proved that 
\begin{align*}
H_{\tilde E_l}(x,x,t) = \sum_{j=0}^l t^{(j-n)/2}V_j(x) + o(t^{-N}), \,\,\, \forall N \in\N,
\end{align*}
as $t\rightarrow 0$, with 
\begin{align*}
    V_j(x) = \frac{1}{2\pi i}\frac{1}{(2\pi)^m}
    \int_C e^{-\lambda}\int_{T_x^*M} b_j(\xi_x,\lambda) d\xi_{x,g^{-1}} d\lambda.
\end{align*}

We summerize the long dicussion above into the theorem below. 
  \begin{thm}
	    \label{thm:wpse_theta-heatkernel-asym}
        Let $P_s$ with $s \in[0,\varepsilon)$ for small $\varepsilon>0$ be a family of perturbed Laplacians defined in \eqref{eq:wpse-theta-def_Ps}. 
For any $f = \pi^\Theta(\tilde f)$ with $ \tilde f \in C^{\infty}(M)$, viewed as a deformed zero-order pseudo
 differential operator by left-multiplication, then  
\begin{align}
    \label{eq:wpse_theta-heatkernel-asym}
    \Tr\brac{ f e^{-t P_s }} \backsim \nsum{j=0}{\infty}
    t^{(j-m)/2} V_j(f,P_s) 
  \end{align}
where upto a factor $(2\pi)^{-m}$, $m = \dim M$, 
\begin{align}
  \label{eq:wpse_theta-heatkernel-asym-V(f,Delta)}
  \begin{split}
  V_{j}(f,\pi^{\Theta}(P_k)) &= \int_{M}f(x) V_{j}(x) dg = \int_{T^*M} 
  \brac{
  	\frac{1}{2\pi i} \int_{ C} e^{-\lambda} b_{j}(\xi_x,\lambda)d\lambda
  } \Omega, \\
  & = \int_{M}\int_{T^{*}_{x}M}\brac{\frac{1}{2\pi i}
  	 \int_{C}e^{-\lambda} b_{j}(\xi_x,\lambda)d\lambda  } d\xi_{x,g^{-1}
  }dg,
  \end{split}
\end{align}
here $\Omega$ is the canonical volume form on $T^*M$ defined in \eqref{eq:wpse-diff-canonical-twom-form} and $C$ is the contour that defines the heat operator. \par
Moreover, the heat asymptotic \eqref{eq:wpse_theta-heatkernel-asym} can be differentiated term by term in $s$:
\begin{align}
            \frac{d}{ds} \Tr(f e^{-t P_s})
            \backsim \sum_{j=0}^\infty t^{(j-m)/2}
            \frac{d}{ds} V_j(f,P_s).
            \label{eq:wpse_theta-variation-heatkernel}
        \end{align}
        
\end{thm}

We  end this section with a quick application of the technical fact \eqref{eq:wpse_theta-variation-heatkernel}.

\subsection{Zeta functions and conformal indices}
The original notion of conformal index for manifolds was introduced in \cite{branson1986conformal}, which admits a generalization in the setting of spectral triples \cite{MR3194491}. As an instance, let us state the result for toric noncommutative manifolds. Denote $k_s = e^{sh}$ with $s\in[0,1]$ and $h = h^* \in C^\infty(M_\Theta)$.
\begin{thm}
    Consider the perturbed Laplacian in \eqref{eq:wpse-theta-def_Ps} without lower order terms $P_s = k_s \Delta$. We assume that $m = \dim M$ is even and denote
$V_j(P_s) \defeq V_j(1,P_s)$ defined in  \eqref{eq:wpse_theta-heatkernel-asym-V(f,Delta)}, then 
    \begin{align*}
        \frac{d}{ds}V_{m}(P_s) = 0, \,\,\, \forall s\in[0,1].
    \end{align*}
     In particular, at $j=m$, the coefficient
     \begin{align}
         \label{eq:wpse-theta-V_j(kD)}
        V_m(\Delta) = V_m(P_0)  = V_m(P_1)= V_m(k\Delta),
     \end{align}
     that is  the value $V_m(k\Delta)$ does not depend on the Weyl factor $k$.
\end{thm}
After verifying the technical assumption that the heat kernel asymptotic can be differentiated term by term in $s$ (eq. \eqref{eq:wpse_theta-variation-heatkernel}), the theorem is a special case of section 2.1 in \cite{MR3194491}. \par
The result can be rephrased using zeta functions. Again, let $P_k = k \Delta \in \Psi^2(M_\Theta)$. To define the complex power $P_k^z$ where $z$ is a complex number, we need to remove zero from the spectrum of $P_k$. To do so, we consider 
\begin{align*}
    \widetilde{\mathcal P_k} \defeq P_k(I-\mathcal P_{\ker P_k}),
\end{align*}
where $\mathcal P_{\ker P_k}$ is the projection on to the kernel of $P_k$. For $z\in\C$, the complex power is  defined by the contour integral
\begin{align*}
  P_k^z \defeq \frac{1}{2\pi i} \int_C \lambda^z (\widetilde{\mathcal P_k} - \lambda)^{-1} d\lambda.
\end{align*}

For $\Re z$ large enough, $P_k^{-z}$ is of trace-class, so the corespondent zeta function is well-defined:
\begin{align*}
    \zeta(P_k,z) = \Tr P_k^{-z}.
\end{align*}
It is well-known that, for instance, see \cite{gilkey1995invariance} lemma 1.10.1, the heat kernel asymptotic \eqref{eq:wpse_theta-heatkernel-asym} gives rise  to a meromorphic extension of  $\zeta(P_k,z)$ to $\C$ with at most simples poles. Moreover, the coefficients $V_j(P_k)$ correspond to the value or the residue of the zeta function at $z_j = (m-j)/2$, $j=0,1,2,\dots$, where $m =\dim M$. In particular,  at $z = 0$, the zeta function is regular and
\begin{align*}
    \zeta(k\Delta,0) = V_j(k\Delta)-\dim \ker k\Delta.
\end{align*}
Since $k$ is invertible, $\dim \ker k\Delta = \dim \ker \Delta$, combine this with \eqref{eq:wpse-theta-V_j(kD)}, we conclude:
\begin{thm}
	Let $k$ be a Weyl factor, then the zeta function of $k\Delta$ at zero is independent of $k$, that is   
	\begin{align}
		\zeta_{k\Delta}(0) =\zeta_{\Delta}(0). 
	\end{align}
\end{thm}

%% file: mod_cur.tex
\section{Modular curvature}
\label{sec:modcur} 

\begin{defn}
    Let $P_k$   be a  perturbed Laplacian via  a Weyl factor $k$ as before, which stands for a noncommutative metric for the noncommutative manifold $M_\Theta$. Via analogy, we define  the associated  modular  curvature to be the functional density  of the second heat coefficient. Precisely, the modular curvature $\mathcal R = \pi^\Theta( \mathcal{\tilde R})$ with with $\mathcal{\tilde R} \in C^\infty(M)$, is defined by the property: for any $f = \pi^\Theta(\tilde f)$, 
 \begin{align*}
     V_2( f, P_k ) = \int_{M} \tilde f \times_{\Theta} \mathcal{\tilde R}, \,\,\,
 \forall f \in C^{\infty}(M). 
\end{align*}
 
\end{defn}
We have  shown in theorem \ref{thm:wpse_theta-heatkernel-asym} that 
 \begin{align*}
     \mathcal R(x) = (2\pi)^{-\dim M}	\int_{T^*_x M} \frac{1}{2\pi i} \int_{\mathcal C} e^{-\lambda} b_2 (\xi_x, \lambda)d\lambda d\xi_{x,g^{-1}},\,\,\, x \in M,
 \end{align*}
 where $b_{2}$ is the second term in the resolvent approximation of $(P_k - \lambda)^{-1}$.
% $P_k$ is a pseudo differential operator on $M$ whose
% resolvent has an approximation by a sequence of symbols $\set{b_j(\lambda)}_{j=0}^\infty$ with $b_j(\lambda) \in S\Sigma^{-j-2}(M,\lambda)$. According to \eqref{eq:wpse_theta-heatkernel-asym-V(f,Delta)}, the density function is given by $\tilde{\mathcal R}(x)  = \pi^\Theta( \mathcal R)$ with $\mathcal R \in C^\infty(M)$:
%  The calculation is proceed as follows.   We first compute the $b_2$ term explicitly, the result in given in    proposition \ref{prop:mod-cur-prop-b2term}.  We break the integration down is break $d\xi_{x,g^{-1}}$ into polar coordinates:
 We will process the integration as follows:
\begin{align}
 	\tilde{\mathcal R}(x) = \int_{0}^{\infty}\frac{1}{2\pi i} \int_{\mathcal C}	 \brac{ 
	\int_{S^*_x M} e^{-\lambda}b_2 (\xi_x, \lambda)d\omega_{x,g^{-1}} }  d\lambda(r^{m-1}dr),
\label{eq:mod-cur-scalarcur-setup}
 \end{align}
 where $d\omega_{x,g^{-1}}$ is the volume form on the unit cosphere inside $T^*_xM$ associated to the metric $g^{-1}$ and $m$ is the dimension of the underlying manifold. 
% Since the unit cosphere in each $T^*_xM$ is compact, we perform  integration against $d\omega_{x,g^{-1}}$ first, while $d\lambda dr$ is handled by lemma \ref{lem:mod-cur-int-lambda} and the rearrangement lemma developed in \cite{MR2907006}, \cite{MR3194491} and \cite{2014arXiv1405.0863L}. 
% 

 At the end, we showed that the results agree with the previous work  \cite{Lesch:2015aa}, \cite{MR3194491} and  \cite{MR3359018}. In the rest of the computation, the tensor calculus is always the deformed version, we will suppress  all $\otimes_\Theta$, $\times_\Theta$ to simplify the notations.
 
\subsection{Set up and the complete expression of the  $b_2$ term}
  We perform the calculation with respect to a perturbed Laplacian $\pi^\Theta(P_k)$ (cf. definition \ref{defn:wpse_theta-pert-laplacian})  whose symbol is of the form:
 \begin{align}
 \sigma(\pi^\Theta(P_k))= k \abs\xi^2 + p_1(\xi_x) + p_0(\xi_x).
 \label{eq:mod-cur-laptpye-sym}
 \end{align}
 Here $p_1(\xi_x)$ and $p_0(\xi_x)$ are the degree one and zero parts respectively whose explicit expressions will be determined in specific examples. Let us consider, at first, the simplest perturbation $\pi^\Theta(k)\Delta$, whose symbol is the leading term of all the perturbed Laplacians appeared in the previous work \cite{MR3194491}, \cite{MR3148618} and \cite{MR3359018}:
% which is the degree zero Laplacian $\Delta_\varphi = k \Delta k$ appeared in \cite[Lemma 1.11]{MR3194491} upto a conjugation.   
 \begin{lem}
 	\label{lem:mod-cur-symbolof-P_k}
    The symbol of the perturbed Laplacian $k \Delta$ is equal to 
    \begin{align*}
        k \times_\Theta \abs\xi^2 = k \abs\xi^2,
    \end{align*}
    where $\abs\xi$ is the length function on $T^*M$.
%  	Let $k \Delta$ be a perturbed Laplacian with $P_k\in \Psi^2(M)$, a differential operator on $M$. Then the  symbol of $P_k$ is equal to $k \times_\Theta \abs\xi^2 = k \abs\xi^2$, where $\abs\xi$ is the length function on $T^*M$. 
\end{lem}
 \begin{proof}
 We have seen in lemma \ref{lem:wpse-diff-sym-lap}  that $\sigma(\Delta) = \abs\xi^2$. Since $k$ is independent of the fiber direction variable $\xi$, $a_j(k,\cdot) =0$ for all $j > 0$, thus, 
\begin{align*}
 k \star_\Theta \abs\xi^2 = a_0(k,\abs\xi^2)_\Theta = k \times_\Theta \abs\xi^2 =
 k\abs\xi^2,  
\end{align*}
the last equal sign holds because $\abs\xi^2$ is $\T^n$-invariant. 
     
     	 \end{proof}

Denote by
 \begin{align} 
 p_2(\xi_x,\lambda) = k \abs\xi^2 - \lambda
 \label{eq:mod-cur-p2}
 \end{align}
 the parametric leading symbol, and its inverse in the deformed algebra $C^\infty(T^*M_\Theta)$.
 \begin{align*}
 b_0(\xi_x,\lambda) = (k\abs\xi^2 -\lambda)^{-1}.
 \end{align*}
The  construction of $b_0$ is explained in appendix \eqref{App:inverseof_p2} in detail. Recall \eqref{eq:wpse-theta-b1} and \eqref{eq:wpse-theta-b2}, the explicit expressions of $b_1$ and $b_2$ can be calculated by repeatedly  applying the Leibniz property of $D$ and $\nabla$. We shall leave the lengthy calculation in appendix \ref{app-modcur-b2term} and \ref{App:int_over_cosphere}, instead, we start with proposition \ref{prop:app-int_b2-cosphsere}: in \eqref{eq:mod-cur-scalarcur-setup}, the integral over the unit sphere $\int_{S^{m-1}} b_2  d\sigma_{S^{m-1}}$ is equal to, up to an overall factor $\op{Vol}(S^{m-1})$:
\begin{align}
    \label{eq:mod-cur-b2term-int-sphere}
    \begin{split}
 &   \,\,\,\,\frac 4m 2 b_0^3 k^2 (\nabla k) b_0 (\nabla k) b_0 \abs\xi^6 g^{-1}
     -(2+\frac4m) b_0^2 k  (\nabla k) b_0 (\nabla k) b_0 \abs\xi^4 g^{-1}\\
     &+ \frac4m b_0^2k   (\nabla k) b_0^2 k (\nabla k) b_0 \abs\xi^6 g^{-1} 
     - b_0^2k (\nabla^2 k) b_0 \abs\xi^2 g^{-1}
     +  \frac4m b_0^3 k^2 ( \nabla^2 k) b_0\abs\xi^4 g^{-1}\\
     & + \frac1m \frac23 b_0^2 k^2  \mathcal S_{\Delta} b_0 \abs\xi^2.        
    \end{split}
\end{align}
    \subsection{Integration in $\lambda$}
%    Due to the homogeneity of $b_2$ in $(\xi,\sqrt{\lambda})$, the result of the contour integral 
%    $\frac{1}{2\pi i}\int_C e^{-t\lambda}b_2(\xi_x,\lambda) d\lambda$ is equal to $b_2(\xi_x,-1)$ when the dimension of the manifold equals two, this was explained in \cite[section 6]{MR3194491}, the argument can be modified to higher dimensions. \par
% 

As explained in \cite[section 6]{MR3194491}, the resolvent parameter can be taken to be $-1$ due to the homogeneity of the symbol. The argument works in higher dimensions  in the following way. \par

    Let $m =\dim M$ be  even. We fix a point $x \in M$ and identify $T^*_xM \cong \R^m$ so that the Riemannian metric is the usual Euclidean metric.  Put
    \begin{align}
    	b_2(r,\lambda) \defeq \int_{S^{m-1}} b_2 (\xi,\lambda) d\sigma_{S^{m-1}}, 
    \end{align}  
 where $r\in [0,\infty)$ and $\lambda \in \Lambda$, a cone region in $\C$ in which $\sqrt{\lambda}$ is well-defined. 
% the right hand side is computed in \eqref{eq:mod-cur-b2term-int-sphere}. 
    We would like to switch the order of $d\lambda(r^{m-1}dr)$ in \eqref{eq:mod-cur-scalarcur-setup}. 
    Integration by parts gives us: for any integer $j>0$,
   \begin{align*}
      &\,\, \int_{0}^{\infty}  \frac{1}{2\pi i}\int_C e^{-\lambda}	b_2(r,\lambda) d\lambda (r^{m-1}dr) \\
      =&\,\,
      \int_{0}^{\infty}  \frac{1}{2\pi i}\int_C e^{-\lambda} 	\frac{d^j}{d\lambda^j} b_2(r, \lambda) d\lambda  (r^{m-1}dr).
     % \,\,\, \forall j \in \Z_{\ge 0}.
   \end{align*}
Recall that  by saying ``$b(r,\lambda)$ is homogeneous in $(r,\sqrt\lambda)$ of degree $d \in \Z$'' we mean that for any $c>0$, $b(cr, c^{2}\lambda) = c^{d} b(r,\lambda)$.  Observe that differentiating in $\lambda$ lower the homogeneity by $2$, thus  if we take $j$ to be $j_{0} = m/2 -1$, the smallest integer so that  the homogeneity of $\frac{d^j}{d\lambda^j} b_2(r, \lambda)$ is strictly less than $-m$, then integral
\begin{align*}
B_{j_{0}}(\lambda) =  \int_{0}^{\infty} \frac{d^j_{0}}{d\lambda^j_{0}} b_2(r, \lambda) (r^{m-1}dr)
\end{align*}
exists and $B_{j_{0}}(\lambda)$ is  homogeneous in $\lambda$ of degree $-1$, that is, $B_{j_{0}}(c\lambda)=c^{-1} B_{j_{0}}(\lambda)$ for any $c>0$.  Indeed, 
    	\begin{align*}
    	B_{j_0}(c \lambda) &= \int_0^\infty (\frac{d^{j_0}}{d\lambda^{j_0}} b_{2})(\sqrt c \frac{r}{\sqrt c},c\lambda) r^{m-1}dr \\
    	&= \int_0^\infty c^{\frac{-4-2j_0}{2}} (\frac{d^{j_0}}{d\lambda^{j_0}} b_{2})(\frac{\xi}{\sqrt c},\lambda) r^{m-1}dr \\
    	&= c^{-2-j_0 + m/2} \int_0^\infty (\frac{d^{j_0}}{d\lambda^{j_0}} b_{2})(r,\lambda) r^{m-1}dr,
    	\end{align*}
and $-2-j_0 + m/2 =-1$ since $j_{0} = m/2 -1$. 
%    When $m=2$, this was  explained in \cite[section 6]{MR3194491}, it is valid to swap the order, moreover, due to the homogeneity of $b_2$, we have
%    \begin{align*}
%    	\frac{1}{2\pi i}\int_C e^{-t\lambda}  (\int_{0}^{\infty} b_2(r,\lambda)rdr) d\lambda  
%    	= \int_{0}^{\infty} b_2(r,-1) rdr.
%    \end{align*} 
%    \begin{align*}
%    \frac{1}{2\pi i} \int_{T^*_xM} \int_{C}e^{-t\lambda} b_{2}(\xi,\lambda)d\lambda d\xi_{g^{-1}}.
%    \end{align*}
%   For higher dimensions, $\int_{0}^{\infty} b_2(r,\lambda)r^{m-1}dr$ diverges, however, integration by parts gives us:
%   \begin{align*}
%      &\,\, \int_{0}^{\infty}  \frac{1}{2\pi i}\int_C e^{-\lambda}	b_2(r,\lambda) d\lambda (r^{m-1}dr) \\
%      =&\,\,
%      \int_{0}^{\infty}  \frac{1}{2\pi i}\int_C e^{-\lambda} 	\frac{d^j}{d\lambda^j} b_2(r, \lambda) d\lambda  (r^{m-1}dr),
%      \,\,\, \forall j \in \Z_{\ge 0},
%   \end{align*}
%   in which the homogeneity of  $\frac{d^j}{d\lambda^j} b_2(r, \lambda)$ is equal to $-4-2j$. We choose $j$ large enough   so that  switching the order of $d\lambda dr$ is valid. 
   %but the issue can be solved by applying integration by parts on $d\lambda$.  
%    One can not switch the order of integration simply because the $\xi$ integral diverges when the dimension of the manifold is greater or equal than $4$ (recall that $b_2$ is a symbol of order $-4$). The issue is not essential, 
%    
    \begin{lem}
    	\label{lem:mod-cur-int-lambda}
  Keep the notations as above. Assume that  $m = \dim M$ is even and set $j_{0} = m/2 -1$. We have
      	\begin{align}
    	\frac{1}{2\pi i}  \int_{C}e^{-\lambda} B_{j_0}(\lambda)d\lambda = B_{j_0}(-1).
    	\label{eq:mod-cur-int-Bj0-dlambda}
    	\end{align}
	Therefore:
\begin{align}
 \frac{1}{2\pi i}\int_C e^{-t\lambda}  (\int_{0}^{\infty} b_2(r,\lambda)r^{m-1}dr) d\lambda  
 = \int_{0}^{\infty}  \frac{d^j_{0}}{d\lambda^j_{0}}\Big|_{\lambda =-1} b_2(r,\lambda) r^{m-1}dr.
    	\label{eq:mod-cur-put-lamda-to-1}
\end{align}

%    	
%    	Let $m$ be the dimension of the manifold $M$ which is even, $r\in [0,\infty)$ and the resolvent parameter $\lambda$ belong to  a conic region  $\Lambda \subset \C$ in which $\sqrt{\lambda}$ is well-defined. Given a function $b_2(r,\lambda)$ which is homogeneous in $(r, \sqrt{\lambda})$ is the following way, for any integer $j\ge 0$,
%    	\begin{align*}
%    	\frac{d^j}{d\lambda^j} b_2(\xi_x, \lambda)	
%    	\end{align*}
%    	 is homogeneous in $(r, \sqrt{\lambda})$ of degree $-4-2j$. 
%    	%    	Since $b_2(\lambda)$ is a parametric symbol, $d/d\lambda$ will lower the order of $b_2$ by two, namely, $\forall j \in \N$,
%%    	\begin{align*}
%%    	\frac{d^j}{d\lambda^j} b_2(\xi_x, \lambda) \in S\Sigma^{-4-2j}(M,\Lambda).
%%    	\end{align*}
%    	Let $j_0 = m/2 -1$ (the smallest integer such that $4+2j>m$), which assures  the existence of the integral
%    	\begin{align}
%    	B_{j_0}(\lambda) = \int_{0}^\infty \frac{d^j_0}{d\lambda^j_0} b_2(r,\lambda) r^{m-1}dr,
%    	\label{eq:mod-cur-B_j0}
%    	\end{align}
%    	then 
%    	\begin{align}
%    	\frac{1}{2\pi i}  \int_{C}e^{-\lambda} B_{j_0}(\lambda)d\lambda = B_{j_0}(-1)
%    	\label{eq:mod-cur-int-Bj0-dlambda}
%    	\end{align}
 \end{lem}
    
    \begin{proof}
    Let $C \in \C$ be a contour  around $[0,\infty)$ used before to define the heat operator via holomorphic functional calculus, then  
    	\begin{align}
	\label{eq:mod-cur-int-heatintegral-scalar}
    	\frac{1}{2\pi i} \int_C e^{-t\lambda} \frac{1}{s - \lambda} d\lambda = e^{-ts}, \,\,\, \forall s,t\ge 0.
    	\end{align}
	Upto homotopy equivalence in the region $\C \setminus [0,\infty)$, we can force the contour $C$ to be contained in a cone region $U_{\delta} = \set{z \in \C \,| \, \pi-\delta <\arg z < \pi+\delta}$ for any $\delta>0$. Write $\lambda = r e^{i\theta}$ in its polar form, due to the homogeneity:
	\begin{align*}
 B_{j_0}(\lambda) = \frac{1}{r e^{i\theta}}e^{i\theta} B_{j_{0}}(e^{i\theta}) = \frac1\lambda e^{i\theta} B_{j_{0}}(e^{i\theta}),
\end{align*}
	where $\theta$ can be chosen to be contained in $(\pi-\delta,\pi+\delta)$ for any $\delta >0$, hence: 
    	    	\begin{align*}
    	\frac{1}{2\pi i} \int_C e^{-t\lambda} B_{j_0}(\lambda) d\lambda = 
    	\brac{ \frac{1}{2\pi i} \int_C e^{-t\lambda} (-\frac1\lambda) d\lambda } B_{j_0}(-1) = B_{j_0}(-1).
    	\end{align*}
	Here we have used \eqref{eq:mod-cur-int-heatintegral-scalar} to conclude that $ \frac{1}{2\pi i} \int_C e^{-t\lambda} (-\frac1\lambda) d\lambda = e^{0} =1$.
    \end{proof}
        
    \subsection{The rearrangement lemma}   
    %----------Rearrangement lemma -----------------------------------
   Integration in $r$ is handle by the rearrangement lemma which will be explained below. We will feel free to use the notations in \cite{Lesch:2015aa} and \cite{2014arXiv1405.0863L}. 
    From now on,  the parameter $\lambda$ is taken to be $-1$.  Put $r = \abs\xi$.   After a substitution $r \mapsto r^2$, the summands in \eqref{eq:mod-cur-b2term-int-sphere}  contain two types:
    \begin{align}
        \label{eq:mod-cur-ty1-int}
    	k f_0(rk) \rho f_1(rk) \,\, \text{or}\,\,
    	k f_0(rk) \rho_1 f_1(rk) \rho_2 f_2(rk),
    \end{align}
    here $k$ is the Weyl factor and $f_j$'s are some smooth functions on $\R_+$, while $\rho_j$'s are tensor fields over $M$ on which $C^\infty(M_\Theta)$ acts from both sides. Introduce the modular operator $\modop$:
    \begin{align}
        \modop(\rho) \defeq  k^{-1} \rho k, 
           \label{eq:mod-cur-modop}  
        \end{align}
    then the rearrangement lemma (cf. \cite[Lemma 6.2]{MR3194491}, \cite[Corollary 3.9]{2014arXiv1405.0863L}) yields:
    \begin{align}
    \label{eq:mod-cur-relem-onevar}
  &\int_0^\infty k f_0(rk) \rho f_1(rk) dr = \mathcal K(\modop)(\rho), \,\, \text{with}\\ 	
  & \mathcal K(s) =\int_0^\infty f_0(r) f_1(rs)dr, \,\, s\in (0,\infty).  	
    \label{eq:mod-cur-relem-funs-onevar}
    \end{align}
  For the second type,
  \begin{align}
  \label{eq:mod-cur-relem-twovar}
  	&\int_0^\infty k f_0(rk) \rho_1 f_1(rk) \rho_2 f_2(rk) dr 
  	= \mathcal G(\modop_{(1)}, \modop_{(2)})(\rho_1 \cdot \rho_2), \,\, \text{with}\\ 	
  	\label{eq:mod-cur-relem-funs-twovar}
  	& \mathcal G(s , t) = \int_0^\infty f_0(r) f_1(rs) f_2(rst) dr, \,\, s,t\in (0,\infty), 
  \end{align}
  where $\modop_{(j)}$ indicates that $\modop$ acts on the $j$-th factor with $j=1,2$. We introduce the following families of modular curvature functions:
  \begin{align}
      \label{eq:mod-cur-familyOneV}
      K_{(p,q)}(s,r) &= r^{p+q-2} (r+1)^{-p} (sr+1)^{-q}, \\
\label{eq:mod-cur-familyTwoV}
     H_{(p,q,l)}(s,t,r) &=r^{p+q+l-2} (r+1)^{-p} (sr+1)^{-q} (str+1)^{-l},
  \end{align}
  where the parameter $r\in [0,\infty)$ and the arguments $s,t \in (0,\infty)$. For instance, when applying the lemma to $b_0^2k (\nabla^2 k) b_0 \abs\xi^2 g^{-1}$  in \eqref{eq:mod-cur-b2term-int-sphere}, the associated function is $K_{(2,1)}(s,r)$. For the term $b_0^3 k^2 (\nabla k) b_0 (\nabla k) b_0 \abs\xi^6 g^{-1}$, the function is $H_{(3,1,1)}(s,t,r)$.

\subsection{Modular curvature on noncommutative two tori}
Let $M = \T^2$ be the two torus with the Euclidean metric, thus  $m = \dim M =2$ and the scalar curvature function $\mathcal S_\Delta =0$. Recall that the modular curvature $\mathcal R$ is obtained by applying $\int_0^\infty(\cdot) rdr$ to all terms in \eqref{eq:mod-cur-b2term-int-sphere} while the resolvent parameter $\lambda$ is replaced by $-1$. After a substitution $r \mapsto r^2$, $rdr$ becomes $dr/2$. View $ r \mapsto k\abs\xi^2$, then for example, the $b_0$ yields the function $f_1(r) = (r+1)^{-1}$. We would like to apply \eqref{eq:mod-cur-ty1-int} to $2b_0^3 k^2 ( \nabla^2 k)
b_0\abs\xi^4 g^{-1} =k^{-1} 2k b_0^3 k^2 ( \nabla^2 k)
b_0\abs\xi^4 g^{-1}$, therefore 
\begin{align*}
  k^{-1}  \int_0^\infty \brac{2k b_0^3 k^2 ( \nabla^2 k)
      b_0\abs\xi^4 g^{-1} } (dr/2) = \mathcal K_1(\modop)(\nabla^2 k) g^{-1}
\end{align*}
with $\mathcal K_1(s) = \int_0^\infty K_{(3,1)}(s,r) dr$. Other terms in \eqref{eq:mod-cur-b2term-int-sphere} can be handled in a similar way, for instance, $b_0^2k   (\nabla k) b_0^2 k (\nabla k) b_0 \abs\xi^6 g^{-1}$ a term yields a modular function with two variables: we first bring the $k$ in the middle of $\nabla k$ in front via the modular operator, that is: $(\nabla k) k = k \modop(\nabla k)$, so we rewrite 
\begin{align*}
    b_0^2k   (\nabla k) b_0^2 k (\nabla k) b_0 \abs\xi^6 g^{-1} = k^{-2} k b_0^2k^3  \modop(\nabla k) b_0^2  (\nabla k) b_0 \abs\xi^6 g^{-1} 
\end{align*}
and apply \eqref{eq:mod-cur-relem-twovar}:
\begin{align*}
    k^{-2}  \int_0^\infty k b_0^2k^3  \modop(\nabla k) b_0^2  (\nabla k) b_0 \abs\xi^6 g^{-1} (dr/2)
   & = k^{-2} \tilde{\mathcal G}_1 (\modop_{(1)},\modop_{(2)})\brac{
    \modop(\nabla k) k
} g^{-1} \\
&= k^{-2} \mathcal G_1 (\modop_{(1)},\modop_{(2)}) \brac{\nabla k \nabla k}g^{-1},
\end{align*}
with $\mathcal G_1(s,t) = s \tilde{\mathcal G}_1(s,t)$ and $\tilde{\mathcal G}_1(s,t) = \int_0^\infty H_{(2,2,1)}(s,t,r) (dr/2)$.

%-------------------
 We collect the terms that involve $\nabla^2 k$ and their associated functions as below: , which are  $\frac4m b_0^3 k^2 ( \nabla^2 k)
b_0\abs\xi^4 g^{-1}$ and $- b_0^2k (\nabla^2 k) b_0 \abs\xi^2 g^{-1}$,
\begin{align}
    \begin{split}
    \frac4m b_0^3 k^2 ( \nabla^2 k)b_0\abs\xi^4 g^{-1} , \,\,\, & \frac4m K_{(3,1)}(s,r) \\
    - b_0^2k (\nabla^2 k) b_0 \abs\xi^2 g^{-1} , \,\,\, & - K_{(2,1)}(s,r),
    \end{split}
    \label{eq:mod-cur-terms-to-onevarfun}
    \end{align}
they yield the following term in \eqref{eq:mod-cur-scalarcur-twotorus}:
\begin{align*}
    k^{-1} \mathcal K(\modop)(\nabla^2 k) g^{-1},
\end{align*}
where
\begin{align*}
    \mathcal K(s) =\frac12 \int_0^\infty \frac4m K_{(3,1)}(s,r) - K_{(2,1)}(s,r)  dr. 
\end{align*}
The constant $1/2$ comes from the substitution $r \mapsto r^2$. Plug in \eqref{eq:mod-cur-familyOneV} and $m=2$, we obtained 
\begin{align}
   \mathcal K(s) = \frac{-2 s+(s+1) \log (s)+2}{2 (s-1)^3}. 
    \label{eq:mod_cur_K-explicit}
\end{align}
Similarly, according to \eqref{eq:mod-cur-relem-twovar} and \eqref{eq:mod-cur-relem-funs-twovar},
 we collect terms in  \eqref{eq:mod-cur-b2term-int-sphere} that contribute to  the function $\mathcal G$ in \eqref{eq:mod-cur-scalarcur-twotorus} in the following table:
\begin{align}
    \begin{split}
     \frac4m 2 b_0^3 k^2 (\nabla k) b_0 (\nabla k) b_0 \abs\xi^6 g^{-1}, \,\,\, & \frac4m H_{(3,1,1)}(s,t,r), \\
     -(2+\frac4m)b_0^2 k  (\nabla k) b_0 (\nabla k) b_0 \abs\xi^4 g^{-1}, \,\,\,  & -(2+\frac4m) H_{(2,1,1)}(s,t,r), \\
     \frac4m b_0^2k   (\nabla k) b_0^2 k (\nabla k) b_0 \abs\xi^6 g^{-1}, \,\,\,  & \frac4m sH_{(2,2,1)}(s,t,r).   
    \end{split}
    \label{eq:mod-cur-terms-to-twovarfun}
\end{align}
%where the second column stands for the functions needed to be integrated over $[0,\infty)$ in $r$.
%     The functions $H_{(p,q,l)}$ are defined in \eqref{eq:mod-cur-familyTwoV}. Notice that in the third row, there is a $k$ between two $\nabla k$, we write $(\nabla k) k = k \modop(\nabla k)$, which yields the function $s$ in front of $H_{(2,2,1)}(s,t,r)$. 
When $m=2$, the two variable function $\mathcal G$ is given by:
    \begin{align*}
        \mathcal G(s,t) =\frac12 \int_0^\infty \brac{
        4H_{(3,1,1)} -4H_{(2,1,1)}+4sH_{(2,2,1)}
        }(s,t,r) dr
    \end{align*}
The explicit expression is given by:
\begin{align}
    \mathcal G(s,t) 
    = {\scriptstyle 
        \frac{(s t-1)^3 \log (s)-(s-1) ((t-1) (s (t-2)+1) (s t-1)+(s-1) (s t (2 t-1)-1) \log (s t))}{(s-1)^2 s (t-1)^2 (s t-1)^3} }.
    \label{eq:mod_cur-calG-explicit}
\end{align}
Finally, taking two overall factors into account: $\mathrm{Vol}(S^{m-1})$ and $(2\pi)^{-m}$ (cf. \eqref{eq:mod-cur-b2term-int-sphere} and Theorem $\ref{thm:wpse_theta-heatkernel-asym}$), we have proved the following theorem. 
\begin{thm} \label{thm:mod-cur-scalarcur-twotorus}
    Let $M_\Theta = \T^2_\Theta$, the noncommutative two torus and $P_k = k  \Delta$. Then the functional $V_2(\cdot,P_k)$ of the second heat coefficient  can be express in in the following way: $\forall f \in C^\infty(M)$, put $f = \pi^\Theta(\tilde f)$,
    \begin{align}
        V_2(f,P_k ) = \int_M \tilde f\times_\Theta \mathcal{\tilde R} dg.
        \label{eq:mod_cur_defn_mathcalR}
    \end{align}
   Keep the notations as in \eqref{eq:mod-cur-modop}-\eqref{eq:mod-cur-relem-funs-twovar}, then upto a constant factor $(2\pi)^{-1}$, %$($cf. Theorem $\ref{thm:wpse_theta-heatkernel-asym}$$)$,
    \begin{align}
         \mathcal{\tilde R} =   k^{-1} \mathcal K (\modop)(\nabla^2 k) g^{-1} + k^{-2} \mathcal G(\modop_{(1)},\modop_{(2)}) \brac{\nabla k \nabla k}g^{-1},
        \label{eq:mod-cur-scalarcur-twotorus}
    \end{align}
   where the modular curvature functions $\mathcal K$ and $\mathcal G$ are given in \eqref{eq:mod_cur_K-explicit} and \eqref{eq:mod_cur-calG-explicit}. 
%     can be written as linear combinations of simple\footnote{That means at most the third divided difference occurs.} divided differences of the $\log$ function. The explicit expressions 
\end{thm}
\begin{rem}\mbox{}
    \begin{enumerate}[(1)]
        \item  $\mathcal R = \pi^\Theta(\mathcal{\tilde R}) \in C^\infty(M_\Theta)$ is the associated modular curvature. 
        \item Originally, the integrand of the right hand side of \eqref{eq:mod_cur_defn_mathcalR} should be $\overline{\mathcal R} \times_\Theta \tilde f$, but since $\mathcal R$ is self-adjoint and $\int_M$ is a  trace with respect to the deformed product, we see that 
    \begin{align*}
    \int_M \overline{\mathcal R} \times_\Theta \tilde f dg = \int_M \tilde f\times_\Theta \mathcal R dg.
    \end{align*}
\item One can verify that  the functions $\mathcal K$ and $\mathcal G$ agree with the modular curvature functions $-F_{0,0}(s)$ and $ G_{0,0}^{\mathfrak R}(s,t)$ in \cite[Thm 3.2]{Lesch:2015aa} in the following way:
    \begin{align*}
        \mathcal K(s) =   -F_{0,0}(s),  \,\,\,\,\,\, 
        \mathcal G(s,t) = \frac1s G_{0,0}^{\mathfrak R}(s,t).
    \end{align*}
    The appearance of the negative sign in front of  $F_{0,0}$ is due to the fact that $(\nabla^2 k)g^{-1} = - \Delta k$. In \cite[Thm 3.2]{Lesch:2015aa}, the quadratic form  is defined as $(k^{-1}\nabla k) (k^{-1}\nabla k) g^{-1}$, compare to our quadratic form $k^{-2}(\nabla k)(\nabla k) g^{-1}$. The factor $1/s$ in front of $G_{0,0}^{\mathfrak R}$ that stands for the inverse modular operator $\modop^{-1} = k (\cdot ) k^{-1}$ is exactly the price to pay to  move  $k^{-1}$ in front of $\nabla k$. 

%---------------------
\item The Laplacian $k \Delta$ considered in the theorem is related to  the degree zero Laplacian with complex structure $\sqrt{-1}$ appeared in \cite{MR3194491} via the conjugation by $k$: $k^2 \Delta \mapsto k \Delta k$. After converting the Weyl factor $k$ and the modular operator $\modop$ to their own logarithms: $h = \log k$ (defined by $k =e^h$) and $\logmodop = \log \modop = -[h,\cdot]$, our modular curvature functions  agree with those in \cite{MR3194491}. This issue is explained in
    \cite[sec. 4.7]{Lesch:2015aa} in detail. Therefore our calculation gives a new confirmation of the results for noncommutative two tori which is independent of the aid of CAS. 
    
% In \cite{Lesch:2015aa}, section 3 and 4  show that after converting the Weyl factor $k$ and the modular operator $\modop$ to their own logarithms: $h = \log k$ and $\logmodop = \log \modop = -[h,\cdot]$, also taking the conjugation by $k$ into account: $k^2 \Delta \mapsto k \Delta k$, our modular curvature functions  agree with  those  for the degree zero Laplacian with complex structure $\sqrt{-1}$ appeared in \cite{MR3194491}.    
    \end{enumerate}
    \end{rem}

\subsection{Modular curvature for even dimensional toric noncommutative manifolds}
Now let us assume that $M$ is even dimensional and $m = \dim M \ge 4$. From lemma \ref{lem:mod-cur-int-lambda}, is modular curvature is obtained by integrating $b_2$ defined in \eqref{eq:mod-cur-b2term-int-sphere} in the following way: set $j_0 = m/2-1$,
\begin{align*}
     \mathcal{\tilde R}= \int_0^\infty \brac{\frac{d^{j_0}}{d\lambda^{j_0}} b_2} \Big|_{\lambda=-1} (r^{m-1}dr).
\end{align*}
As before, we perform replace $r$ by $r^2$ so that the volume form $r^{m-1}dr$ becomes $r^{m/2-1}dr/2$. Recall the functions $K_{(p,q)}(s,r)$ and $H_{(p,q,l)}(s,t,r)$ defined in \eqref{eq:mod-cur-familyOneV} and \eqref{eq:mod-cur-familyTwoV}, take again the term $\frac4m b_0^3 k^2 ( \nabla^2 k)b_0\abs\xi^4 g^{-1}$ as an example, it leads to the function in $s$:
\begin{align*}
    \frac4m  \frac12 \int_0^\infty \frac{d^{j_0}}{d\lambda^{j_0}}\Big|_{\lambda=-1} \frac{r^2}{(r-\lambda)^3}
    \frac{1}{sr-\lambda} r^{m/2-1}dr, \,\,\,\, j_0 = m/2-1.
\end{align*}
Let us consider in general
\begin{align}
    \int_0^\infty \frac{d^{j_0}}{d\lambda^{j_0}}\Big|_{\lambda=-1}
    \frac{r^{p+q-2}}{(r-\lambda)^p} \frac{1}{(sr-\lambda)^q} r^{m/2-1}dr,
    \,\,\,\, j_0 = m/2-1.
\end{align}
Via a substitution $u = 1/r$, the integral becomes:
\begin{align*}
   & \int_0^\infty  u^{-j_0} \frac{d^{j_0}}{d\lambda^{j_0}}\Big|_{\lambda=-1} \frac{1}{(1-u\lambda)^p}
    \frac{1}{(s-u\lambda)^q}
    (-du)\\
    =\,\, &\int_0^\infty \frac{d^{j_0}}{dx^{j_0}}\Big|_{x=-u} \frac{1}{(1-x)^p} \frac{1}{(s-x)^q} (-du)
    \\
    =\,\,& \int_0^\infty \frac{d^{j_0}}{du^{j_0}} 
    \frac{1}{(1-u)^p} \frac{1}{(s-u)^q} (-du)
    = \brac{\frac{d}{du}}^{j_0-1}\Big|_{u=0} \frac{1}{(1-u)^p} \frac{1}{(s-u)^q}, 
\end{align*}
here we need the fact that  $p,q$ are both positive integers so that the limit at infinity equals zero. Due to the homogeneity of $b_2$ in $r$, all the terms in \eqref{eq:mod-cur-b2term-int-sphere} can be handle in a similar way, therefore, we upgrade functions in \eqref{eq:mod-cur-familyOneV} and \eqref{eq:mod-cur-familyTwoV} as follows: 
\begin{align}
 \label{eq:mod-cur-familyOneV-dimfour}
    \tilde{K}_{(p,q)}(s,m) & = \brac{\frac{d}{du}}^{m/2-2}\Big|_{u=0} \frac{1}{(1-u)^p} \frac{1}{(s-u)^q}, \\
    \tilde H_{(p,q,l)}(s,t,m) &= \brac{\frac{d}{du}}^{m/2-2}\Big|_{u=0} \frac{1}{(1-u)^p}
    \frac{1}{(s-u)^q} \frac{1}{(st-u)^l}. 
    \label{eq:mod-cur-familyTwoV-dimfour}
\end{align}
Base on \eqref{eq:mod-cur-terms-to-onevarfun} and \eqref{eq:mod-cur-terms-to-twovarfun}, we can write down the modular curvature function for dimension $m\ge 4$: 
\begin{align}
    \mathcal K(s,m) &= \frac12\brac{
    \frac4m \tilde{K}_{(3,1)}(s,m) - \tilde{K}_{(2,1)}(s,m)
    },
    \label{eq:mod-cur-dim4-calK}
    \\
    \label{eq:mod-cur-dim4-calG}
    \mathcal G(s,t,m) &= \frac12\brac{
    \frac8m  \tilde H_{(3,1,1)} - (2+\frac4m) \tilde H_{(2,1,1)}
    +\frac4m s \tilde H_{(2,2,1)}
    }(s,t,m). 
\end{align}

The only term left in \eqref{eq:mod-cur-b2term-int-sphere} is $\frac1m \frac23 b_0^2 k^2 \mathcal S_\Delta b_0 \abs\xi^2$. Apply the rearrangement lemma, we see that it   becomes $\frac1m \frac23 k^{-m/2+1} F(\modop)(\mathcal S_\Delta)$ after integration, where the function $F(s) = \frac12 \tilde{K}_{(2,1)}(s,m)$. Since $\mathcal S_\Delta$ is $\T^n$-invariant, in particular, it commutes with $k$, in other words, $\modop(\mathcal S_\Delta) = \mathcal S_\Delta$, therefore $F(\modop)(\mathcal S_\Delta) = F(1) \mathcal
S_\Delta$ and we denote
\begin{align}
    F(1) &= \frac12 \tilde{K}_{(2,1)}(1,m) = \frac12 ((1-u)^{-3})^{(m/2-2)} \Big|_{u=0}\nonumber \\
& = \begin{cases}
     \frac14 (-1)^{m/2-2} (m/2)! & m=6,8,10,\dots, \\
     \frac12 & m=4 .
 \end{cases}
 \label{eq:mod-cur-dim4-cm}
\end{align}

\begin{thm}
    \label{thm:mod_cur-anyevendim-mflds}
    Let $M_\Theta$ be a noncommutative toric manifold whose dimension is an even integer $m$ and  $\pi^\Theta(P_k) = \pi^\Theta(k) \Delta$. Then the associated modular curvature $\mathcal{\tilde R}$,
    %$\mathcal R(k)$ $($the density function of the second heat coefficient functional, cf. \eqref{eq:mod_cur_defn_mathcalR}$)$,
    upto an overall constant $\op{Vol}(S^{m-1})(2\pi)^{-m}$,  is of the form:   
     \begin{align}
          \label{eq:mod-cur-scalarcur-Mdim-m}
         \mathcal{\tilde R} &=  \brac{  k^{-m/2} \mathcal K (\modop,m)(\nabla^2 k) + k^{-m/2-1} \mathcal G(\modop_{(1)},\modop_{(2)},m) \brac{\nabla k \nabla k}}g^{-1} \\
        &+   c_m k^{-m/2+1 }\mathcal S_{\Delta}. 
        \nonumber
     \end{align}
     When $m\ge 4$, the modular curvature functions $\mathcal K$ and $\mathcal G$  are given in \eqref{eq:mod-cur-dim4-calK}, \eqref{eq:mod-cur-dim4-calG} respectively.  
    While the constant $c_m$ is equal to $\frac1m \frac23 F(1)$, where $F(1)$ is calculated in  \eqref{eq:mod-cur-dim4-cm}. 
%     where $c_m$ is a constant depending only on the dimension of the manifold and $\mathcal S_{\Delta} \in C^\infty(M)$ is the scalar curvature function.  The modular curvature functions $\mathcal K$ and $\mathcal G$ are given by 
%     \eqref{eq:mod_cur-mathcalK-m} and \eqref{eq:mod_cur-mathcalG-m} respectively. As before, they are  linear combinations of simple divided differences of the $\log$ function.
\end{thm}

In dimension four, $m/2-2 =0$, thus no differentiation involves when computing $\tilde K_{(p,q)}$ and $\tilde H_{(p,q,l)}$. So one can quickly compute $\tilde{K}_{(3,1)}(s,4) = \tilde{K}_{(1,1)}(s,4) = 1/s$, therefore $\mathcal K(s,4) = 0$. Meanwhile, 
\begin{align*}
\tilde H_{(3,1,1)}(s,t,4) = \tilde H_{(2,1,1)}(s,t,4) = \frac{1}{s^2t}, \,\,\,
\tilde H_{(2,2,1)}(s,t,4) = \frac{1}{s^3t},
\end{align*}
as a result,  $\mathcal G(s,t,4) = 0$.

\begin{cor}
	Let $m = \dim M = 4$.  For the perturbed Laplacian $\pi^\Theta(P_k) = \pi^\Theta(k) \Delta$, the modular  curvature is simply: 
	\begin{align}
	\label{eq:mod_cur-modcur-dim4-kDelta}
    \mathcal R(k) = c k^{-1} \mathcal S_{\Delta}. 
	\end{align}
    where
    \begin{align*}
        c = (4\pi)^{-2} \frac16.
    \end{align*}
\end{cor}
\begin{rem}
    The value of $c$ agrees with the classical result: upto a factor $(4\pi)^{-\dim M/2}$, the density of the second heat coefficient for the scalar Laplacian operator equals $1/6$ times the scalar curvature.  
\end{rem}

\begin{proof}
It remains to determine the coefficient for the scalar curvature term. We recall 
\begin{align*}
    F(1) = \frac12, \,\,\, \op{Vol}(S^3) = \frac{2 \pi^{4/2}}{\Gamma(2)} = 2 \pi^2,     
\end{align*}
thus
\begin{align*}
    c = \op{Vol}(S^3) (2\pi)^{-4} \frac1m \frac23 F(1) = (4\pi)^{-2} \frac16.
\end{align*}

% 	Using \eqref{eq:mod_cur-mathcalK-m} and \eqref{eq:mod_cur-mathcalG-m}, or \eqref{eq:mod_cur-mathcalKdim4} and \eqref{eq:mod_cur-mathcalGdim4}, we conclude that  both modular functions $\mathcal{K}(s)$ and $\mathcal G(s_1,s_2)$ are zero in dimension four, therefore only the scalar curvature term survives.
\end{proof}

\subsection{Comparison with \cite{MR3359018}  and \cite{MR3369894}}
%This is not surprising due to the result in \cite{2013arXiv1301.6135F} and \cite{Fathizadeh:2014aa}.
In this section, we would like to reproduce the results on noncommutative four tori in \cite{MR3359018}  and \cite{MR3369894}. The perturbed Laplacian $\pi^\Theta(P_k)$ considered in  \cite[lemma 3.3]{MR3359018} has the following symbol: $\sigma(P_k) = p_2 + p_1 + p_0$ with 
\begin{align}
\label{eq:mod_cur_lowerorder-symbols}
p_2  = k \abs\xi^2, \,\, p_1 = \frac{-i}{2}(\nabla k) D \abs\xi^2, \,\, p_0 = -\Delta k + (\nabla k) k^{-1} (\nabla k) g^{-1}, 
\end{align}
where the Laplacian $\Delta$ is associated to the flat metric on $\T^4$.  
Based on the previous computation, the contribution to the modular curvature functions $\mathcal K$ and $\mathcal G$ from the leading term $p_2$  is equal to zero. It remains to count the contribution from $p_1$ and $p_0$ which is in the last line of the $b_2$ term in \eqref{eq:mod-cur-prop-b2term}: 
\begin{align}
\label{eq:mod_cur_b2from-lowerorder-symbols}
	- iD(b_0 p_1 b_0) (\nabla p_2) b_0 -  b_0 p_0 b_0 - b_1 p_1 b_0 + i Db_0 (\nabla p_1) b_0.
\end{align} 
Similar to the computation in appendix  \ref{app-modcur-b2term},  we list the result of each summand as below:
For $- iD(b_0 p_1 b_0) (\nabla p_2) b_0$:
\begin{align*}
   & \frac12\brac{
	b_0^2 k(\nabla k) b_0 (\nabla k) b_0 
+ b_0 (\nabla k) b_0^2 k ( \nabla k) b_0 }(D\abs\xi^2)^2 \abs\xi^2 \\
    -& \frac12  (b_0 (\nabla k) b_0  (\nabla k) b_0 (D^2\abs\xi^2)\abs\xi^2).
\end{align*}
For $- b_1 p_1 b_0$:
\begin{align*}
   \frac12 b_0^2k (\nabla k) b_0 (\nabla k) b_0 (D\abs\xi^2)^2 \abs\xi^2
		+ \frac14 b_0(\nabla k) b_0 (\nabla k) b_0 (D\abs\xi^2)^2. 
\end{align*}
For $-  b_0 p_0 b_0$:
\begin{align*}
    b_0(\Delta k) b_0 - b_0(\nabla k) k^{-1} (\nabla k) b_0 g^{-1}.
\end{align*}
The last one $i Db_0 (\nabla p_1) b_0 = \frac12 b_0^2 k (\nabla^2k) b_0 (D\abs\xi^2)^2$.

Sum up the terms above and perform integrating over the cosphere bundle, namely the following substitution: 
\begin{align*}
    (D\abs\xi^2)^2 \mapsto\frac{4\abs\xi^2}{m}\op{Vol}(S^{m-1})g^{-1},& \,\,\,
   D^2\abs\xi^2 \mapsto 2 \op{Vol}(S^{m-1})g^{-1},
\end{align*}
We group the sum into two parts  associated to the modular curvature function $\mathcal K(s)$ and $\mathcal G(s,t)$ respectively. More precisely, \eqref{eq:mod_cur_b2from-lowerorder-symbols} can be written as  $I_1 + I_2$
upto an overall factor $\mathrm{Vol}(S^3)$,  with
\begin{align}
    I_1 =  b_0(\Delta k) b_0 + \frac12 \frac4m b_0^2 k (\nabla^2k) b_0 \abs\xi^2 g^{-1},
\end{align}
and:
\begin{align}
    \begin{split}
   I_2  = & \,\,\frac4m b_0^2 k(\nabla k) b_0 (\nabla k) b_0 \abs\xi^4 g^{-1}+ \frac12 \frac4m b_0 (\nabla k) b_0^2 k ( \nabla k) b_0 \abs\xi^4g^{-1} \\  
   +&\,\, (\frac14 \frac4m -1) b_0(\nabla k) b_0 (\nabla k) b_0 \abs\xi^2 g^{-1}
   - b_0(\nabla k) k^{-1} (\nabla k) b_0 g^{-1}.
    \end{split}
    \end{align}
  %The two parts $I_1$ and $I_2$ yield the modular curvature function $\mathcal K$ and $\mathcal G$ respectively.  

Notice that $\Delta k = -(\nabla^2 k)g^{-1}$, therefore $I_1$ is of the form 
    \begin{align*}
        I_1 \mapsto k^{-2} \mathcal K(\modop)(\nabla^2 k) g^{-1}
    \end{align*}
   after integration,  with
    \begin{align}
        \mathcal K(s) = \frac12 \brac{-K_{(1,1)}(s,4) + \frac12 \frac4m K_{(2,1)}(s,4)}
        = -\frac14\frac1s.
        \label{eq:mod_cur-calK-nc4tori}
    \end{align}
Meanwhile,  
\begin{align*}
    I_2 \mapsto  k^{-3} \mathcal G(\modop_{(1)},\modop_{(2)})(\nabla k \nabla k)
\end{align*}
where $\mathcal G(s,t)$ equals $(8s^2t)^{-1}$, which comes from the sum:
\begin{align*}
    \frac12 \brac{
        \brac{ \frac4m H_{(2,1,1)}+ \frac12 \frac4m s H_{(1,2,1)} + 
    (\frac14\frac4m -1) H_{(1,1,1)}
}(s,t,4) - \frac1s K_{(1,1)}(st,4)
    },   
\end{align*}
here we have used the fact that
\begin{align*}
    H_{(2,1,1)}(s,t,4) = s H_{(1,2,1)}(s,t,4) = H_{(1,1,1)}(s,t,4) = \frac1s K_{(1,1)}(st,4)=\frac{1}{s^2t}.
\end{align*}

We summarise the computation as below.
\begin{thm}
    Let $M_\Theta = \T^4_\Theta$ be a noncommutative four torus. With respect to the perturbed Laplacian $\pi^\Theta(P_k)$  whose symbol is defined in eq. \eqref{eq:mod_cur_lowerorder-symbols}, the modular curvature is of the form:
    \begin{align*}
        \mathcal R =  k^{-2} \mathcal K(\modop)(\nabla^2 k) g^{-1} + k^{-3} \mathcal G(\modop_{(1)},\modop_{(2)})(\nabla k \nabla k)
    \end{align*}
    upto an overall factor $\mathrm{Vol}(S^3)(2\pi)^{-4}$,   where 
    \begin{align}
    \mathcal K(s) = \frac1{4s}, \,\,\, \mathcal G(s,t)  = - \frac{1}{8s^2 t}.
    \label{eq:mod-cur-calKG}
\end{align}
\end{thm}
The result should be compared with \cite[Eq. (1)]{MR3369894}.

%% file: app.tex
% !TEX root =  main.tex
\appendix
\section{Inverse of the leading symbol $(k \abs\xi^{2} - \lambda)$} 
\label{App:inverseof_p2}
% the \\ insures the section title is centered below the phrase: AppendixA

For $\lambda \in \C \setminus [0,\infty)$, we would like to show that the leading symbol of a perturbed Laplacian $k \abs\xi^2 - \lambda$ (see definition \ref{defn:wpse_diff-defn-Weylfactor} and \ref{defn:wpse_theta-pert-laplacian}) is invertible  in the symbol algebra $S\Sigma(M_\Theta, \Lambda)$, upto smoothing symbols. \par
    First, we recall a technical result.  Given a toric Spin manifold $M$ with spinor bundle $\slashed S$ and the Dirac operator $\slashed D$, the deformed spectral triple $(C^\infty(M_\Theta), L^2(\slashed S), D)$ is a regular spectral triple in the sense of, for instance, \cite{MR1789831}, \cite{MR1995874}. In particular, the subalgebra $C^\infty(M_\Theta)$ inside $B(L^2(\slashed S))$ is closed under holomorphic functional calculus. We refer to \cite[Prop. 5]{2010LMaPh..94..263Y}
    for the proof. \par    
   Given a Weyl factor  $\pi^\Theta(k)$ (with $k\in C^\infty(M)$), whose spectrum is contained in $(0,\infty)$, thus  for any $s >0$, and $\lambda \in \C \setminus (0,\infty)$, the resolvent $(s\pi^\Theta(k) - \lambda)^{-1}$ still lies in $C(M_\Theta) = \pi^\Theta( C^\infty(M))$. In particular, there exists a unique smooth function in $C^\infty(M)$,
denoted by $(sk - \lambda)^{-1}$ (here the $-1$ power stands for the deformed inverse, other than the reciprocal of a function), such that $\pi^\Theta((sk - \lambda)^{-1}) = (s \pi^\Theta(k) - \lambda)^{-1}$.
% Since $\pi^\Theta(k)$ (with $k\in C^\infty(M)$) is  a positive invertible element in the $C^*$-algebra $C(M_\Theta)$, then for any $s >0$, and $\lambda \in \C \setminus [0,\infty)$, the resolvent $(s\pi^\Theta(k) - \lambda)^{-1}$ is well-defined in $C(M_\Theta)$. 
% Due to lemma \ref{lem:wpse-theta-holomorophically-closed},  the resolvent $(s\pi^\Theta(k) - \lambda)^{-1}$ indeed, belongs to $C^\infty(M_\Theta) = \pi^\Theta(C^\infty(M))$. Therefore there exists a unique smooth function in $C^\infty(M)$,
% denoted by $(sk - \lambda)^{-1}$ (here the $-1$ power stands for the inverse of an operator other than the reciprocal of a function), such that $\pi^\Theta((s\pi^\Theta(k) - \lambda)^{-1}) = (s \pi^\Theta(k) - \lambda)^{-1}$. 
As a $C^\infty(M)$-valued function, $(s,\lambda) \mapsto (sk - \lambda)^{-1}$ is smooth in $s \in (0,\infty)$ and holomorphic in $\lambda \in \C \setminus (0,\infty)$. To avoid the singularity caused by $s$ being zero, we choose a cut-off function 
\begin{align}
\rho(s) = 
\begin{cases}
0 & \text{for $t\le 1/2$} \\ 1 & \text{for $t\ge 1$},
\end{cases}
\label{eq:wpse-theta-cutoff-rho}
\end{align}  
and denote  
\begin{align*}
r(s,\lambda) = \rho(s)(s k - \lambda)^{-1}.
\end{align*}
Finally, we claim that the inverse (upto smoothing symbols) of $k\abs\xi^2 -\lambda$ in $S\Sigma(M_\Theta)/S\Sigma^{-\infty}$ is given by composing $r(s,\lambda)$ with the squared length function on $T^*M$:  
\begin{align}
b_0(\xi_x,\lambda) = r(\abs{\xi_x}^2,\lambda),  \,\,\, \xi_x \in T^*_xM,
\label{eq:wpse-theta-b0}
\end{align}
which is a well-defined smooth functions on $T^*M$.

\begin{prop}
	\label{prop:wpse-theta-b_0}
	The function $b_0(\xi_x,\lambda)$ defined in \eqref{eq:wpse-theta-b0} belongs to $S\Sigma^{-2}(M_\Theta,\lambda)$ and serves as the inverse of $p_2(\xi_x,\lambda) = k\abs\xi^2 - \lambda$ in the deformed symbol algebra $S\Sigma(M_\Theta,\Lambda)$, that is $p_2 \times_\Theta b_0 \backsim 1$ and $b_0 \times_\Theta p_2 \backsim 1$. Here $\backsim$ means  ``upto smoothing symbols''.    
\end{prop}

\begin{proof}
Let $\times_\Theta$ and $*_\Theta$ denote the multiplication in $S\Sigma(M_\Theta,\Lambda)$ and $C^\infty(M_\Theta)$ respectively. Since the length function  $\abs\xi$ is $\T^n$-invariant, let $s \in (0,\infty)$  as before and fix a $\lambda$, we see that 
\begin{align*}
(p_2 \times_\Theta b_0) \big|_{\xi_x} = \brac{ (sk - \lambda)*_\Theta r(s,\lambda)}\big|_{s = \abs{\xi_x}^2}.
\end{align*}	
In the right hand side above, $(sk - \lambda)$ and $r(s,t)$ are functions in $s$ valued in $C^\infty(M)$, thus the $*_\Theta$-multiplication makes sense. By the construction of $r(s,\lambda)$, $(sk - \lambda)*_\Theta r(s,\lambda) = 1$ when $s\geq 1$, hence $(p_2 \times_\Theta b_0) \big|_{\xi_x} = 1$ for all $\abs{\xi_x}\geq 1$, thus $p_2 \times_\Theta b_0 \backsim 1$. Same argument shows that $b_0 \times_\Theta p_2 \backsim 1$. \par
One can quickly verify that for any $c>0$, $b_0(c \xi_x, c^2 \lambda) = c^{-2} b_0(\xi_x,\lambda)$, provided $\abs{\xi_x}$ large enough. Hence $b_0$ belongs to $S\Sigma^{-2}(M_\Theta,\lambda)$.

\end{proof}

    \section{Computation of the $b_2$ term} \label{app-modcur-b2term}
    The goal of this section is to compute the full expression of the $b_2$ term in the resolvent aproximation. Recall from \eqref{eq:wpse-theta-b1} and \eqref{eq:wpse-theta-b2}, we need to compute:
\begin{align}
    b_1 & = ( b_0 p_1 + (-i)D b_0  \nabla p_2)  (-b_0),
    \label{eq:app-mod-cur-b1term-v1}
\\
b_2 & = \left(
\vphantom{\frac12}
b_0 p_0 + b_1 p_1 + (-i)D b_0 \nabla p_1 + (-i)Db_1   \nabla p_2 \right. \\
&\left.   - \frac12 D^2 b_0  \nabla^2 p_2 
- \frac12 Db_0  D^2 p_2 ( \nabla^3 \ell) \right)
 (-b_0),
\label{eq:app-mod-cur-b2term-v2}
\end{align}
where $b_0 = p_2^{-1} = (k\abs\xi^2 -\lambda)^{-1}$.

 Since $\nabla \abs\xi^2 =0$ (proved in lemma \ref{lem:mod-cur-xisquared-derivatives}), the Leibniz rule gives us
\begin{align}
    \nabla  k \abs\xi^2 & = (\nabla k)\abs\xi^2, 
    \label{eq:mod-cur-hord-p2-1}
    \\
    \nabla^2 k \abs\xi^2 & = (\nabla^2 k) \abs\xi^2 + k (\nabla^2 \abs\xi^2).
    \label{eq:mod-cur-hord-p2-2}
\end{align}

The derivatives of $b_0 = (k\abs\xi^2 -\lambda)^{-1}$ are obtained by applying the standard resolvent identity:
\begin{align*}
    (Db_0) =   - b_0 (Dp_2) b_0, 
\end{align*}
which holds as well when $D$ is replaced by $\nabla$. We notice that $D p_2$ commutes with $b_0$, but $\nabla p_0$ does not, because 
$\nabla k$ and $k$ do not commute. We continue the computation:
    \begin{align}
         D b_0 = -(b_0  b_0) D p_2 = -(b_0^2   k) D \abs\xi^2,
        \label{eq:mod-cur-Db_0}
    \end{align}
    and
    \begin{align}
        \begin{split}
             D^2 b_0 & = D(-b_0 b_0  D p_2)\\
    &= -b_0  b_0  D^2 p_2  
    + 2 b_0  b_0  b_0  D p_2  D p_2\\
    & = - b_2^2  k D^2 \abs\xi^2 + 2 b_0^3  k^2 (D \abs\xi^2)^2,
        \end{split}
        \label{eq:mod-cur-D^2b_0}
    \end{align}
%\end{lem}
Combine  \eqref{eq:app-mod-cur-b1term-v1}, \eqref{eq:mod-cur-Db_0} and \eqref{eq:mod-cur-hord-p2-1}, we get:
%The $b_1$ term defined in equation \eqref{eq:mod-cur-b1term-v1} can be further computed:
\begin{align}
     b_1 = -i b_0^2 k (D\abs\xi^2 \cdot \nabla k) b_0 \abs\xi^2 - b_0p_1b_0.
    \label{eq:mod-cur-b_1-vesion2}
\end{align}

\begin{prop}
     \label{prop:mod-cur-prop-b2term}
    The $b_2$ term is given by:
    \begin{align}
        \begin{split}
       b_2  &= 2 b_0^3 k^2 ( \nabla k) b_0 (\nabla k) b_0 \abs\xi^4 (D\abs\xi^2)^2\\
   & - b_0^2 k  (\nabla k) b_0 (\nabla k) b_0 \brac{
   \abs\xi^2(D\abs\xi^2)^2 + \abs\xi^4 D^2\abs\xi^2}\\
   & +  b_0^2k   (\nabla k) b_0^2 k (\nabla k) b_0 \abs\xi^4 (D\abs\xi^2)^2\\
   &-\frac12 (b_0^3 k^2) ( \nabla^3 \ell)(D\abs\xi^2) (D^2 \abs\xi^2) 
   -\frac12 b_0^3 k^2 D^2\abs\xi^2 \nabla^2 \abs\xi^2 \\
& -\frac12 b_0^2k (\nabla^2 k) b_0 \abs\xi^2 D^2\abs\xi^2 +   b_0^3 k^2 (\nabla^2 k) b_0\abs\xi^2 (D\abs\xi^2)^2\\
&- iD(b_0 p_1 b_0) (\nabla p_2) b_0 -  b_0 p_0 b_0 - b_1 p_1 b_0+ iD b_0 (\nabla p_1) b_0.      
        \end{split}
       \label{eq:mod-cur-prop-b2term}
    \end{align}
   
\end{prop}

\begin{proof}
    We first compute the last two terms (terms with $1/2$ in front) in equation \eqref{eq:app-mod-cur-b2term-v2}.
    From \eqref{eq:mod-cur-D^2b_0} and \eqref{eq:mod-cur-hord-p2-2}:
\begin{align}
    \begin{split}
         &\,\, D^2 b_0 \times_\Theta \nabla^2 p_2 \\
     = &\,\, (- b_2^2  k D^2 \abs\xi^2 + 2 b_0^3 k^2 (D \abs\xi^2)^2)
    ((\nabla^2 k) \abs\xi^2 + k (\nabla^2 \abs\xi^2))\\
    = & \,\,-b_0^2k (D^2 \abs\xi^2 \cdot \nabla^2 k)\abs\xi^2 - b_0^2 k^2 D^2\abs\xi^2 \nabla^2 \abs\xi^2 \\
    &\,\; + 2 b_0^3 k^2 ( (D\abs\xi^2)^2 \nabla^2k)\abs\xi^2 + 2b_0^3 k^3 (D\abs\xi^2)^2 \nabla^2 \abs\xi^2. 
  \end{split}
    \label{eq:mod-cur-D^b0-symd^2p2}
   \end{align}
Combine \eqref{eq:mod-cur-Db_0} and \eqref{eq:mod-cur-hord-p2-2}:
\begin{align}
    \begin{split}
         ( \nabla^3 \ell) Db_0 D^2p_2
 & =  ( \nabla^3\ell) (-b_0^2 k) (D\abs\xi^2)k (D^2 \abs\xi^2) \\
 & = (-b_0^2 k^2) ( \nabla^3\ell)(D\abs\xi^2) (D^2 \abs\xi^2).
    \end{split}
    \label{eq:mod-cur-l-Db0-D^2p2}
\end{align}
Apply $D$ onto \eqref{eq:mod-cur-b_1-vesion2}, 
\begin{align*}
    D b_1 & = -i\left [
    D(b_0^2k) (D\xi^2 \cdot \nabla k) b_0 \xi^2 + b_0^2 k D(D\abs\xi^2 \cdot \nabla k) b_0 \abs\xi^2 \right.\\
& \left. 
    + b_0^2 k (D\abs\xi^2 \cdot \nabla k) (Db_0) \abs\xi^2 
+ b_0^2 k (D\abs\xi^2 \cdot \nabla k) b_0 D\abs\xi^2 \right] - D(b_0 p_1 b_0)\\
& = -i\left[
-2 b_0^3 k^2 D\abs\xi^2 (D\abs\xi^2 \cdot \nabla k) b_0  \abs\xi^2 + b_0^2 k (D\abs\xi^2 \cdot \nabla k)b_0 D \abs\xi
\right.\\& \left.
+ b_0^2 k (D^2\abs\xi^2 \cdot \nabla k) b_0\abs\xi^2 
 - b_0^2 k (D\abs\xi^2 \cdot \nabla k) b_0^2 k (D \abs\xi^2)\abs\xi^2
\right] \\
&- D(b_0 p_1 b_0),
\end{align*}
thus
\begin{align}
    \begin{split}
         &\; D b_1  \nabla p_2\\
    =&\; -i \left[-2 b_0^3 k^2 (D\abs\xi^2 \cdot \nabla k) b_0 (D\abs\xi^2 \cdot \nabla k) \abs\xi^4
        \right. \\ 
        & \left. \; +
        b_0^2 k (D\abs\xi^2 \cdot \nabla k) b_0 (D\abs\xi^2 \cdot \nabla k)\abs\xi^2
    \right. \\ 
        & \left. \; +
b_0^2 k (D^2\abs\xi^2 \cdot \nabla k) b_0 (\nabla k) \abs\xi^4
    \right. \\ 
        & \left. \;     
    - b_0^2 k (D\abs\xi^2 \cdot \nabla k) b_0^2 k (D\abs\xi^2 \cdot \nabla k)\abs\xi^4 
    \right]\\
    &- D(b_0 p_1 b_0) \nabla p_2. 
    \end{split}
    \label{eq:mod-cur-Db1-symdp2}
\end{align}

We substitute \eqref{eq:mod-cur-D^b0-symd^2p2}, \eqref{eq:mod-cur-l-Db0-D^2p2} and  \eqref{eq:mod-cur-Db1-symdp2} into \eqref{eq:app-mod-cur-b2term-v2}:
\begin{align*}
     b_2 & = 2 b_0^3 k^2 (D\abs\xi^2 \cdot \nabla k) b_0 (D\abs\xi^2 \cdot \nabla k) b_0 \abs\xi^4\\
    & -  b_0^2 k (D\abs\xi^2 \cdot \nabla k) b_0 (D\abs\xi^2 \cdot \nabla k) b_0 \abs\xi^2 \\
    & -   b_0^2 k (D^2\abs\xi^2 \cdot \nabla k) b_0 (\nabla k)b_0 \abs\xi^4\\
& + b_0^2 k (D\abs\xi^2 \cdot \nabla k) b_0^2 k (D\abs\xi^2 \cdot \nabla k) b_0 \abs\xi^4 \\
&+ \frac12 (-b_0^2 k^2) (\nabla^3 \ell)(D\abs\xi^2) (D^2 \abs\xi^2) b_0 \\
&-\frac12 b_0^2k (D^2 \abs\xi^2 \cdot \nabla^2 k) b_0 \abs\xi^2\\
&-\frac12 b_0^2 k^2 D^2\abs\xi^2 \nabla^2 \abs\xi^2 b_0\\
&+  b_0^3 k^2 ( (D\abs\xi^2)^2\nabla^2k) b_0 \abs\xi^2\\
&- iD(b_0 p_1 b_0) (\nabla p_2) b_0 -  b_0 p_0 b_0 - b_1 p_1 b_0 + iD b_0 (\nabla p_1) b_0.
\end{align*}
To  obtain \eqref{eq:mod-cur-prop-b2term}, we just need to move  the vertical and the horizontal derivatives of $\abs\xi^2$ to the right end for each summand above. This operation is valid because of the $\T^n$-invariant property of the function $\abs\xi^2$.

\end{proof}

\section{Integration over the cosphere bundle $S^*M$} 
\label{App:int_over_cosphere}

Using symmetries, one can quickly compute the surface integral of the function $x^2$ over the unit sphere $x^2+y^2+z^2 =1$ in $\R^3$:
\begin{align*}
    \int_{x^2+y^2+z^2 =1} x^2 dS = \frac13 \int_{x^2+y^2+z^2 =1}(x^2+y^2+z^2)   dS = \frac13 \mathrm{Vol}(S^2).
\end{align*}
In general, we have
\begin{lem} 
    \label{lem:mod-cur-int-sphere-symmetries}
    Let $\xi = (\xi_1,\dots,\xi_m)$ be a coordinate system of $\R^m$ and $S^{m-1}$ is the unit sphere with  induced measure $d\sigma_{S^{m-1}}$. Then
%     Let $(x,\xi)$ be a  local coordinate system around $x$ such that the metric tensor $g^{ij} = \delta_{ij}$
%     at $x$. Thus the unit sphere coincides with the Euclidean  unit sphere in $\R^m$, we denote the induced measure by $d\sigma_{S^{m-1}}$ and $\xi_s$, $\xi_a$ are components of $\xi$, then
    \begin{align}
    \int_{S^{m-1}} \xi_s\xi_a d\sigma_{S^{m-1}} =
    \begin{cases}
        0 & \text{when $s\neq a$}\\ \frac{\abs\xi^2}{m} \op{Vol}(S^{m-1}),
        & \text{when $s = a$.}
    \end{cases}
    \label{eq:mod-cur-int-sphere-asym}
\end{align}
\end{lem}

In an orthonormal local coordinate system: $g^{ij} = \xi_{ij}$ and $\abs\xi^2 = \sum_j \xi_j^2$, hence
\begin{align*}
    ((D\abs\xi^2)^2)_{ij} = ((D\abs\xi^2) \otimes (D\abs\xi^2))_{ij} = 4\xi_i \xi_j, \,\,\,
    (D^2\abs\xi^2)_{ij} = 2 \delta_{ij}
\end{align*}

Apply the lemma above, one can quickly conclude:
\begin{lem}
   Let $S^{m-1} \subset T_x^*M \cong \R^m$ be the unit sphere with respect to the Riemannian metric $g^{-1}$ and $d\sigma_{S^{m-1}}$ be the induced measure as before. Then
    \begin{align}
         \int_{S^*_xM} (D\abs\xi^2)^2 d\sigma_{S^{m-1}}   
     & = \frac{4\abs\xi^2}{m}\op{Vol}(S^{m-1})(g^{-1} \Big|_x),
     \label{eq:mod-cur-int-sphere-(Dxi)^2}\\
      \int_{S^*_xM} D^2\abs\xi^2 d\sigma_{S^{m-1}}  & = 2 \op{Vol}(S^{m-1})(g^{-1} \Big|_x).
     \label{eq:mod-cur-int-sphere-Dxi^2}
    \end{align} 
\end{lem}

%-----------------------------------------
%Let us start with the covariant derivatives of the phase function $\ell$.
Now we have to deal with the horizontal derivatives $\nabla^j \abs\xi^2$ with $j=1,2$ which leads to the appearance of the scalar curvature function $\mathcal S_\Delta$. Start with $\nabla^3 \ell$, the following lemma was proved in \cite{MR538027}, which is recalled here for completeness:
\begin{lem}
    \label{lem:wpse-diff-jetsof-phafun}
In local coordinates at a given point $x \in M$, 
\begin{align*}
   (  \nabla^3 \ell)_{ijk} = -\frac13 \sum_p \xi_p \brac{
        R\indices{^p_{ikj}}+R\indices{^p_{jki}}    }. 
\end{align*}

%  Let us denote:
% \begin{align}
%     \rho_{\ell}^{j,l} = \nablagc^l \symd^j \ell,\,\,\, j,l\ge 1,
%     \label{eq:mod-cur-tensor-rho}
% \end{align}
% and we assume that the connection $\nabla$ is torsion free, then 
% \begin{align}
%     \rho_{\ell}^{2,1} &=\rho_{\ell}^{\abrac{ij},\abrac{l}} = (\nablagc \symd^2 \ell)_{ijl} = -\frac{1}{3}
%     \xi_p\brac{R\indices{^p_{ilj}} + R\indices{^p_{jli}}}
%     \label{eq:mod-cur-rho-2-1},\\
%     \rho_{\ell}^{1,2}& =\rho_{\ell}^{\abrac{l},\abrac{ij}}
%      = (\nablagc^2 \symd \ell)_{lij} = -\frac16
% \xi_p\brac{R\indices{^p_{jil}} + R\indices{^p_{ijl}}}.
%     \label{eq:mod-cur-rho-1-2}
% \end{align}
% 
\end{lem}
\begin{proof}
    We adopt the following convention for the curvature tensor on the tangent bundle $TM$:
\begin{align*}
    (R(X,Y)Z)^l = R\indices{^l_{ijk}} X^j Y^k Z^l,
\end{align*}
where $X$, $Y$ and $Z$ are vector fields. 
In local coordinates:
\begin{align*}
    (\nabla^3 \ell)_{ikj} - (\nabla^3 \ell)_{ijk}& = \brac{(\nabla_j\nabla_k - \nabla_k\nabla_j)(\nabla \ell)}_i
   \\ & = -(\nabla \ell)_p R\indices{^p_{ijk}} = \xi_p R\indices{^p_{ikj}}.
\end{align*}
The   negative sign appears because the covariant derivatives are taken on the cotangent bundle. 
% Therefore:
% \begin{align*}
%    (\nabla^3 \ell)_{ikj}  = (\nabla^3 \ell)_{ijk} + \xi_p R\indices{^p_{ikj}}.
% \end{align*}
Since the connection is torsion free, $(\nabla^3 \ell)$ is symmetric in the first two indices, thus 
\begin{align*}
   (\nabla^3 \ell)_{jki} = (\nabla^3 \ell)_{jik} +  \xi_p R\indices{^p_{jki}}
   = (\nabla^3 \ell)_{ijk} + \xi_p R\indices{^p_{jki}}.
\end{align*}
Again, due to the symmetry in the first two indices,  all the components of  the tensor $\nabla^3 \ell$ fall into $\set{(\nabla^3 \ell)_{ijk}, (\nabla^3 \ell)_{ikj}, (\nabla^3 \ell)_{jki}}$. Since $\symd^3 \ell = 0$, summing up the three terms gives us
\begin{align*}
    0 &= 3 (\nabla^3 \ell)_{ijk}  +\xi_p R\indices{^p_{ikj}} + \xi_p R\indices{^p_{jki}},\\
    (\nabla^3 \ell)_{ijk}&=
    - \frac13 \brac{
    \xi_p R\indices{^p_{ikj}} + \xi_p R\indices{^p_{jki}}
    }. 
\end{align*}
% The equations \eqref{eq:mod-cur-rho-2-1} and \eqref{eq:mod-cur-rho-1-2} follow quickly from partial  symmetrization of $(\nabla^3 \ell)_{ijk}$. 
\end{proof}

\begin{lem}
    \label{lem:mod-cur-xisquared-derivatives}
   Evaluating at at given point $x \in M$,  the horizontal and the vertical derivatives of $\abs\xi^2$ are given by:
    \begin{align}
        \nabla \abs\xi^2 & = 0
        \label{eq:mod-cur-1sthor-xisquared}
        \\
        (\nabla^2\abs\xi^2) & = 
        (\nabla^3 \ell)(\nabla \ell) g^{-1} + (\nabla \ell) (\nabla^3 \ell) g^{-1}
                \label{eq:mod-cur-2ndhor-xisquared}
    \end{align}
where the contraction is implemented in the following way:
\begin{align*}
    (\nabla^2\abs\xi^2)_{jk}  = (\nabla^3 \ell)_{kij} (\nabla \ell)_a g^{ka} + 
    (\nabla \ell)_a (\nabla^3 \ell)_{kij}g^{ak}.
\end{align*}
In particular, when $g^{ij} = \delta_{ij}$, 
\begin{align}
    \label{eq:mod-cur-2ndhor-xisquared-complete}
    (\nabla^2\abs\xi^2)_{jk} = \frac23 \sum_{p,i} \xi_p \xi_i R_{pjik}.
\end{align}
%     while
%     \begin{align}
%         (D^2 \abs\xi^2)^{ij} = 2 g^{ij}\,\,\,\, (D\abs\xi^2)^j = 2\xi_i g^{ij},
%         \label{eq:mod-cur-xisquared-verder}
%     \end{align}
%     where $g^{ij}$ are the tensor components of $g^{-1}$, which is the metric tensor on $T^*M$.
\end{lem}
\begin{proof}
    According to the definition of the horizontal differential $\nabla$, we have to compute:
    \begin{align*}
        \nabla^k \abs{ d_y \ell(\xi_x,y)}^2 \Big|_{y=x}
        = \nabla^k \abrac{ d_y \ell(\xi_x,y) , d_y \ell(\xi_x,y)}_{g^{-1}}
        , \,\,\, k=1,2.
    \end{align*}
    Set $d \ell = d_y \ell(\xi_x,y)$, we rewrite the metric pairing $\abrac{d \ell, d\ell}_{g^{-1}}$ as $d\ell\otimes d\ell \cdot g^{-1}$.   Using the Leibniz's rule and the fact that the connection is Levi-Civita: $\nabla g^{-1} =0$:
    \begin{align*}
    \nabla \abs\xi^2 & = \nabla \brac{d\ell\otimes d\ell \cdot g^{-1}} = (2 \nabla d\ell )\otimes d\ell\cdot g^{-1} \\    
    & =\nabla^2 \ell \otimes \ell \cdot g^{-1}.
\end{align*}
Notice that $\nabla$ is torsion-free, when evaluating at $y=x$
    \begin{align*}
        ( \nabla^2 \ell) |_{y=x} = (\symd^2 \ell) |_{y=x}=0, 
    \end{align*}
    here $\symd^2$ is the symmetrization of $\nabla^2$. 
    Thus we have proved \eqref{eq:mod-cur-1sthor-xisquared}. However, the second covariant derivative is nonzero. Again, since $\nabla g^{-1} =0$:
    \begin{align*}
    \nabla^2( (d\ell \otimes d\ell)\cdot g^{-1}) 
   & = (\nabla^2 d\ell) \otimes d\ell \cdot g^{-1} + d\ell \otimes (\nabla^2 d\ell) \cdot g^{-1} \\
&+ (\nabla d\ell) \otimes (\nabla d\ell) \cdot g^{-1}. 
\end{align*}
At $y=x$: $(\nabla d\ell)= \nabla^2 \ell = 0$ as before, thus we have reached \eqref{eq:mod-cur-2ndhor-xisquared}. \par
At last, we need to compute the contraction \eqref{eq:mod-cur-2ndhor-xisquared-complete}. When $g^{ij} =\delta_{ij}$,  apply lemma \ref{lem:wpse-diff-jetsof-phafun}:
% \begin{align*}
%     (  \nabla^3 \ell)_{ijk} = -\frac13 \xi_p \brac{
%           R_{pikj}+R_{pjki}  } 
% \end{align*}
%  The contraction is implemented as follows:
\begin{align*}
    (\nabla^3 \ell)_{ijk}(d\ell)_a g^{aj} = -\frac13 \sum_{i,p} \xi_p \xi_i \brac{
        R_{pikj}+R_{pjki}}.
\end{align*}
Since the curvature tensor $R_{pikj}$ is anti-symmetric in $(p,i)$ and in $(k,j)$ respectively, the first term summed to zero and  for the second term, we use the minus sign to switch $k$ and $j$, thus the result becomes:
\begin{align*}
    (\nabla^3 \ell)_{ijk}(d\ell)_a g^{aj} = \frac13 \sum_{i,p} \xi_p \xi_i R_{pjik}.
\end{align*}
The second term in \eqref{eq:mod-cur-2ndhor-xisquared} provides the same answer, therefore we have proved \eqref{eq:mod-cur-2ndhor-xisquared-complete}. 

% After substituting the explicit expression of $\nablagc^2 \symd \ell$  in lemma  \ref{lem:wpse-diff-jetsof-phafun},  we have proved  \eqref{eq:mod-cur-1sthor-xisquared} and \eqref{eq:mod-cur-2ndhor-xisquared}. \par
% The vertical derivative $D$ is the partial derivatives in the $\xi$ variable, thus the calculation of \eqref{eq:mod-cur-xisquared-verder} is straightforward. 
\end{proof}

%\int_{S^*_xM} (D\abs\xi^2)^2 d\sigma_{S^{m-1}}

\begin{cor}
    In a orthonormal local coordinate system, we compute the following contractions:
    \begin{align}
        \label{eq:mod-cur-loc-symdsymd^2l}
        (D\abs\xi^2) (D^2\abs\xi^2) (\nabla^3\ell) &= -\frac83\sum_{k,p,i} \xi_p\xi_kR_{piki}\\
    (D^2 \abs\xi^2)(\nabla^2 \abs\xi^2)  &=\frac43\sum_{k,p,i} \xi_p\xi_kR_{piki}  \label{eq:mod-cur-loc-nabla^2xi}
    \\
     (D\abs\xi^2)^2 (\nabla^2 \abs\xi^2)  &= 0
    \end{align}
\end{cor}
\begin{proof}
    Since  $(D^2 \abs\xi^2)_{ij} = 2g^{ij} = 2\delta_{ij}$, take  \eqref{eq:mod-cur-2ndhor-xisquared-complete} into account:
    \begin{align*}
        (D^2\abs\xi^2) (\nabla^2\abs\xi^2) =\frac43 \sum_{p,i,k} \xi_p \xi_i R_{pkik},
    \end{align*}
    which is \eqref{eq:mod-cur-loc-nabla^2xi}. For \eqref{eq:mod-cur-loc-symdsymd^2l}, we need $(D\abs\xi^2)_j = 2\xi_j$ and lemma \ref{lem:wpse-diff-jetsof-phafun}:
    \begin{align*}
        (D\abs\xi^2)_k (D^2\abs\xi^2)_{ij} (\nabla^3\ell)_{ijk} &= -\frac13 (2\xi_k) (2\delta_{ij}) \sum_p\xi_p
        \brac{R_{pikj} + R_{pjki}}\\
       & =-\frac83\sum_{p} \xi_p\xi_kR_{piki}.
    \end{align*}
    At last, 
    \begin{align*}
        ((D\abs\xi^2)^2)_{ij} (\nabla^2 \abs\xi^2)_{ij} = 
        \sum_{p,l}( 4 \xi_i\xi_j) \frac23  (\xi_{p} \xi_l R_{pilj}).
    \end{align*}
    Due to the anti-symmetries of the curvature tensor,  the right hand side  vanishes when summing over all indices $i,l,p,j$.
\end{proof}

\begin{cor}
    Keep the notations as above, 
%     \begin{align}
%          \int_{S^*M} D^2\abs\xi^2 \cdot \nabla^2 k  d\sigma_{S^{m-1}}  &=  2 \op{Vol}(S^{m-1}) \Delta k
%          \label{eq:mod-cur-int-sphere-symd^k-(Dxi)^2}
%          \\
%      \int_{S^*M} (D\abs\xi^2)^2  \cdot \nabla^2 k d\sigma_{S^{m-1}}    &=
%      \frac{4\abs\xi^2}{m}\op{Vol}(S^{m-1})\Delta k
%  \label{eq:mod-cur-int-sphere-symd^k-Dxi^2}
% \end{align}
\begin{align}
    \label{eq:mod-cur-int-sphere-Dxi^2-nabla^2}
    \int_{S^*M} (D\abs\xi^2) (D^2\abs\xi^2) (\nabla^3\ell) d\sigma_{S^{m-1}}&=
    -\frac{8}{3m}\mathrm{Vol}(S^{m-1}) \mathcal S_\Delta
    \\
    \int_{S^*M} (D^2 \abs\xi^2)(\nabla^2 \abs\xi^2)  d\sigma_{S^{m-1}
}&=\frac{4}{3m}\mathrm{Vol}(S^{m-1}) \mathcal S_\Delta
\label{eq:mod-cur-int-sphere-D^2xi-nabla^2} 
\end{align}
\end{cor}
\begin{proof}
    According to \eqref{eq:mod-cur-loc-symdsymd^2l} and \eqref{eq:mod-cur-loc-nabla^2xi}, it suffices to show that 
    \begin{align*}
       \int_{S^*_xM}\brac{ \sum_{p,i,k} \xi_p \xi_i R_{pkik}}  d\sigma_{S^{m-1}} 
       =\frac1m \mathrm{Vol}(S^{m-1}) \mathcal S_\Delta, 
    \end{align*}
    here $\mathcal S_\Delta = \sum_{p,k} R_{pkpk}$ is the trace of the Ricci tensor. Indeed,
  we apply lemma \ref{lem:mod-cur-int-sphere-symmetries} again, 
\begin{align*}
    \int_{S^*_xM}\brac{ \sum_{p,i,k} \xi_p \xi_i R_{pkik}}  d\sigma_{S^{m-1}}& 
    = \int_{S^*_xM} \brac{\sum_{p,k} \xi_p^2 R_{pkpk}}d\sigma_{S^{m-1}}
    =  \frac1m \mathrm{Vol}(S^{m-1}) \sum_{p,k}  R_{pkpk} 
   \\
   &= \frac1m \mathrm{Vol}(S^{m-1}) \mathcal S_\Delta.
\end{align*}
\end{proof}

Apply the following substitution rules (eq. \eqref{eq:mod-cur-int-sphere-Dxi^2-nabla^2},  \eqref{eq:mod-cur-int-sphere-D^2xi-nabla^2},  \eqref{eq:mod-cur-int-sphere-(Dxi)^2} and \eqref{eq:mod-cur-int-sphere-Dxi^2})
\begin{align*}
   (D\abs\xi^2) (D^2\abs\xi^2) (\nabla^3\ell) \mapsto -\frac{8}{3m}\mathrm{Vol}(S^{m-1}) \mathcal S_\Delta, & \,\,\, (D^2 \abs\xi^2)(\nabla^2 \abs\xi^2) \mapsto \frac{4}{3m}\mathrm{Vol}(S^{m-1}) \mathcal S_\Delta \\
   (D\abs\xi^2)^2 \mapsto\frac{4\abs\xi^2}{m}\op{Vol}(S^{m-1})(g^{-1} \Big|_x),& \,\,\,
   D^2\abs\xi^2 \mapsto 2 \op{Vol}(S^{m-1})(g^{-1} \Big|_x),
\end{align*}
to the  $b_2$ term in proposition \ref{prop:mod-cur-prop-b2term}. The result is summerized below. 
\begin{prop}
    \label{prop:app-int_b2-cosphsere}
       Keep the notations as above. Assume that the lower order symbols of the Laplacian is zero, that is $p_1 = p_0 =0$. Along the fiber $T^*_xM$ for some $x \in M$, the integral over the unit sphere $\int_{S^{m-1}} b_2  d\sigma_{S^{m-1}}$ is equal to, up to an overall factor $\op{Vol}(S^{m-1})$:
%        that is $p_1 = p_0 =0$, substitute the results in corollary $\ref{cor:mod-cur-int-sphere-xi^2}$ and  $\ref{cor:mod-cur-int-sphere-scalar-curvature}$ into $b_2$  $($in equation $\eqref{eq:mod-cur-prop-b2term}$$)$, we obtain that    
       \begin{align}
           \begin{split}
              & \,\,\,\,\frac 4m 2 b_0^3 k^2 (\nabla k) b_0 (\nabla k) b_0 \abs\xi^6 g^{-1}
     -(2+\frac4m) b_0^2 k  (\nabla k) b_0 (\nabla k) b_0 \abs\xi^4 g^{-1}\\
     &+ \frac4m b_0^2k   (\nabla k) b_0^2 k (\nabla k) b_0 \abs\xi^6 g^{-1} 
     - b_0^2k (\nabla^2 k) b_0 \abs\xi^2 g^{-1}
     +  \frac4m b_0^3 k^2 ( \nabla^2 k) b_0\abs\xi^4 g^{-1}\\
     & + \frac1m \frac23 b_0^2 k^2  \mathcal S_{\Delta} b_0 \abs\xi^2,
           \end{split}
           \label{eq:app-mod-cur-b2term-int-sphere}
       \end{align}
       where $m$ is the dimension of the manifold. 
    \end{prop}
    \begin{rem}
        The contraction $(\nabla^2 k) g^{-1} = -\Delta k$ is nothing but the Laplacian of the conformal factor $k$.
    \end{rem}